\documentclass[11pt]{amsart} 
\usepackage{geometry}             

\usepackage{float}
\usepackage{bm}
\usepackage{fullpage,xcolor}
\usepackage[mathscr]{euscript}
\usepackage[all]{xy}
\usepackage{epsfig}
\usepackage[T1]{fontenc}
\usepackage{amsfonts}

\usepackage{graphicx}
\usepackage{amssymb}
\usepackage{amsmath}
\usepackage{amsthm}
\usepackage{mathrsfs}
\usepackage{epstopdf}
\usepackage{url}
\usepackage{hyperref}
\usepackage[msc-links,alphabetic]{amsrefs}
\usepackage{upgreek}

\usepackage{color}
\usepackage{subfig} 

\newcommand{\al}{\alpha}
\newcommand{\be}{\beta}

\newcommand{\de}{\delta}
\newcommand{\na}{\nabla}

\renewcommand{\th}{\theta}

\newcommand{\si}{\sigma}
\newcommand{\om}{\omega}
\newcommand{\Ga}{\Gamma}

\newcommand{\Om}{\Omega}
\newcommand{\De}{\Delta}
\newcommand{\Si}{\Sigma}
\newcommand{\ZZ}{{\mathbb Z}}

\newcommand{\QQ}{{\mathbb Q}}
\newcommand{\RR}{{\mathbb R}}

\newcommand{\sD}{\mathscr D}

\newcommand{\cC}{\mathscr C}

\newcommand{\lto}{\longrightarrow}
\newcommand{\lk}{\operatorname{\ell{\it k}}}
\newcommand{\vlk}{\operatorname{{\it v}\ell{\it k}}}
\newcommand{\ind}{\operatorname{ind}}

\newcommand{\nullity}{\operatorname{nullity}}
\newcommand{\sig}{\operatorname{sig}}

\newcommand{\Hom}{\operatorname{Hom}}
\newcommand{\sm}{\smallsetminus}
\newcommand{\co}{\colon}
\newcommand{\im}{{\text{im}}}
\newcommand{\id}{{\text{id}}}
\newcommand{\Int}{\operatorname{Int}}

\newcommand*\wbar[1]{
  \hbox{ \kern-0.2em%
    \vbox{%
      \hrule height 0.5pt  
      \kern0.25ex
      \hbox{%
        \kern-0.15em
        \ensuremath{#1}%
        \kern-0.05em
      }%
    }%
  \kern0.05em}%
}

\newtheorem{theorem}{Theorem}[section]
\newtheorem{lemma}[theorem]{Lemma}
\newtheorem{proposition}[theorem]{Proposition}

\newtheorem{corollary}[theorem]{Corollary}
\newtheorem*{question}{Question}
\newtheorem{conjecture}[theorem]{Conjecture}

\theoremstyle{definition}     
\newtheorem{definition}[theorem]{Definition}

\theoremstyle{remark}
\newtheorem{remark}[theorem]{Remark}
\newtheorem{example}[theorem]{Example}

\title[Signature and concordance of virtual knots]{Signature and concordance of virtual knots}
\author[H. U. Boden]{Hans U. Boden}
\address{Mathematics \& Statistics, McMaster University, Hamilton, Ontario}
\email{boden@mcmaster.ca}
\urladdr{math.mcmaster.ca/~boden}

\author[M. Chrisman]{Micah Chrisman}
\address{Mathematics, Ohio State University at Marion, Marion, Ohio}
\email{chrisman.76@osu.edu}
\urladdr{https://micah46.wixsite.com/micahknots}

\author[R. Gaudreau]{Robin Gaudreau}
\address{Mathematics, University of Toronto, Toronto, Ontario}
\email{robin.gaudreau@mail.utoronto.ca}


\date{\today}                 

\begin{document}

\begin{abstract}
We introduce Tristram-Levine signatures of virtual knots and use them to investigate virtual knot concordance. The signatures are defined first for almost classical knots, which are virtual knots admitting homologically trivial representations. The signatures and $\om$-signatures are shown to give bounds on the \emph{topological} slice genus of almost classical knots, and they are applied to address a recent question of Dye, Kaestner, and Kauffman on the virtual slice genus of classical knots. A conjecture on the topological slice genus is formulated and confirmed for all classical knots with up to 11 crossings and for 2150 out of 2175 of the 12 crossing knots.

The Seifert pairing is used to define directed Alexander polynomials, which we show satisfy a Fox-Milnor criterion when the almost classical knot is slice.
We introduce virtual disk-band surfaces and use them to establish realization theorems for Seifert matrices of almost classical knots. As a consequence, we deduce that any integral polynomial $\De(t)$ satisfying $\De(1)=1$ occurs as the Alexander polynomial of an almost classical knot.
 
In the last section, we use parity projection and Turaev's coverings of knots to extend the Tristram-Levine signatures to all virtual knots. A key step is a theorem saying that parity projection preserves concordance of virtual knots. This theorem implies that the signatures, $\om$-signatures, and Fox-Milnor criterion can be lifted to give slice obstructions for all virtual knots. There are 76 almost classical knots with up to six crossings, and we use our invariants to determine the slice status for all of them and the slice genus for all but four.  Table \ref{table-2} at the end summarizes our findings.

\end{abstract}

\maketitle
\section*{Introduction}
A knot in $S^3$ is said to be \emph{slice} if it bounds an embedded disk in $D^4$, and two oriented knots $K,J$ in $S^3$ are said to be \emph{concordant} if their connected sum $K \# -J$ is slice. The set of concordance classes of knots forms an abelian group $\cC$ with addition given by connected sum. Since its introduction by Fox and Milnor in \cite{Fox-1962, Fox-Milnor-1966}, the subject of knot concordance has been of considerable interest to geometric topologists. In the 1980s, breakthroughs in 4-dimensional topology revealed a vast chasm separating smooth and topological concordance; indeed the groundbreaking results of Freedman and Donaldson combine to produce striking examples, including knots that are topologically slice but not smoothly slice \cite{Gompf-1986}. The $(-3,5,7)$ pretzel knot is perhaps the first and most famous example, but since then many other examples have been discovered.

Over the past 50 years, our understanding of topological concordance has progressed with  
the introduction of increasingly sophisticated invariants,
such as the Arf invariant \cite{Robertello-1965}, 
the Trotter-Murasugi and Tristram-Levine signatures \cite{Trotter-1962, Murasugi-1965, Tristram-1969}, 
the Casson-Gordon invariants \cite{Casson-Gordon-1978, Casson-Gordon-1986}, 
and the Cochran-Orr-Teichner invariants \cite{Cochran-Orr-Teichner-2003}.  
In a similar fashion, knowledge about
smooth concordance has advanced by means of the many invariants arising from gauge theory and knot homology, primarily the $d$-, $\tau$- and $\Upsilon$-invariants coming from Heegaard-Floer theory \cite{Ozsvath-Szabo-2003, Ozsvath-Szabo-Stipsicz-2017} and the Rasmussen $s$-invariant from Khovanov homology \cite{Rasmussen-2010}. Further results have been obtained from the many variants that can be defined using instanton, monopole, and Pin-monopole Floer homology and Khovanov-Rozansky $sl_n$ knot homology, see  \cite{Hom-2017, Lobb-2009}. 
 
When combined with constructive techniques, these invariants can be successfully deployed to determine the slice genus $g_4(K)$ and the unknotting number $u(K)$ of many low-crossing knots. Recall that the \emph{slice genus} of $K$ is the minimum genus among all oriented surfaces $F$ embedded in  $D^4$ with boundary $\partial F = K$, and the
 \emph{unknotting number} $u(K)$ is the minimum number of crossing changes needed to unknot $K$. These invariants are related by the inequality $g_4(K) \leq u(K).$ There are in fact two slice genera; the smooth slice genus $g_4(K)$ is a minimum over all smoothly embedded surfaces in $D^4$, and the topological slice genus $g^{\rm top}_4(K)$
is the minimum over all locally flat surfaces in $D^4$.
For knots up to nine crossings, $g_4(K) = g^{\rm top}_4(K)$,  but already for knots with ten crossings, the difference begins to emerge (see \cite{Lewark-McCoy-2017} and KnotInfo \cite{Knotinfo}).

Virtual knots were introduced by Kauffman in \cite{Kauffman-1999}, and concordance and cobordism of virtual knots were studied in \cites{Carter-Kamada-Saito, Turaev-2008-a, Kauffman-2015}.
While many of the standard invariants of classical knots extend to virtual knots in a straightforward way, very few concordance invariants have been extended. Apart from the Rasmussen invariant, which was extended to virtual knots by Dye, Kaestner, and Kauffman \cite{Dye-Kaestner-Kauffman-2014} using Manturov's generalization of Khovanov homology for virtual knots \cite{Manturov-2007}, none of the other concordance invariants for classical knots have been extended to the virtual setting. 

Our work is motivated in part by the following interesting question: 
\begin{question}\cite{Dye-Kaestner-Kauffman-2014}
Can the extension from the category of classical knots to virtual knots lower the slice genus?
\end{question}
One can ask this question about the smooth slice genus or the topological slice genus, so there are really two problems here.
For instance, Dye, Kaestner and Kauffman show that this cannot happen for any classical knot $K$ with smooth slice genus equal to its Rasmussen invariant, and that includes all positive knots \cite[Theorem 6.8]{Dye-Kaestner-Kauffman-2014}.
Similarly, the main result of \cite{Boden-Nagel-2016} implies that
this cannot happen for any classical knot with slice genus one.
Since that result holds for both smoothly slice knots and topologically slice knots, it follows that any classical knot like the $(-3,5,7)$ pretzel knot which is topologically slice but not smoothly slice remains so after passing to the virtual category.

In \cite{Boden-Chrisman-Gaudreau-2017, Boden-Chrisman-Gaudreau-2017t}, we develop methods for slicing virtual knots and use virtual unknotting operations to determine the slice genus.  For instance, we prove that Turaev's graded genus and the writhe polynomial are concordance invariants, and we use them to obtain useful slice obstructions for virtual knots. Unfortunately, these invariants are trivial on classical knots, and thus they are not helpful in addressing the above question.
 
Our goal in this paper is to extend signatures to virtual knots and to apply them to the above question about the topological slice genus of classical knots. For instance, our results resolve the above question for any classical knot whose signature satisfies $|\si(K)| = 2 g_4^{\rm top}(K)$. Taken with results in 
 \cite{Dye-Kaestner-Kauffman-2014}  and \cite{Boden-Nagel-2016}, they show that the topological and smooth slice genera do not change in passing to the virtual category for all classical knots with 11 or fewer crossings with two possible exceptions: $11a_{211}$ and $11n_{80}$ \cite{Knotinfo}.

We recall briefly the history and construction of knot signatures and outline the difficulties in extending them to the virtual setting.
Trotter defined the signature as an invariant of knots in \cite{Trotter-1962}, and Murasugi showed it to be invariant under concordance in \cite{Murasugi-1965}. The more general Tristram-Levine signatures were introduced in \cite{Tristram-1969} and used  to define a surjective homomorphism $\phi \colon \cC \to \ZZ^\infty \oplus (\ZZ/2)^\infty \oplus (\ZZ/4)^\infty$ in \cite{Levine-1969}.

The knot signatures are defined in terms of the linking pairing associated to a choice of Seifert surface $F$ for the knot. The Seifert matrix $V$ depends on the choice of surface as well as basis for $H_1(F)$, but its $S$-equivalence class is an invariant of the underlying knot \cite[\S 5.3]{Kawauchi-1990}. Since the signature $\sig(V + V^\uptau)$ is invariant under $S$-equivalence, it gives a well-defined invariant of the knot. 

For virtual knots, there are several obstacles to defining signatures using this approach. Firstly, Seifert surfaces do not always exist, and even when they do, the $S$-equivalence class of the associated Seifert matrices  depends on the choice of Seifert surface and does not determine a well-defined invariant of the underlying knot.

Before delving into details, we take a moment to review related results in  \cite{Im-Lee-Lee-2010} and \cite{Cimasoni-Turaev}. In \cite{Im-Lee-Lee-2010}, Im, Lee, and Lee introduce signature-type invariants, denoted $\si_\xi(K)$, for checkerboard colorable virtual knots in terms of Goeritz matrices.
These invariants depend on a choice of checkerboard coloring $\xi$, and it is not generally known whether they are invariant under virtual knot concordance.
In \cite{Cimasoni-Turaev}, Cimasoni and Turaev extend many invariants of classical knots, including signatures, to knots in quasi-cylinders. Their results do not immediately give concordance invariants for virtual knots for several reasons. One is that their notion of concordance is more restrictive and does not take into account stable equivalence. Another more serious issue is their assumption that $H_2(M) = 0$ for the quasi-cylinder $M$. This assumption is key to showing that invariants of $S$-equivalence classes of Seifert triples give well-defined invariants of knots in quasi-cylinders. However, this condition is not satisfied for a knot $K$ in a thickened surface $\Si \times I$.
Indeed, in \cite[\S 8.4, p.~558]{Cimasoni-Turaev} they write ``it is very unlikely that any Seifert type invariant can be constructed in this
general setting.'' 

We develop a  different approach to extending knot signatures to virtual knots. It starts
with defining signatures for \emph{almost classical knots}, which are virtual knots that can be represented by a homologically trivial knot in a thickened surface.
If $K$ is such a knot, then it admits a Seifert surface, which can be used to determine Seifert matrices $V^+$ and $V^-$ as in \cite{Boden-Gaudreau-Harper-2016}. 
In the case of a classical knot, 
the two matrices $V^+$ and $V^-$ are transposes of one another, but this is no longer true in the more general setting of almost classical knots. 
In that case, one can define signatures in terms of the symmetrization of either $V^+$ or $V^-$, and we show that the two symmetrizations are equal, hence so are their signatures. 

Just as for classical knots, the signature of an almost classical knot is invariant under  $S$-equivalence of the Seifert pair $(V^+,V^-)$. However, the $S$-equivalence class of $(V^+,V^-)$  does not give a well-defined invariant of $K$ but rather depends on the choice of Seifert surface $F$. This behavior is a departure from the situation for classical knots, but we view it as a feature rather than a bug; the reason being that for the purposes of obstructing sliceness, it is often useful to be able to consider a different Seifert surface for $K$ (see Subsection \ref{subsec-computations}).

In any case, the same ideas can be used to define the directed Alexander polynomials $\na^\pm_{K,F}(t)$ and $\om$-signatures $\widehat{\si}_\om(K,F)$ for almost classical knots. Like the signatures, these invariants obstruct sliceness and depend on the choice of Seifert surface $F$. The two main results here are Theorem \ref{thm:main} and \ref{thm:main-2}. The first result shows that for an almost classical knot $K$ with Seifert surface $F$, its topological slice genus $g^{\rm top}_s(K)$ is bounded below by $|\widehat{\si}_\om(K,F)|/2$. The second result shows that if $K$ is topologically slice, then there exist polynomials $f^\pm(t) \in \ZZ[t]$ such that $\na^\pm_{K,F}(t) = f^\pm(t) f^\pm(t^{-1})$. This result is the analogue of the Fox-Milnor condition in the virtual setting.  We apply these invariants to the problem of determining sliceness and the slice genus of almost classical knots up to six crossings, and the results are summarized in Table \ref{table-2}. 
This table and the others include classical knots, and our results show that the knot signature of a classical knot $K$ continues to give an effective lower bound on its virtual topological slice genus. 

To extend the signatures from almost classical knots to all virtual knots, we use parity projection, as defined by Manturov \cite{Manturov-2010},  and Turaev's notion of lifting of knots in surfaces \cite{Turaev-2008-a}.
A key result is Theorem \ref{thm_parity_conc}, which shows that if $K_0$ and $K_1$ are virtual knots and are concordant, then so are the virtual knots  $P_n(K_0)$ and $P_n(K_1)$ obtained under parity projection. Here $P_n$ denotes projection with respect to the mod $n$ Gaussian parity. For any virtual knot $K$, set ${K}_0 = P_0^\infty(K) = \lim_{n\to\infty} (P_0)^n (K)$, the image under stable projection. Then ${K}_0$ is an almost classical knot, and  its concordance class is determined by that of $K$. Thus, the signatures, $\om$-signatures, and directed Alexander polynomials of $K_0$ all lift  to give sliceness obstructions for $K$.
  
We end this introduction with a brief outline of the contents of this paper.
In Section \ref{sec-1}, we introduce the basic notions  such as Gauss diagrams, virtual knot concordance, Carter surfaces,  the virtual knot concordance group, and almost classical knots. 
In Section \ref{sec-3}, we introduce various invariants for almost classical knots, including the Alexander-Conway polynomials, the knot signature and nullity, and the directed Alexander polynomials and $\om$-signatures. 
In Section \ref{sec-4}, we construct virtual Seifert surfaces realizing any Seifert pair as arising from an almost classical knot. A modification of the construction shows how to realize any null-concordant Seifert pair by a ribbon almost classical knot diagram.
In Section \ref{sec-5}, we show how to compute the signatures, $\om$-signatures, and directed Alexander polynomials for almost classical knots. Included is a skein relation for  $\na^\pm_{K,F}(t)$ (see equation \eqref{eq-skein}) and a method for computing the signature under crossing changes (see equations \eqref{eq-Conway-1} and \eqref{eq-Conway-2}). These computations are applied to the problem of determining the slice genus of almost classical knots up to six crossings, which are given in Table \ref{table-2},
and used to provide evidence in support of the conjectured equality of the virtual and classical slice genera for classical knots with up to 12 crossings. 
In Section \ref{sec-2}, we review parity and relate parity projection to lifting knots along covers. The main theorem is Theorem \ref{thm_parity_conc}, showing that parity projection preserves concordance.
At the end of the paper, we present tables of almost classical knots up to six crossings along with  their Alexander-Conway polynomials
(Table \ref{table-1}), and their graded genera, signatures, $\om$-signatures, and slice genera (Table \ref{table-2}), and pairs of Seifert matrices (Table \ref{table-3}). Figure \ref{slice-Gauss} on p.~\pageref{slice-Gauss}
shows all slice almost classical knots up to six crossings with their slicings, and
Figure \ref{ACknot-diagrams} on p.~\pageref{ACknot-diagrams} shows almost classical knots up to six crossings realized as knots in thickened surfaces.

Notation: Throughout this paper all homology groups will be taken with $\ZZ$ coefficients unless otherwise noted. 
Decimal numbers such as 4.99 and 5.2012 refer to virtual knots in Green's tabulation \cite{Green}.

\section{Preliminaries} \label{sec-1}

Concordance is an equivalence relation on classical knots that was extended to virtual knots in \cite{Carter-Kamada-Saito, Turaev-2008-a, Kauffman-2015}. Following the latter two approaches, we define virtual knots and links as equivalence classes of Gauss diagrams, which we take a moment to explain. 

\subsection{Gauss diagrams}
\begin{figure}[ht]
\centering
\includegraphics[scale=0.820]{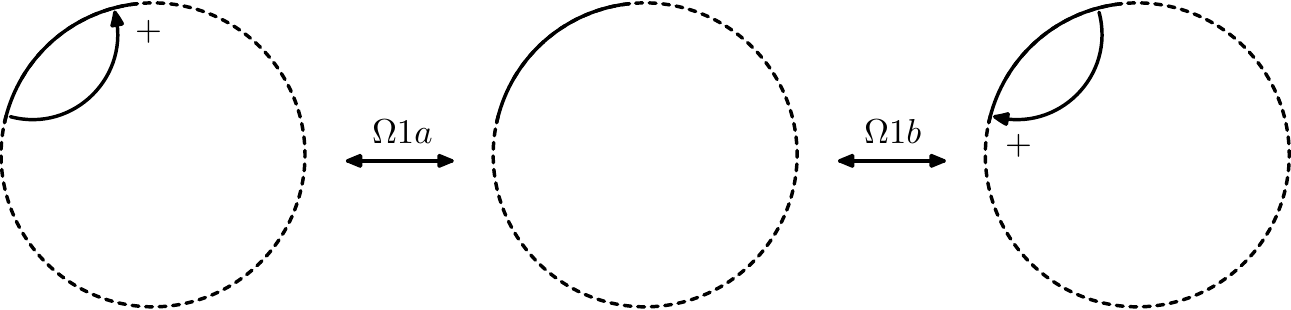} \smallskip \\
\includegraphics[scale=0.820]{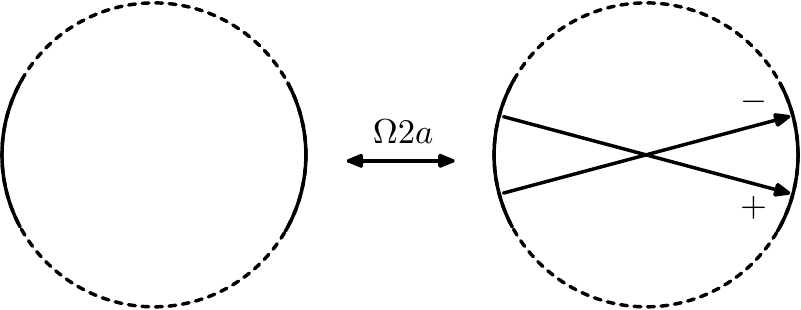}  \bigskip \\
 \includegraphics[scale=0.820]{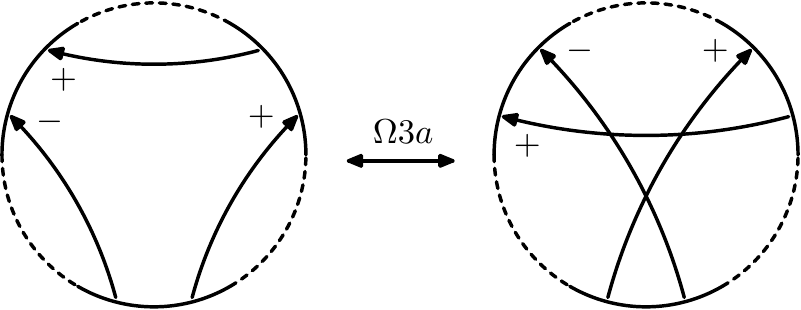}  \qquad  
  \includegraphics[scale=0.820]{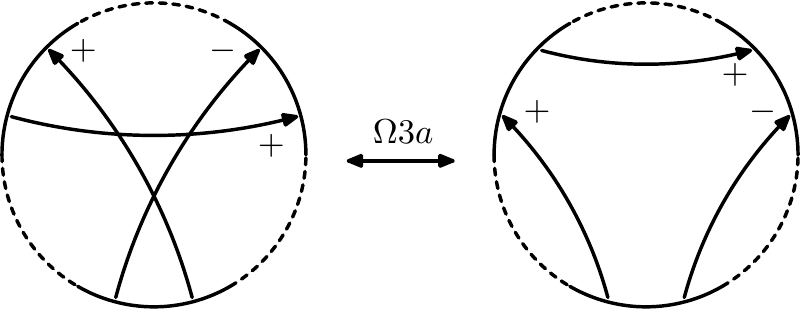}
\caption{Reidemeister moves for Gauss diagrams.}
\label{GD-RM}
\end{figure}

A Gauss diagram is a decorated trivalent graph consisting of one or more core circles, oriented counterclockwise, together with a finite collection of signed, directed chords connecting distinct pairs of points on the circles. Each core circle represents a knotted curve on a surface,  and the  directed chords, which are also called arrows, connect preimages of the double points of the underlying immersed curve; they point from the over-crossing arc to the under-crossing arc, and their sign ($+$ or $-$) indicates the writhe of the crossing. A virtual knot or link is then an equivalence class of Gauss diagrams under the equivalence generated by the Reidemeister moves.  In \cite{Polyak}, Polyak showed that all Reidemeister moves can be generated by the four moves $\Om1a, \Om1b, \Om 2a,$ and $\Om 3a$, which are depicted for diagrams on one component in Figure \ref{GD-RM}.

An equivalent and alternative definition for virtual knots and links is as equivalence classes of virtual knot and link diagrams as explained in \cite{Kauffman-1999}. Note that a virtual link diagram is said to be oriented if every component has an orientation. We use $K^r$ to denote the knot or link with its orientation reversed.

\subsection{Virtual knot concordance} \label{subsec-VKC}

We say that two virtual knots $K_0$ and $K_1$ are \emph{concordant} if $K_0$ can be transformed into  $K_1$ by a finite sequence of $b$ births, $d$ deaths, $s$ saddle moves, and Reidemeister moves, such that $s=b+d$. Births, deaths, and saddles are local moves on a virtual knot or link diagram; they are the same as in classical concordance. However, there is an equivalent description of them in terms of Gauss diagrams that is particularly convenient, and these are depicted in Figure  \ref{GD-conc}. Thus a saddle is an oriented smoothing along a chord whose endpoints are disjoint from the endpoints of all other chords. A birth is the addition of a disjoint unknotted component, whereas a death is its removal.  Saddles are indicated with a dotted line segment as in Figure \ref{GD-conc}.

\begin{figure}[ht]
\centering
\includegraphics[scale=0.82]{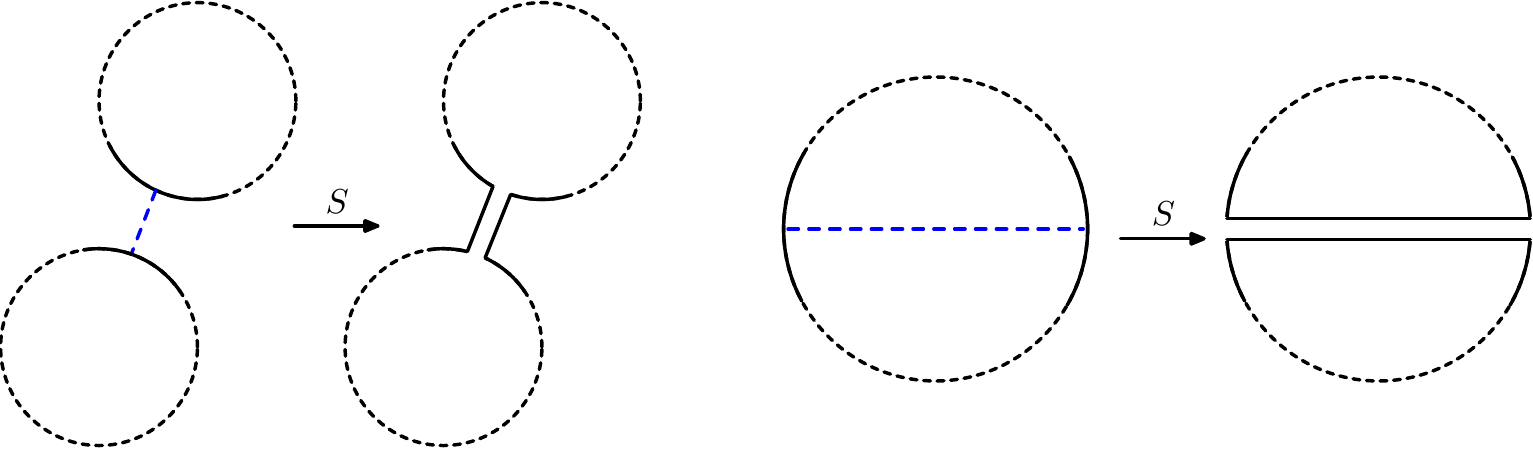}   
\caption{Concordance is generated by saddle moves, births and deaths.}
\label{GD-conc}
\end{figure}

Given an arbitrary virtual knot $K$, it is elementary to show that there exists a finite sequence of  births, deaths, and saddles transforming $K$ to the unknot. Given such a sequence, set $g= (s-b-d)/2$, the genus of the surface cobordism from $K$ to the unknot,  where $b,d,s$ are the numbers of births, deaths, and saddles, respectively. 
The \emph{slice genus} of $K$ is defined to be the minimum $g$ over all such sequences.  
The virtual knot $K$ is said to be \emph{slice} if it is concordant to the unknot, and it is called \emph{ribbon} if it is concordant to the unknot by a sequence of moves that includes only saddles and deaths. 

It is not known whether every slice virtual knot is ribbon. This is the virtual analogue of  Fox's question, which asks whether every classical slice knot is ribbon \cite{Fox-1962-b}. 
Note that there could be classical knots which are virtually ribbon but not classically ribbon, i.e., knots which admit ribbon virtual knot diagrams but not ribbon classical knot diagrams. Thus, a weaker version of Fox's question is whether every classical knot that is slice is virtually ribbon. 

It is tempting to define concordance for welded knots in an analogous way. Recall that  welded knots are equivalence classes of Gauss diagrams under the equivalence generated by Reidemeister moves along with the first forbidden move $f1$ shown in Figure \ref{GD-f1}. However, concordance of welded knots leads a trivial theory, and in fact one can show that every welded knot is concordant to the unknot \cite{Gaudreau-2018}.
\begin{figure}[ht]
\centering
\includegraphics[scale=0.820]{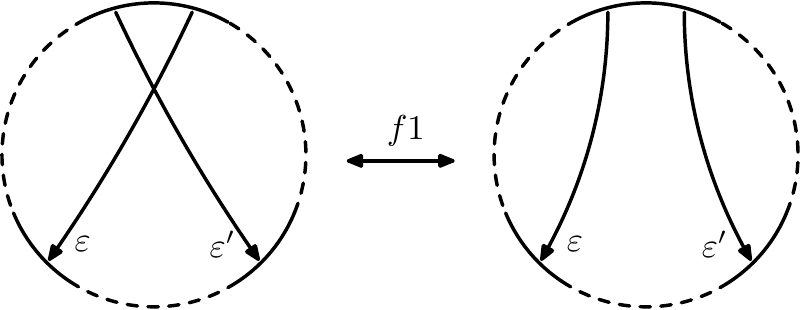} 
\caption{The first forbidden move.}
\label{GD-f1}
\end{figure}

\subsection{Concordance of knots in thickened surfaces}
Given a Gauss diagram $D$, we review the construction of the Carter surface $\Si$. Our discussion is based on \cite{Carter}, and the analogous construction for virtual knot diagrams can be found in \cite{Kamada-Kamada-2000}.

Suppose that $D$ has $n$ chords on the core circle O, which is oriented counterclockwise. The portions of O from one crossing to the next are called arcs. We thicken O to give an annulus and perform $n$ plumbings, one for each chord, according to its sign (see Figure \ref{plumbing}). This results in an oriented surface with boundary, and we form a closed surface by gluing disks to each boundary component of the plumbed annulus. The resulting closed oriented surface is called the \emph{Carter surface}.  

\begin{figure}[ht]
\def\svgwidth{0.85\textwidth} 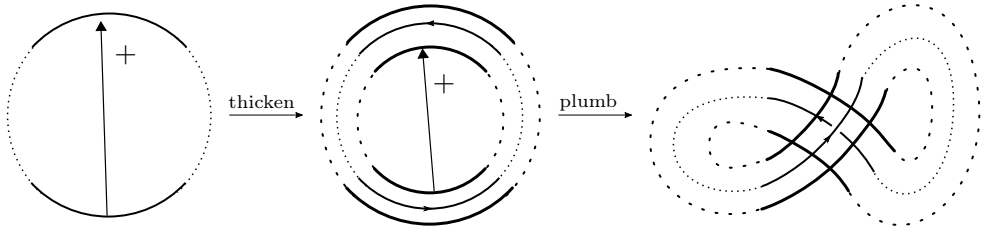
\caption{Plumbing along a chord in a Gauss diagram.}
\label{plumbing}
\end{figure}

Conversely, given a knot $K$ in a thickened surface $\Si \times I$, its crossing information determines a Gauss diagram which in turn determines a virtual knot. Stabilization of a knot $K$ in a thickened surface $\Si \times I$ is the result of performing surgery on $\Si$ by attaching a 1-handle disjoint from $K$. The opposite procedure is called  \emph{destablization}, and it involves surgery on $\Si$ removing a 1-handle disjoint from $K$. Notice that stablization and destablization do not affect the underlying Gauss diagram, so they preserve the associated virtual knot. Two knots $K_0$ in $\Si_0 \times I$ and $K_1$ in $\Si_1 \times I$ in thickened surfaces are \emph{stably equivalent} if they become equivalent under a finite number of stablizations and destablizations. By results of \cite{Carter-Kamada-Saito}, there is a one-to-one correspondence between virtual knots  and knots in thickened surfaces up to stable equivalence.

In \cite{Turaev-2008-a}, Turaev studied concordance for knots in thickened surfaces, which is defined as follows.

\begin{definition} \label{defn-conc}
Two oriented knots $K_0$ in $\Si_0 \times I$  and $K_1$ in $ \Si_1 \times I$ are called \emph{concordant} if there exists an  oriented $3$-manifold~$M$ with  
$\partial M \cong -\Si_0 \sqcup \Si_1$ and an annulus~$A \subset M \times I$ with $\partial A = -K_0 \sqcup K_1$. If the annulus $A$ is smoothly embedded, then $K_0$ and $K_1$ are said to be \emph{smoothly concordant}. If the annulus $A$ is locally flat, then  $K_0$ and $K_1$ are said to be \emph{topologically concordant}.
A knot $K$ in $\Si \times I$ which is (smoothly or topologically) concordant to the unknot is called (smoothly or topologically) \emph{slice}. 
\end{definition}

One can verify that stably equivalent knots are smoothly concordant, and thus both notions of concordance of knots in surfaces induce equivalence relations on virtual knots. By \cite[Lemma 12]{Carter-Kamada-Saito}, it follows that the notion of smooth concordance  for virtual knots is equivalent to the definition in Subsection \ref{subsec-VKC}. 
See Figure \ref{comparison} for an illustration of this correspondence.

\begin{figure}[ht]
\def\svgwidth{0.95\textwidth} 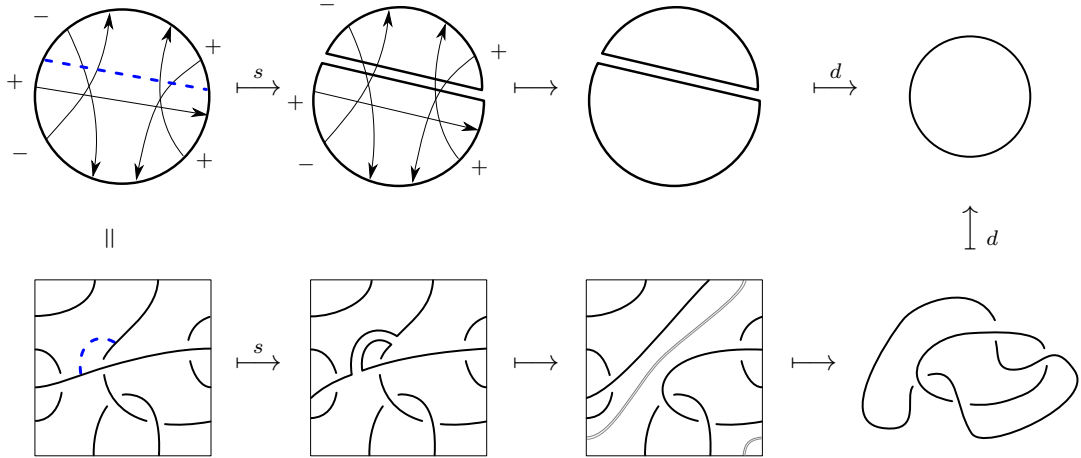
\caption{A side-by-side comparison of virtual knot cobordism, and cobordism of the knot in the surface. The torus is destabilized along the grey curve to obtain the bottom right diagram.}
\label{comparison}
\end{figure}

\subsection{Almost classical knots} \label{subsec-ack}

In this subsection, we give three equivalent definitions of the index of a crossing $c$ in a virtual knot $K$ in terms of its three representations, namely as a virtual knot diagram, as a Gauss diagram, and as a knot diagram on a surface.

First, we review the definition of \emph{virtual linking numbers}. Suppose  $L=J \cup K$ is a virtual link with two components, and define $\vlk(J, K)$ as the sum of the writhe of the crossings where $J$ goes over $K$. Notice that if $L$ is classical, then $\vlk(J, K) = \vlk(K, J)$, but the virtual Hopf link in Figure \ref{vhopf} shows that this is not true in general for virtual links. 

\begin{figure}[ht]
\includegraphics[scale=0.90]{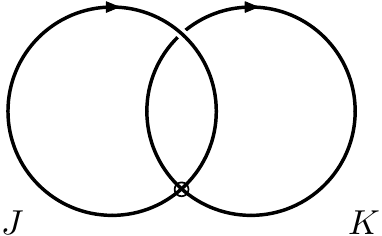}  
\caption{The virtual Hopf link has $\vlk(J, K) = 1$ and $\vlk(K, J) =0$ } \label{vhopf}
\end{figure}

Suppose $K$ is a virtual knot diagram and $c$ is a crossing of $K$. Then the oriented smoothing of $K$ at $c$ is a virtual link diagram with two components $K'$ and $K''$, where $K'$ denotes the component that contains the outward pointing over-crossing arc at $c$. The index of $c$ is defined by setting $$\ind(c)=\vlk(K', K'')-\vlk(K'',K').$$

Now suppose $c$ is an arrow in a Gauss diagram $D$, which we draw with $c$ pointing up. Set $$\ind(c)= r_+(c)-r_-(c)+\ell_-(c)-\ell_+(c), $$ where $r_\pm(c)$ are the numbers of $\pm$-arrows intersection $c$ and pointing to the right, and $\ell_\pm(c)$ are the numbers of $\pm$-arrows pointing to the left. An easy exercise shows that $\vlk(K', K'') = r_+(c)-r_-(c)$ and $\vlk(K'',K') = \ell_+(c)-\ell_-(c),$ thus this definition of $\ind(c)$ agrees with the previous one.

Lastly, suppose $K\subset\Si \times I$ is a knot in a thickened surface, and let $p \co \Si \times I \to \Si$ denote the projection map. For any crossing $c$ of $K$,  let $K'$ and $K''$ be the two components obtained from the oriented smoothing at $c$ as above and set 
\begin{equation}\label{eq:intersect}
\ind(c)=[p_*(K')]\cdot [p_*(K'')],
\end{equation}
the algebraic intersection of the homology classes $[p_*(K')], [p_*(K'')]\in H_1(\Si)$. 

Given a classical knot $K$ represented as a knot in $S^2 \times I$, then the Jordan curve theorem implies that every chord $c$ of $K$ has $\ind(c)=0.$  The converse is however false, and virtual knots such that each chord has index zero are called \emph{almost classical} \cite{Silver-Williams-2006b}.

\begin{definition}  \label{defn:AC}
A virtual knot $K$ is \emph{almost classical} if it can be represented by a Gauss diagram $D$ such that every chord $c$ satisfies $\ind(c) = 0$. 
\end{definition}

One can alternatively define almost classical knots as those admitting Alexander numberable diagrams, and this approach generalizes to define almost classical links. As explained in \cite{Boden-Gaudreau-Harper-2016}, a virtual link is almost classical if and only if it can be represented by a homologically trivial link in a thickened surface. 
If $K \subset \Si \times I$ is an almost classical knot or link, then by the modified Seifert algorithm in \cite[\S 6]{Boden-Gaudreau-Harper-2016}, one can construct a compact oriented surface $F$ in $\Si \times I$ with $\partial F =K$. We call $F$ a \emph{Seifert surface} for $K$.


\section{Signatures for almost classical knots} \label{sec-3}

In this section, we recall the Seifert pairing from \cite{Boden-Gaudreau-Harper-2016} and use it to define invariants of almost classical knots. We extend the signature and Levine-Tristram signatures to almost classical knots, where unlike in the classical case, they depend on the choice of Seifert surface. Also, we introduce the directed Alexander polynomials of the pair $(K,F)$, where $K$ is an almost classical knot and $F$ is a Seifert surface. These are two distinct generalizations of the Alexander-Conway polynomial. The main results are Theorems \ref{thm:main} and \ref{thm:main-2}, which show that the Tristram-Levine signatures provide lower bounds on the slice genus and
that the directed Alexander polynomials give slice obstructions.

\subsection{Linking numbers} \label{subsec:link-num}
We review the basic properties of the linking numbers in $\Si \times I$. 
Let $J,K$ be disjoint oriented knots in $\Si \times I$. By \cite[Proposition 7.1]{Boden-Gaudreau-Harper-2016}, we see
that  the relative homology group $H_1(\Si \times I\sm J,\Si \times \{1\})$ is infinite cyclic and generated by a meridian $\mu$ of $J$. Let $[K]$ denote the homology class of $K$ in  $H_1(\Si \times I\sm J,\Si \times \{1\})$, and define $\lk(J,K)$ to be the unique integer $m$ such that $[K]=m \mu$. The linking number $\lk(J,K)$ can be computed as $J\cdot B$, the algebraic intersection number, where $B$ is a 2-chain  in $\Si \times I$ with $\partial B = K -v$ for some 1-cycle $v$ in $\Si \times \{1\}$. In practice, to compute $\lk(J,K)$, we count, with sign, the number of times that $J$ crosses above $K$ in $\Si \times I$. Here, ``above'' is taken with respect to the positive $I$-direction in $\Si \times I$. Thus, computing the linking number  $\lk(J,K)$ is formally similar to computing the virtual linking number, which were introduced in Subsection \ref{subsec-ack}.

Linking numbers in $\Si \times I$ are not symmetric but rather satisfy (see \S10.2 \cite{Cimasoni-Turaev}):   
\begin{equation} \label{eq:linking}
\lk(J,K) - \lk(K,J) = p_* [K]\cdot p_*[J], 
\end{equation}
where $\cdot$ is the algebraic intersection number in $\Si$ of the projections of $J$ and $K$ to $\Si$. 

\subsection{Seifert forms and Seifert matrices} \label{subsec:SfSm}
Suppose that $K$ is an almost classical link, realized as a homologically trivial link in $\Si \times I$. Suppose further that $F$ is a  Seifert surface for $K$. The Seifert forms $\th^{\pm}\co H_1(F) \times H_1(F) \to \ZZ$ are defined by $\th^{\pm}(x,y)=\lk(x^{\pm},y)$, where $x^{\pm}$ denote the $\pm$ ``pushoffs'' of $x$ into $\Si \times I \sm F$. 
Each of $\th^+$ and $\th^-$ is bilinear, and they satisfy 
\begin{equation} \label{skew-ness}
\th^-(x,y) -\th^+(x,y) =\lk(x^-,y) -\lk(x^+,y) =  \langle x,y\rangle_F,
\end{equation}
where  $\langle\cdot,\cdot\rangle_F$ denotes the intersection form on the surface $F$.

The Seifert forms $(\th^+,\th^-)$ are represented by a pair of Seifert matrices $(V^+,V^-)$, which we introduce next.
The homology group $H_1(F)$ is a free abelian group of rank $n$ for some integer $n \geq 0$.
Let $\{a_1,\ldots, a_{n}\}$ be an ordered basis for $H_1(F)$, and define
the two Seifert matrices $(V^+,V^-)$ by setting the $i,j$ entry of $V^\pm$ equal to $\th^\pm(a_i,a_j)$.
The Seifert matrices are not invariants of the almost classical link $K$; they depend on the choice of Seifert surface $F$ and basis for $H_1(F)$. A different choice of basis for $H_1(F)$ alters the matrices $(V^+,V^-)$ by simultaneous unimodular congruence. Recall that
an integral square matrix $P$ is said to be \emph{unimodular} if it has determinant $\det(P) = \pm 1$, and that two integral square matrices $V,W$ are said to be \emph{unimodular congruent} if there is a unimodular matrix $P$ such that $W =  P V P^\uptau.$

Assume now that $K$ is almost classical knot with Seifert surface $F$.
For any pair of Seifert matrices $(V^+,V^-)$, since the intersection form $\langle\cdot,\cdot\rangle_F$ on $F$ is skew-symmetric and non-singular, equation \eqref{skew-ness} implies that $V^- -V^+$ is a skew-symmetric matrix satisfying $\det(V^- -V^+)=1.$ In Subsection \ref{subsec:proof_real},  we will see that any pair of integral $2g \times 2g$ matrices 
satisfying these conditions occurs as the Seifert matrices for some
almost classical knot (cf. Theorem \ref{thm_realize}).

\subsection{The Alexander-Conway polynomial} 
Suppose $K$ is an almost classical link, realized as a homologically trivial link in $\Si \times I$, and $F$ is a  Seifert surface for $K$. Let $(V^+, V^-)$ be the Seifert pair associated to $K$, $F$, and a choice of basis for $H_1(F).$
The Alexander-Conway polynomial of $K$ is then defined by setting
\[
\De_K(t)=\det\left(t^{1/2} \, V^{-} - \, t^{-1/2} \, V^{+}\right).
\]
A proof that $\De_K(t)$ is independent of the choice of Seifert surface $F$ and basis for $H_1(F)$ can be found in \cite{Boden-Gaudreau-Harper-2016}, and it follows that $\De_K(t) \in \ZZ[t^{1/2},t^{-1/2}]$ is well-defined up to multiplication by $t^k$, where $k$ is the \emph{virtual genus} of $K$, namely the smallest genus among all surfaces $\Si$ containing a representative for $K$.
In case $K$ is an almost classical knot, then $\De_K(t) \in \ZZ[t,t^{-1}]$. 

For classical links, $V^{-} = (V^{+})^\uptau,$ and this shows that $\De_K(t)$ is a \emph{balanced} polynomial, i.e., that it satisfies $\De_K(t) = \De_K(t^{-1}).$ 
Indeed, the Alexander polynomial of any classical knot $K$ satisfies: (i) $\De_K(1)=1$ and  (ii) $\De_K(t^{-1}) \doteq \De_K(t)$, where we write $f(t) \doteq g(t)$ for $f(t),g(t) \in \ZZ[t^{1/2},t^{-1/2}]$ if $f(t)= t^\ell g(t)$ for some integer $\ell$.
A well-known result due to Seifert \cite{Seifert-1935} shows that any integral polynomial of even degree satisfying (i) and (ii) occurs as the Alexander polynomial of a classical knot (see \cite[Theorem[8.13]{Burde-Zieschang-Heusener} for a proof). 

If $K$ is an almost classical knot, it is no longer true that  $V^{-}$ equals $(V^{+})^\uptau,$ thus $\De_K(t)$ is not necessarily balanced. Nevertheless, the Alexander-Conway polynomial of any almost classical knot continues to satisfy the first condition.
In Subsection \ref{subsec:proof_alex} we will see that every integral polynomial $\De(t)$ with $\De(1)=1$ occurs as the Alexander polynomial of some 
almost classical knot (cf. Theorem \ref{thm:real_alex}).

\subsection{Signature for almost classical knots}  \label{subsec:sign-AC}
Suppose that $K$ is an almost classical link, realized as a link in $\Si \times I$, and that $F$ is a Seifert surface for $K$ in $\Si \times I$. The \emph{signature} and \emph{nullity} of the pair $(K,F)$ are defined to be 
\[
\si^\pm(K,F)=\sig(V^{\pm}+(V^{\pm})^{\uptau}) \quad \text{and} \quad n^\pm(K,F)=\nullity(V^{\pm}+(V^{\pm})^{\uptau}).
\] 
The following lemma implies that $\si^+(K,F)= \si^-(K,F)$ and $n^+(K,F)= n^-(K,F)$, and henceforth we use $\si(K,F)$ and $n(K,F)$ to denote the signature and nullity of $(K,F)$.

\begin{lemma} \label{first-lem}
$V^++(V^+)^{\uptau}=V^-+(V^-)^{\uptau}$.
\end{lemma}
\begin{proof} First note that for two knots $J_1,J_2$ on $F$, $\lk(J_1^{\pm},J_2)=\lk(J_1,J_2^{\mp})$. Let $\{a_1, \ldots, a_{2g}\}$ be a collection of simple closed curves on $F$ giving a basis for $H_1(F)$.   
Writing the Seifert matrices with respect to this basis, we see that $V^{\pm}$ has $(i,j)$ entry  
\begin{eqnarray*}
\lk(a_i^{\pm},a_j) &=& \lk(a_j,a_i^{\pm})+ p_*[a_i]\cdot p_*[a_j] \\
              &=& \lk(a_j^{\mp},a_i)+p_*[a_i]\cdot p_*[a_j].  
\end{eqnarray*} 
The first term is the $(i,j)$ entry of the transposed matrix $(V^{\mp})^{\uptau}$, thus it follows that $V^+ -V^-=(V^-)^{\uptau}-(V^+)^{\uptau}$, which proves the lemma. 
\end{proof}

As in the classical case, it is useful to generalize the signature to Tristram-Levine signature functions, and in the case of an almost classical knot, we actually get a pair of signature functions. To define the signature functions, it is first necessary to define the  \emph{directed Alexander polynomials}, which are given by setting 
$$\na^\pm_{\! K,F}(t)=\det \left( t^{1/2} \, V^{\pm}-t^{-1/2}(V^{\pm})^{\uptau}\right).$$ 
Notice that $\na^\pm_{\! K,F}(t) \in \ZZ[t, t^{- 1}]$, and they are balanced polynomials, namely they satisfy $\na^\pm_{\! K,F}(t)=\na^\pm_{\! K,F}(t^{-1})$. We call $\na^+_{\! K,F}(t)$ the \emph{up} Alexander polynomial and $\na^-_{\! K,F}(t)$ the \emph{down} Alexander polynomial.
Both polynomials depend on the choice of Seifert surface $F$, and the up and down Alexander polynomials are generally distinct from one another and from the Alexander-Conway polynomial. Of course, for classical knots, all three polynomials coincide; i.e., if $K$ is classical then $\De_K(t) = \na^+_{\! K,F}(t) =\na^-_{\! K,F}(t)$. In particular, in this case the up and down Alexander polynomials are independent of the choice of Seifert surface. 

For $\om$ a complex unit number, $\om \neq 1$, the matrices $(1-\om) V^{\pm}+(1-\wbar{\om}) (V^{\pm})^{\uptau}$ are Hermitian, and we define the \emph{$\om$-signatures} by setting
\[
\widehat{\si}_{\om}^{\pm}(K,F)=\sig \left((1-\om) V^{\pm}+(1- \wbar{\om}) (V^{\pm})^{\uptau}\right).
\] 
If the matrix $(1-\om) V^{\pm}+(1-\wbar{\om}) (V^{\pm})^{\uptau}$ is non-singular, then we will show that $\widehat{\si}_{\om}^{\pm}(K,F)$ provides an obstruction to sliceness of $K$. We will further relate non-singularity of the above Hermitian matrix to the vanishing of $\na^\pm_{\! K,F}(\om).$ As usual, for $\om = -1,$ we have that $\widehat{\si}_{-1}^{\pm}(K,F)=\si(K,F)$. Lemma \ref{first-lem} implies that $\na^+_{K,F}(-1)=\na^-_{K,F}(-1)$, and notice also that $\na^\pm_{K,F}(-1) \neq 0 \; \Longleftrightarrow \; n(K,F)=0.$

\subsection{Signature and concordance of virtual knots}
In this subsection, we study the signature functions and the directed Alexander polynomials as obstructions to sliceness.  The method of the proof follows that of the classical case  (cf. \cite[Chapter 8]{Lickorish}). The main issue is to adapt the classical proof to linking numbers in $\Si \times I$, as defined in \S \ref{subsec:link-num}. 

The main results are Theorems \ref{thm:main} and \ref{thm:main-2}, and they will follow from a sequence of lemmas, which we now state and prove.

For the next lemma, let $K \subset \Si \times I$ be a knot in a thickened surface, $W$ a compact oriented $3$-manifold with $\partial W=\Si,$ and $S$ a locally flat oriented surface in $W \times I$ with $K=\partial S$. By \cite[\S 9.3]{Freedman-Quinn}, $S$ has a normal bundle homeomorphic to $S \times D^2$, which we denote by $N$. 

\begin{lemma} \label{lem-inc-iso}
The inclusion of pairs $(\Si \times I\sm K, \Si \times \{1\}) \hookrightarrow (W \times I\sm S, W \times \{1\})$ induces an isomorphism in homology $H_1(\Si \times I\sm K, \Si \times \{1\}) \to H_1(W \times I\sm S, W \times \{1\}) \cong \ZZ$.
\end{lemma}
\begin{proof} 
Let $i\co  W \times I \sm S \to W \times I$ be inclusion, $p\co  W \times I \to W$ be projection onto the first factor, and $j\co  W \to W \times I \sm S$ be defined by $x \to (x,1)$. Note that $p \circ i \circ j =\id_W$. Since $p_*\co H_*(W \times I) \to H_*(W)$ is an isomorphism, it follows that $i_* \co H_*(W \times I \sm S) \to H_*(W \times I)$  is a split surjection and $j_* \co H_*(W) \to H_*(W \times I \sm S)$ is a split injection.

Let $\Int(N)$ denote the interior of $N \subset W \times I$ and
apply a Mayer-Vietoris argument to $W \times I=(W \times I \sm S) \cup \Int(N)$
to obtain
\[
\xymatrix{
0 \ar[r] & H_1(S \times S^1) \ar[r] &H_1(W \times I \sm S) \oplus H_1(\Int(N)) \ar[r] & H_1(W \times I) \ar[r]& 0.}
\]
Since $j_*$ is split, the long exact sequence for the pair $(W \times I\sm S,W \times \{1\})$ splits, giving the short exact sequence:
\[
\xymatrix{
0 \ar[r] & H_1(W \times \{1\}) \ar[r] &H_1(W \times I \sm S)  \ar[r] & H_1(W \times I\sm S, W \times \{1\}) \ar[r] & 0.
}
\]
From the two short exact sequences, we conclude that $H_1(W \times I\sm S, W \times \{1\})$ is infinite cyclic. Since the summand $H_1(S^1)$ of $H_1(S \times S^1)$ is generated by a meridian of $K$, we conclude that the induced map  $H_1(\Si \times I\sm K, \Si \times \{1\}) \to H_1(W \times I\sm S, W \times \{1\})$ sends a meridian of $K$ to a generator of  $H_1(W \times I\sm S, W \times \{1\})$. Thus the inclusion of pairs induces an isomorphism in homology as claimed. 
\end{proof}

The next two results are adapted  from Lemmas 8.13 and 8.14 in \cite{Lickorish}. We provide detailed proofs for completeness.

\begin{lemma} \label{lemma_L813}
For $i=1,2$, suppose $\varphi_i\co Y_i \to W \times I$ are maps of orientable surfaces into $W \times I$ such that $\im(\varphi_1) \cap \im(\varphi_2)=\varnothing$ and that $K_i=\varphi_i(\partial Y_i)$ is a knot in $(\partial W) \times I$. Then $\lk(K_1,K_2)=0$.
\end{lemma}

\begin{proof} Using standard arguments, it may be assumed each $\varphi_1, \varphi_2$ are locally flat embeddings. By definition, $\lk(K_1,K_2)$ is the homology class of $K_2$ in $H_1(\Si \times I\sm K_1, \Si \times \{1\}) \cong H_1(W \times I\sm \im(\varphi_1),W \times \{1\})$. Since $\im(\varphi_2) \subset W \times I\sm \im(\varphi_1)$, $K_2$ is trivial in $H_1(\Si \times I\sm K_1, \Si \times \{1\})$. Thus, $\lk(K_1,K_2)=0$. 
\end{proof}

\begin{lemma} \label{lemma_L814} 
Let $F$ be a Seifert surface for $K$ in $\Si \times I$ and let $S$ be a locally flat orientable surface in $W \times I$ with $\partial S = K$. Then $F \cup S$ bounds a two-sided $3$-manifold $M \subset W \times I$ with $M \cap (\Si \times I)=F$.
\end{lemma}

\begin{proof} 
As above, let $N \approx S \times D^2$ denote the normal bundle of $S$ in $W\times I$. We will construct a map of pairs $$\psi\co  (W\times I \sm N, W \times \{1\}) \lto (S^1,\{-1\})$$ 
inducing an isomorphism on the relative first homology groups and so that $M=\psi^{-1}(1)$ is a 3-manifold with the desired properties. Throughout, we identify $H_1(S^1,\{-1\}) \cong H_1(S^1).$

Let $X=\Si \times I \sm \Int(N(K))$ be the complement of an open tubular neighborhood of $K$ in $\Si \times I$, and
notice that $X = \Si \times I \cap \left(W\times I \sm \Int(N)\right).$  
In the following, we will identify $F$ with $F \cap X$. 

Define the map $\psi$ first on $X$ so that $\psi(x,t) = e^{i \pi t}$ for $(x,t) \in F \times [-1,1]$, a product neighborhood of $F$ in $X$,  and so that $\psi$ maps the rest of $X$ to $-1 \in S^1.$ Notice that $\psi$ sends $\Si \times \{1\}$ to $-1 \in S^1,$ and that $\psi_* \co H_1(X,\Si \times \{1\}) \lto H_1(S^1)$ is an isomorphism.

Now extend $\psi$ over the rest of $\partial N$  so that $\psi^{-1}(1) = F \cup (S \times \{x_0\})$, where $x_0 \in \partial D^2$ is chosen so that $\partial S \times \{x_0\}$ is a longitude for $K$.

In order to extend $\psi$ from $X \cup \partial N$ to all of $W\times I \sm \Int(N),$ we 
use obstruction theory. We will
work with an arbitrary but fixed triangulation of the pair $(W\times I \sm \Int(N), W \times \{1\})$. Let $T$ be a spanning tree in the 1-skeleton which includes a maximal tree in $X \cup \partial N$.
Extend $\psi$ over $T$ so that it maps every 0-simplex and 1-simplex in $W \times \{1\}$ to $-1 \in S^1$.
For a 1-simplex $\si$ not in $T,$ define $\psi$ so that, for a 1-cycle $z$ obtained as the sum of $\si$ and a 1-chain in $T$, $[\psi(z)] \in H_1(S^1)$ is the image of $[z]$ under the isomorphism of Lemma \ref{lem-inc-iso} 
$$H_1(W\times I \sm \Int(N), W \times \{1\}) \stackrel{\cong}{\longleftarrow} H_1(X,\Si \times \{1\}) \stackrel{\psi_*}{\lto} H_1(S^1).$$
To extend $\psi$ over the 2-simplices, note that
the boundary of any 2-simplex $\tau$ is evidently trivial in $H_1(W\times I \sm \Int(N), W \times \{1\}).$ Thus
$[\psi (\partial \tau)] = 0$ in $H_1(S^1)$, and $\psi$ can be extended over $\tau.$ 
Similarly, $\psi$ can be extended over all 3-simplices, and then over all 4-simplices,
such that $W \times \{1\}$ is mapped to $-1\in S^1$.

Choose a triangulation of $S^1$ in which 1 is not a vertex so that 
$\psi$ becomes a simplicial map of pairs, and notice that $M=\psi^{-1}(1)$ is a bi-collared 3-manifold in $W \times I$ with boundary $\partial M = F \cup S$ as claimed.
\end{proof}

The following theorem establishes the slice obstructions and slice genus bounds from the directed Alexander polynomials and the signature functions.

\begin{theorem}\label{thm:main}
Let $K$ be an almost classical knot, represented as a knot in the thickened surface $\Si \times I$ with Seifert surface $F$, and suppose $W$ is an oriented $3$-manifold with $\partial W=\Si$ and  $S$ is a locally flat orientable surface in $W \times I$ with $\partial S=K.$  
 
If $\om \neq 1$ is a unit complex number such that $\na^{\pm}_{\! K,F}(\om) \ne 0$, then $|\widehat{\si}^{\pm}_{\om}(K,F)| \leq 2 \, {\rm genus} (S).$ Thus if $K$ is topologically slice, then $\widehat{\si}^{\pm}_{\om}(K,F)=0$, and in particular $\si(K,F)=0$.
\end{theorem}

\begin{proof} 
Since it is always true that $|\widehat{\si}^{\pm}_{\om}(K,F)| \leq 2 \, {\rm genus} (F),$ we can assume that
${\rm genus} (S) \leq {\rm genus} (F)$. 
Let $M$ be the $3$-manifold given by Lemma \ref{lemma_L814}, so $M$ is compact and oriented and $\partial M=F \cup S$ is a closed surface of genus $g+k$, where $g$ is the genus of $F$ and $k$ is the genus of $S$. By \cite[Lemma 8.16]{Lickorish}, the subspace $U \subset H_1(\partial M;\ZZ)$ consisting of elements that map to zero under inclusion into $H_1(M,\QQ)$ has ${\rm rank}(U) = g+k.$ Thus $U \cap H_1(F; \ZZ)$ has rank at least $g-k$, and we can choose an integral basis $[f_1], \ldots,[f_{2g}]$ of $H_1(F;\ZZ)$ such that $[f_1],\ldots,[f_{g-k}]$ lie in $U$ and each $f_i$ is a simple closed curve on $F$. 

Let $V^{\pm}$ be the $2g \times 2g$ Seifert matrix $(\lk(f_i^{\pm},f_j))$. For $i=1, \ldots, g-k$, there are integers $n_i$ such that $n_i [f_i]$ vanishes in $H_1(M,\ZZ)$. It follows that the curve $n_i f_i$ bounds the image of an orientable surface $Y_i$ mapped into $M$. Pushing the $Y_i$ into $M \times \{\pm 1\}$ gives surfaces $Y_i^{\pm}$ disjoint from each $Y_j$. Thus by Lemma \ref{lemma_L813}, $\lk((n_i f_i)^{\pm}, n_i f_j)=0$ for $1 \leq i,j \leq g-k.$ This implies that there is a $(g-k) \times (g-k)$ block of zeros in the upper left hand corner of $V^{\pm}$.

In general, given a quadratic form $Q$ over a field, a subspace consisting of elements $x$ for which $Q(x)=0$ is called \emph{isotropic}, and the isotropy index of $Q$ is the dimension of a maximal isotropic subspace. If $Q$ is non-singular with rank $N$ and signature $\si$, then it is well-known that the isotropy index  is given by $\min((N+\si)/2, (N-\si)/2),$ where  $(N+\si)/2 =n_+$ and  $(N-\si)/2=n_-$, the number of positive and negative eigenvalues of $Q$.  

Since $\om \ne 1$ and $\na_{K,F}^{\pm}(\om) \ne 0$, the matrix: 
\[
(1-\om) V^{\pm}+(1-\wbar{\om}) (V^{\pm})^\uptau=-(1-\wbar{\om})(\om V^{\pm}-(V^{\pm})^{\uptau})
\]
is non-singular.  Therefore, the quadratic form associated to the Hermitian matrix 
$(1-\om) V^{\pm}+(1-\wbar{\om}) (V^{\pm})^\uptau$ is also non-singular.
Since the Seifert matrices $V^{\pm}$ both have a $(g-k) \times (g-k)$ block of zeros, so does the Hermitian matrix. Thus the isotropy index of the quadratic form is at least $g-k.$ Let $\si = \widehat{\si}^{\pm}_{\om}(K,F)$ be the signature of this quadratic form.
Since it has rank $2g$, we see that $2(g-k)\leq \min(2g+\si, 2g-\si).$ Using this, one can easily show that $0 \leq |\widehat{\si}^{\pm}_{\om}(K,F)| \leq 2k,$ and this completes the proof of the theorem.
\end{proof}

\begin{figure}[ht]
\centering
\includegraphics[scale=0.9]{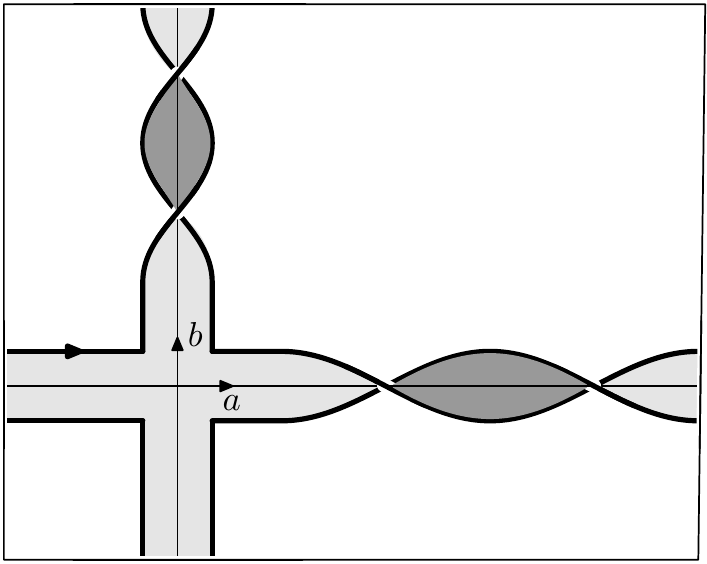}
\caption{A Seifert surface for $K=4.105$, drawn as a knot in $T^2 \times I$.}
\label{AC-4-105}
\end{figure}

\begin{example}
The almost classical knot $K=4.105$ has Seifert surface $F$ and basis $\{a,b\}$ for $H_1(F)$ shown in Figure \ref{AC-4-105}.
By our convention, the positive push-offs $a^+$ and $b^+$ are obtained by pushing up along the lighter region of $F$ and down along the darker region.
From this, one can readily compute that $\lk(a^\pm,a)=1= \lk(b^\pm,b),$
$\lk(a^+,b)=0= \lk(b^+,a)$, $\lk(a^-,b)=1$ and $\lk(b^-,a)=-1$. Thus,  the two Seifert matrices are given by
 $$V^+ = \begin{bmatrix}1 & 0 \\ 0 & 1 \end{bmatrix} \quad \text{and} \quad
V^- = \begin{bmatrix}1 & 1 \\ -1 & 1 \end{bmatrix}.$$ 
Therefore $V^+ + (V^+)^\uptau =V^- + (V^-)^\uptau$ is a diagonal matrix with signature 2, and we see that $\si(K,F) = 2$. Theorem \ref{thm:main} applies to show that 4.105 is not slice. On the other hand, since 4.105 admits a genus one cobordism to the unknot, we see that its slice genus is equal to one.
\end{example}

The Alexander polynomial of a classical slice knot must satisfy the famous Fox-Milnor condition, which asserts that $\De_K(t) = f(t)f(t^{-1}) $ for some $f(t) \in \ZZ[t].$ The following example shows that the Fox-Milnor condition fails for almost classical slice knots.

\begin{example}
The almost classical knot $K=5.2160$ has Alexander polynomial $\De_K(t) = t-1+t^{-1}$ (see Table \ref{table-1}), which is the same as the Alexander polynomial of the trefoil. Notice that $\De_K(t)$ does not satisfy the Fox-Milnor condition for any $f(t) \in \ZZ[t]$, even though $5.2160$ is known to be slice (see Figure \ref{slice-Gauss}). 
\end{example}

In Theorem \ref{thm:real_null2}, we will prove that \emph{any} integral polynomial $\De(t)$ with $\De(1)=1$ occurs as the  Alexander polynomial of an almost classical slice knot.
Although the Fox-Milnor condition does not extend to almost classical knots using the Alexander polynomial, the following alternative formulation gives
useful slice criteria in terms of
 the \emph{directed} Alexander polynomials.

\begin{theorem}\label{thm:main-2}
Let $K$ be an almost classical knot, represented as a knot in the thickened surface $\Si \times I$ with Seifert surface $F$. If $K$ is topologically slice, then there are polynomials $f^\pm(t) \in \ZZ[t]$ such that $\na^{\pm}_{\! K,F}(t)=f^\pm(t)f^\pm(t^{-1})$.
\end{theorem}

\begin{proof} Let $D$ be a slice disk for $K$. Repeating the previous argument with $S=D$, it follows that there is a $g \times g$ block of zeros in the upper left hand corner of the Seifert matrices $V^\pm$.  Thus, we can write these matrices in block form:
$$V^{\pm} = 
\begin{bmatrix} \mathbf{0} & Q^\pm \\ R^\pm & S^\pm
\end{bmatrix},$$
where $Q^\pm, R^\pm, S^\pm$ are $g \times g$ integral matrices.
 
Thus it follows that
\begin{eqnarray*}
\na^\pm_{\! K,F}(t) & = & \det \begin{bmatrix} \mathbf{0} & t^{1/2} Q^\pm -t^{-1/2}(R^\pm)^\uptau \\ t^{1/2}R^\pm - t^{-1/2} (Q^\pm)^\uptau & 
t^{1/2} S^\pm - t^{-1/2} (S^\pm)^\uptau
\end{bmatrix}, \\
& = & \det\left( t Q^\pm -(R^\pm)^\uptau\right) \; \det\left(t^{-1} Q^\pm - (R^\pm)^\uptau\right).
\end{eqnarray*}
Taking $f^\pm(t)= \det\left( t Q^\pm -(R^\pm)^\uptau \right)$, the conclusion follows.
\end{proof}

\begin{example}
For the almost classical knot $K=6.87857$ with Seifert surface $F$ from the diagram in Figure \ref{ACknot-diagrams}, its Seifert pair is
$$V^+ = \begin{bmatrix}2 & 0 \\ 0 & -1 \end{bmatrix} \quad \text{and} \quad
V^- = \begin{bmatrix}2 & 1 \\ -1 & -1 \end{bmatrix}.$$ 
The associated signature and $\om$-signatures all vanish, and so 
Theorem \ref{thm:main} is inconclusive on the question of whether this knot is slice.

On the other hand, the directed Alexander polynomials are given by  $$\na^+_{K,F}(t) =-2t+4-2t^{-1} \quad \text{and} \quad  \na^-_{K,F}(t)=-t+6-t^{-1}.$$ Neither of these polynomials satisfies the Fox-Milnor condition, so Theorem \ref{thm:main-2} implies that $6.87857$ is not slice. In fact, applying a single crossing change to $6.87857$ gives a slice knot, and thus it has slice genus one.
\end{example}

\section{Realization theorems} \label{sec-4}
In this section, we provide necessary and sufficient conditions for a pair of matrices $(V^+,V^-)$ to occur as the Seifert matrices of some almost classical knot, and we use it to show that any integral polynomial $\De(t)$ with $\De(1)=1$ occurs as the Alexander polynomial of some almost classical knot.
The arguments are constructive and assume some familiarity with the basics of virtual knot theory as found, for example, in \cite{Kauffman-1999, Kamada, Kauffman-2012}. We will also make essential use of virtual disk-band surfaces, which are prototypes of the more general notion of virtual Seifert surfaces defined and studied in \cite{Chrisman-2017}.

\subsection{Virtual Disk-Band Surfaces} \label{sec_virt_band} 
In this subsection, we introduce virtual disk-band surfaces, defined as follows.

\begin{definition} \label{defn:vbs}
A \emph{virtual disk-band surface} consists of a finite union of disjoint disks $D_1,\ldots, D_n$ in $\RR^2$, with a finite collection of bands $B_1,\ldots, B_m$ in $\RR^2 \sm \Int(D_1 \cup \cdots \cup D_n)$, connecting the disks. The bands may have (classical) twists, and in any region of the plane, at most two bands intersect. Each such band crossing is either classical or virtual as in Figure \ref{fig_virt_band_cross}. 
\end{definition}

\begin{figure}[ht]
\centering
\includegraphics[scale=0.80]{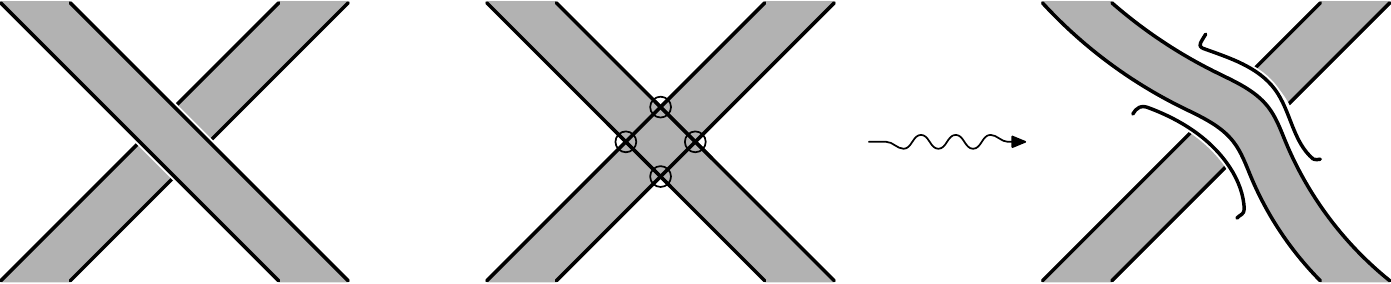} 
\caption{A classical band crossing (left), a virtual band crossing (middle), and a 1-handle attached at a virtual band crossing (right).}
\label{fig_virt_band_cross}
\end{figure}


We will work exclusively with orientable virtual disk-band surfaces here, but we remark that Definition \ref{defn:vbs} can also be used to describe non-orientable spanning surfaces. Taking the boundary of a virtual disk-band surface gives a virtual knot diagram which is almost classical, and which can be seen as follows. 
 
View the virtual disk-band surface $S$ in $S^2 \times I$ and  construct a higher genus surface by attaching one-handles to $S^2$ at each virtual band crossing as in Figure \ref{fig_virt_band_cross}. The 1-handles are attached to allow one band to pass along the 1-handle over the other band. The result is a disk-band surface in $\Si \times I,$ where $\Si$ has genus equal to the number of virtual band crossings of $S$. The boundary of this disk-band surface is an almost classical knot in $\Si \times I$ representing $K$. The next lemma establishes the converse.

\begin{lemma} \label{disk-band-lem}
Suppose $K$ is an almost classical knot or link, realized as a knot or a link in a thickened surface $\Si \times I$. Then there exists a  virtual disk-band surface whose boundary is a virtual knot diagram for $K$. 
\end{lemma}

\begin{proof} (sketch) The knot can be represented as a knot in $\Si \times I$ with Seifert surface $F$. Just as with any surface with boundary, $F$ can be decomposed as a union of disks and bands in $\Si \times I$. Assuming $\Si$ has genus $g$, it can be realized as the identification space of the $4g$-gon $P$, drawn in the plane, under the usual identification of its sides. 
Under further isotopy, we can arrange that the images of the disks of $F$ are pairwise disjoint and lie in the interior of $P$, and that the images of the bands are disjoint from the disks and that at most two bands meet in any region. We can also arrange that the bands  meet the boundary of the $4g$-gon $P$ only along its edges and not at any of its vertices. It is now a simple matter to draw the associated virtual disk-band surface in the plane by extending the bands of $F$ outside the $4g$-gon $P$ and introducing virtual crossings whenever two bands cross outside $P$.  The result is a virtual disk-band surface with boundary a virtual knot diagram for $K.$
\end{proof}

\begin{figure}[ht]
\def\svgwidth{0.55\textwidth}  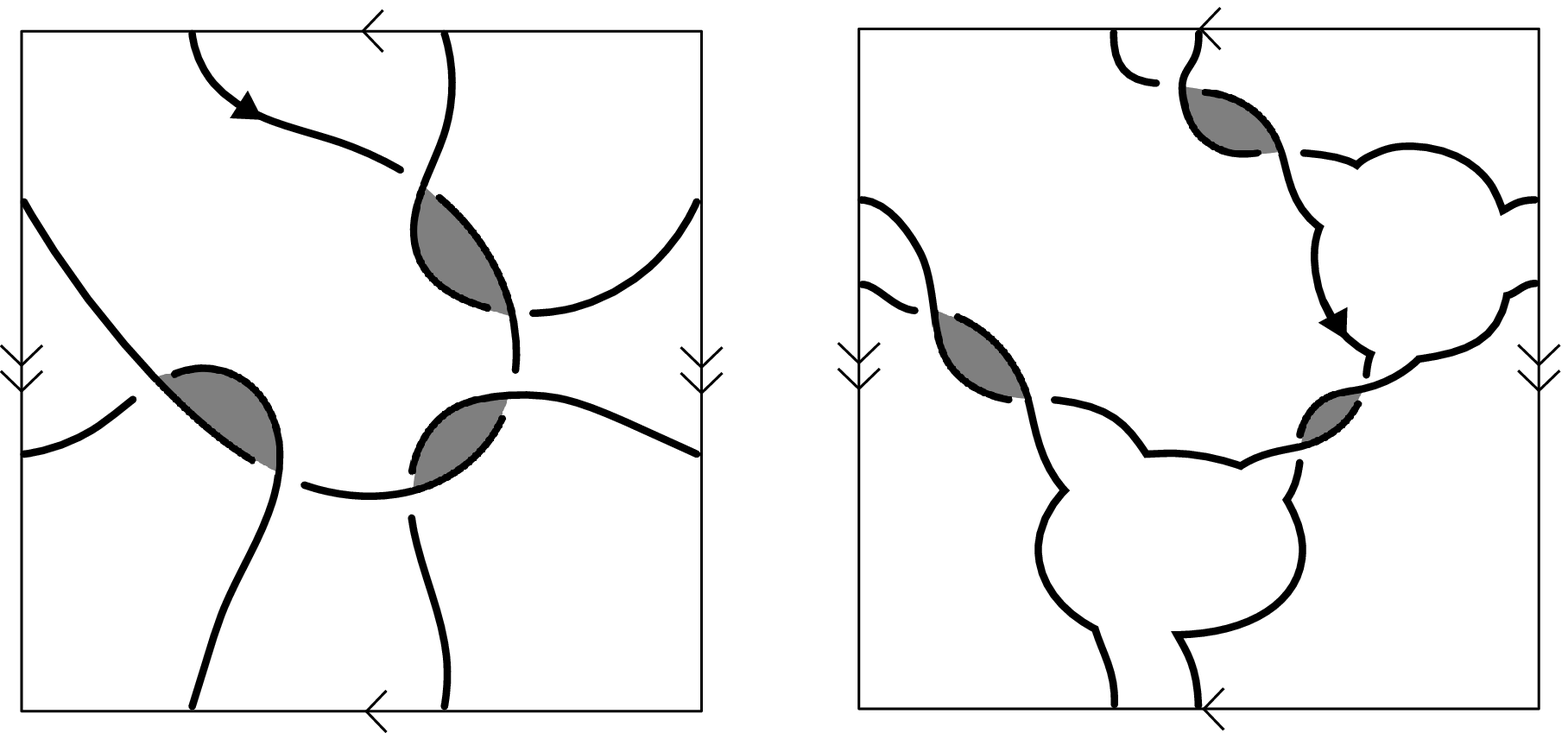 \quad 
\def\svgwidth{0.65\textwidth} {\small 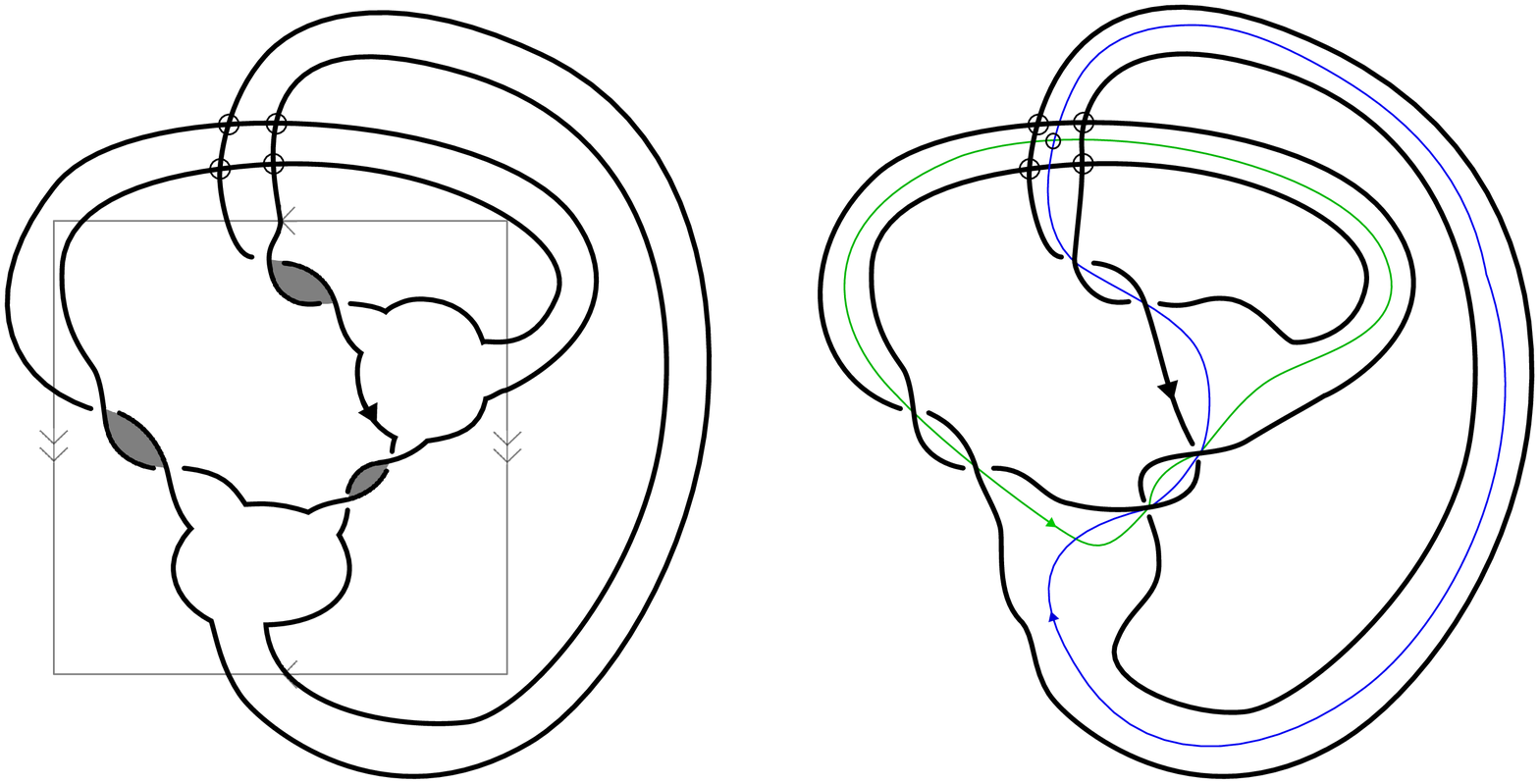}
  \caption{The Seifert surface for the almost classical knot 6.90228 in the thickened torus transforming   into a virtual disk-band surface.}
\label{fig-disk-band-surface}
\end{figure}
Figure \ref{fig-disk-band-surface} shows the evolution of  a Seifert surface in $\Si \times I$ to a virtual disk-band surface for the almost classical knot 6.90228.
An orientation of $K$ induces an orientation on the Seifert surface $F$ and the virtual disk-band surface $S$, and this determines the directions of the positive and negative push-offs.
Here we follow the convention used in Subsection \ref{subsec-computations}, so a small left-handed meridian pierces $F$ and $S$ from the negative side and exits from the positive side. For example, in Figure \ref{fig-disk-band-surface}, over the shaded regions of $F$ and $S$, the positive push-off is above the page, and elsewhere it is below the page.

The next lemma follows from a straightforward comparison between the conventions for computing linking numbers of curves in $\Si \times I$ relative to $\Si \times \{1\}$ and the conventions for computing virtual linking numbers. The details are left to the reader.

\begin{lemma} \label{vlk-lem}
Suppose $K$ is an almost classical knot, realized as a knot in a thickened surface $\Si \times I$
with Seifert surface $F$.
Suppose $a_1,\ldots, a_n$ are simple closed curves giving a basis for $H_1(F)$, and let $\al_1, \ldots, \al_n$ be their images  on the virtual disk-band surface $S$ associated to $F$ as in Lemma \ref{disk-band-lem}. Then   
\begin{equation}
lk(a_i^\pm, a_j) = vlk(\al_i^\pm, \al_j).
\end{equation}
\end{lemma}

As a consequence of Lemma \ref{vlk-lem}, the Seifert matrices $(V^+,V^-)$ associated to an almost classical knot $K$ with Seifert surface $F$ and a choice of basis for $H_1(F)$ can be computed entirely in terms of the virtual linking numbers of the corresponding curves on the virtual disk-band surface $S$ associated to $F$ as in Lemma \ref{disk-band-lem}. 

\begin{example}  Consider the virtual disk-band surfaces shown in Figure \ref{fig-disk-band-surface}. Using the basis $\{\al,\be\}$ in the last frame, one can compute the Seifert pair associated to this surface to be 
\[
V^+=\left[\begin{array}{cc} 2 & 1\\ 1 & 2  \end{array} \right]
\quad \text{and } \quad
V^-=\left[\begin{array}{cc} 2 & 2 \\ 0 & 2  \end{array} \right].
\]
\end{example}

Just as in the classical case, the calculation of the Seifert matrices is especially simple when the virtual disk-band surface $S$ consists of only one disk with bands attached, and this can always be arranged by isotopy. In that case, the cores of the bands are simple closed curves on $S$ which give a natural choice of basis for $H_1(F)$, and each band has an even number of half twists, and as such can be drawn without twists by using positive or negative kinks (cf. Figure \ref{band-twists}).
Virtual disk-band surfaces of that form are called \emph{virtual band surfaces} and they are especially useful in establishing the realization theorems for Seifert pairs in the next two subsections.

\subsection{Realization of Seifert pairs} \label{subsec:proof_real} 
In Subsection \ref{subsec:SfSm}, we observed that any pair $(V^+,V^-)$ of Seifert matrices for an almost classical knot $K$ satisfies $V^- -V^+$ is skew-symmetric and $\det(V^- -V^+)=1.$ In this subsection, we establish the converse result. We begin with the following definition.

\begin{definition} \label{defn:Seifert-pair}
A \emph{Seifert pair} is a pair $(V^+,V^-)$ of integral square $2g \times 2g$ matrices such that  $V^- -V^+$ is skew-symmetric and $\det(V^- -V^+)=1.$
Equivalence of Seifert pairs is given by simultaneous unimodular congruence. 
\end{definition}

\begin{remark} \label{remark-useful}
If $(V^+,V^-)$ is a Seifert pair, then a standard argument (see \cite[A1]{Burde-Zieschang-Heusener}, for instance) shows that $V^- - V^+$ is unimodular congruent to $H^{\oplus g}$, a block sum of $g$ copies of the $2\times 2$ matrix:
\begin{equation} \label{block-H}
H=\left[\begin{array}{cc} 0 & 1 \\ -1 & 0 \end{array} \right].
\end{equation}
\end{remark}

The next theorem is the main result in this section, and we will prove it by constructing a virtual band surface realizing any given Seifert pair $(V^+,V^-)$.

\begin{theorem} \label{thm_realize}
A pair $(V^+, V^-)$ of integral square $2g \times 2g$ matrices represents the pair $(\th^+,\th^-)$ of Seifert forms associated to an almost classical knot $K$ with Seifert surface $F$ if and only if $V^- -V^+$ is skew-symmetric and $\det(V^- -V^+)=1.$
\end{theorem}

\begin{proof}
Suppose $(V^+, V^-)$ is a Seifert pair (cf.~Definition \ref{defn:Seifert-pair}). 
Our goal is to construct an almost classical knot $K$ and Seifert surface $F$ whose Seifert matrices equal $(V^+, V^-)$. Notice that it is enough to prove this up to simultaneous unimodular congruence, because if a Seifert pair $(V^+,V^-)$ is realized by some almost classical knot $K$ and Seifert surface $F$ and choice of basis for $H_1(F)$, then any pair of matrices simultaneously unimodular congruent to $(V^+,V^-)$ can be realized by making a change of basis for $H_1(F)$.

Suppose then that $(V^+,V^-)$ is a pair of integral square $2g \times 2g$ matrices such that $V^- -V^+$ is skew-symmetric and $\det(V^- -V^+)=1.$ By Remark \ref{remark-useful},
we have that $V^- - V^+$ is unimodular congruent to the block sum $H^{\oplus g}$.
Therefore, it is sufficient to prove the statement under the assumption that $V^--V^+ = H^{\oplus g}$. In this case, $V^-$ is determined by $V^+$, thus we will show that any integral square $2g \times 2g$ matrix $V^+$ can be realized as the Seifert matrix of a virtual band surface.

\begin{figure}[ht]
\centering
\includegraphics[scale=1.20]{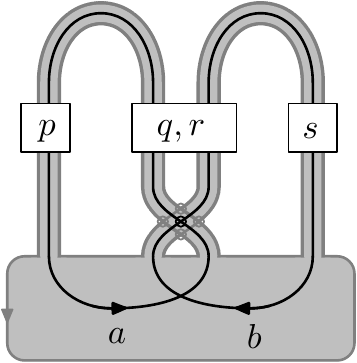} 
\caption{A virtual band surface with $g=1$ and kinks and virtual half twists as indicated.}
\label{base-case}
\end{figure}

The proof is by induction on $g$. 
For $g=1$, consider the virtual band surface given in Figure \ref{base-case}. Here $\{a,b\}$ denotes an ordered basis for $H_1(F)$, as discussed in Section \ref{sec_virt_band}, and an easy exercise shows that the intersection form $\langle\cdot,\cdot\rangle_F$ is represented by the matrix $H$ with respect to this basis. The labelled boxes in Figure \ref{base-case} indicate the number and types of kinks and virtual half twists that need to be inserted.

\begin{figure}[ht]
\centering
\includegraphics[scale=0.750]{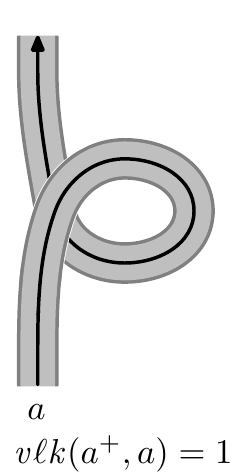} \quad
\includegraphics[scale=0.750]{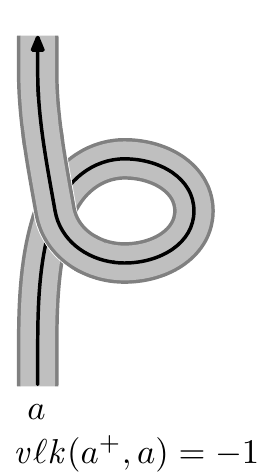} \qquad
\includegraphics[scale=0.750]{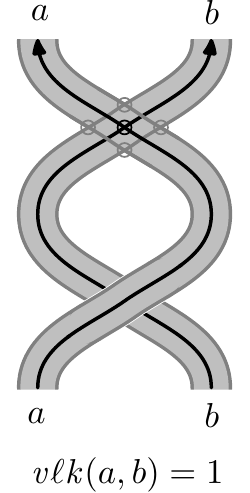} \quad
\includegraphics[scale=0.750]{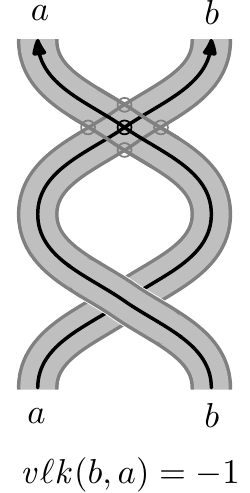} \quad
\includegraphics[scale=0.750]{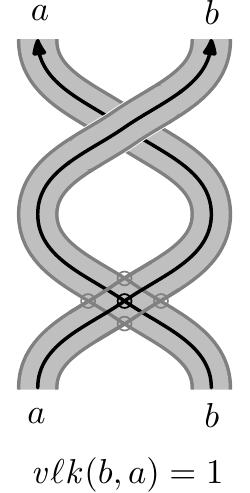}\quad
\includegraphics[scale=0.750]{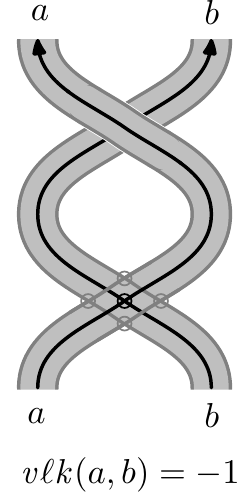} 
\caption{Positive and negative kinks (on left) and the four types of virtual half-twisted bands (on right).}
\label{band-twists}
\end{figure}

For instance, suppose $$V^+ =\begin{bmatrix} p & q \\ r & s \end{bmatrix}.$$
Then to realize $V^+$ as the virtual linking matrix, insert $|p|$ kinks on the first band and $|s|$ kinks on the second, using right handed kinks when $p$ or $s$ is positive and left-handed kinks when they are negative (see two kinks on left in Figure \ref{band-twists}). Also insert $|q|$ virtual half twists where $a$ crosses over $b$ and $|r|$ virtual half twists where $b$ crosses over $a$, again using right-handed virtual twists when $q$ or $r$ is positive and left-handed virtual twists when $q$ or $r$ is negative (see four virtual half twists on right of Figure \ref{band-twists}).
With these choices, it is not difficult to verify that 
\begin{eqnarray*}
&\vlk(a^+,a) = p, \quad \vlk(a^+,b)=q,&\\
&\vlk(b^+,a)=r, \quad \vlk(b^+,b)=s.&
\end{eqnarray*}
Thus, this virtual band surface has Seifert matrix $V^+$ with respect to the basis $\{a,b\}.$

To proceed with the proof, we argue by induction. Suppose $(V^+,V^-)$ is a pair of integral square $(2g+2) \times (2g+2)$ matrices such that $V^- -V^+$ is skew-symmetric and $\det(V^- -V^+)=1.$ Under simultaneous unimodular congruence, we can arrange that $V^- -V^+ = H^{\oplus g+1}$. As before, it is enough to construct a virtual band surface $F$ realizing $V^+$. Write $V^+ = (v_{ij})$, where $1 \leq i,j \leq 2g+2$, and set $V_1 = (v_{ij})_{1\leq i,j \leq 2g}$ and $V_2 =(v_{ij})_{2g+1\leq i,j \leq 2g+2}$. Thus, $V_1$ and $V_2$ are the diagonal block submatrices of $V^+$ of sizes $2g \times 2g$ and $2 \times 2$, respectively, such that: 
$$V^+=
\left[\begin{array}{cc}
V_1 & * \\ 
* & V_2  
\end{array}\right].$$
By induction, we have virtual band surfaces realizing $V_1$ and $V_2$, and Figure \ref{band-induct} depicts a virtual band surface obtained from combining these two surfaces. In the figure, bands are represented by lines. The large boxes labelled $V_1$ and $V_2$ indicate what can be arranged by induction. Furthermore, it is straightforward to see that, with the indicated basis $\{a_1,\ldots, a_{2g+2} \}$ for $H_1(F)$, the intersection form $\langle \cdot,\cdot\rangle_F$ is a block sum of $g+1$ copies of the matrix $H$. Thus $V^- - V^+ = H^{\oplus (g+1)}$.

To complete the proof, one must insert virtual twists into the small unlabelled boxes in Figure \ref{band-induct}, and these are chosen to achieve the desired linking of the last two bands around the first $2g$ bands. This will involve a combination of left or right virtual half twists, as shown in Figure \ref{band-twists}, and here we no longer assume $a_j$ is oriented upwards. To be very specific, for the box involving the bands $a_i$ and $a_j$, where $1 \leq i \leq 2g$ and $2g+1 \leq j \leq 2g+2$, we insert the 2-strand virtual braid so that $\vlk(a_i,a_j) = v_{ij}$ and $\vlk(a_j,a_i) = v_{ji}$. The pair of Seifert matrices for the resulting virtual band surface with respect to the basis $\{a_1,\ldots, a_{2g+2} \}$ is now equal to the given matrices $(V^+, V^-),$ and
this completes the induction and finishes the argument. 
\end{proof}

\begin{figure}[ht]
\centering
\includegraphics[scale=1.30]{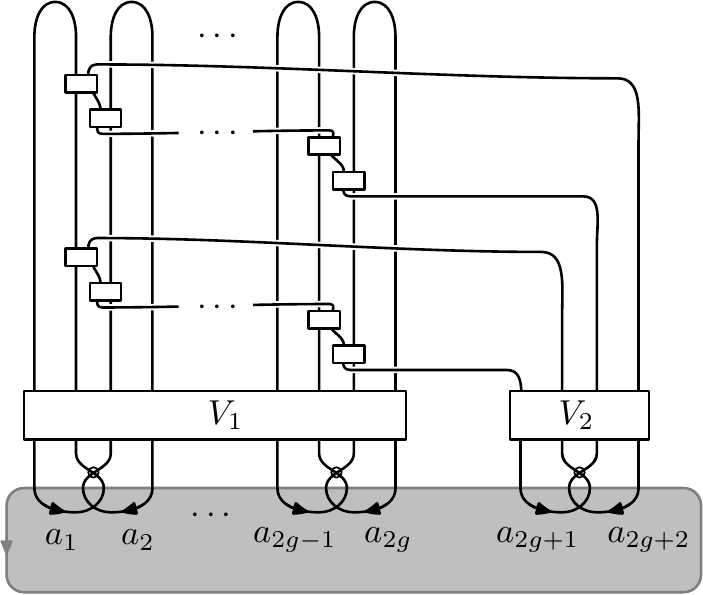} 
\caption{A virtual band surface for the inductive step.}
\label{band-induct}
\end{figure}

\subsection{Realization of Alexander polynomials}  \label{subsec:proof_alex}
For any almost classical knot, the Alexander polynomial satisfies $\De_K(1)=1$, and in this subsection, we establish the converse result, that any integral polynomial $\De(t)$ satisfying $\De_K(1)=1$ occurs as the Alexander polynomial of some almost classical knot.

The proof uses Theorem \ref{thm_realize}, and we begin by defining the \emph{algebraic} Alexander polynomial associated to an integral square $2g \times 2g$ matrix $A$ defined as 
$$\De_A(t) = \det\left(A(t-1) -I\right).$$ 
Setting $t=1$, we see that this polynomial satisfies $\De_A(1)=1$. The following lemma is useful and proved using companion matrices.
The details are left to the reader.

\begin{lemma} \label{alg_real}
For any integral polynomial $\De(t)$ satisfying $\De(1)=1$, there exists an integral square $2g \times 2g$ matrix $A$ 
with $\De_A(t) \doteq \De(t).$
\end{lemma}


The next result follows by combining Theorem \ref{thm_realize} and
Lemma \ref{alg_real}.

\begin{theorem} \label{thm:real_alex}
For any integral polynomial $\De(t)$ satisfying $\De(1)=1$, there exists an almost classical knot $K$ with $\De_K(t) \doteq \De(t).$
\end{theorem}

\begin{proof}
Let 
$J = H^{\oplus g}$ be the block sum of $g$ copies of the $2 \times 2$ matrix $H$ from equation \eqref{block-H}. 
Thus, $J$ is an integral square matrix with $\det(J)=1$ and $J^2 =-I,$ where $I$ denotes the $2g \times 2g$ identity matrix.

Suppose $A$ is an integral square $2g \times 2g$ matrix.
By Theorem \ref{thm_realize}, we can find an almost classical knot $K$ whose Seifert matrices $(V^+,V^-)$ satisfy $V^- - V^+ = J$ and $V^- = -J A.$

Therefore, 
\begin{eqnarray*}
\De_K(t) &=& \det\left(t V^- - V^+\right) \\
&=& \det\left((t-1) V^- + J\right) \\
&=& \det\left(J\right) \det\left((t-1) V^- + J\right)\\
&=& \det\left((t-1) J V^- -I\right) \\
&=& \det\left((t-1) A -I\right) \doteq \De_A(t).
\end{eqnarray*}

Thus,  there exists an almost classical knot $K$ with
$\De_K(t) \doteq \De_A(t)$.
Lemma \ref{alg_real} implies that any integral polynomial $\De(t)$ satisfying
$\De(1)=1$ can be realized as $\De_A(t)$ for some integral square  $2g \times 2g$ matrix $A$, and the above argument shows the same is true for the Alexander polynomials of almost classical knots. 
This completes the proof.
\end{proof}

\begin{remark} In fact, in Theorem \ref{thm:real_null2} we will see that any integral polynomial $\De(t)$ with $\De(1)=1$ can be realized as the Alexander polynomial of a slice and even ribbon 
almost classical knot.
\end{remark}

\subsection{Realization for null-concordant Seifert pairs} \label{subsec-real}
The notion of \emph{algebraic concordance} for almost classical knots is defined  in terms of null-concordant Seifert pairs, which we introduce next.
\begin{definition}
A pair $(V^+,V^-)$ of integral square $2g \times 2g$ matrices is called \emph{null-concordant}  if $V^+$ and $V^-$ are simultaneously unimodular congruent to matrices in block form:
$$\begin{pmatrix} \mathbf{0} & P^\pm \\ Q^\pm & R^\pm \end{pmatrix},$$
where  $P^\pm, Q^\pm, R^\pm$ are integral $g \times g$ matrices. 
\end{definition}

The proof of Theorem \ref{thm:main-2} shows that, if $K$ is an almost classical slice knot and $F$ is any Seifert surface for $K$, then the associated Seifert pair $(V^+, V^-)$is null-concordant. More generally, if $K$ is any almost classical knot with Seifert surface $F$ such that the Seifert pair $(V^+, V^-)$ is null-concordant, then Theorems \ref{thm:main} and \ref{thm:main-2} continue to hold. This is summarized in the following proposition.

\begin{proposition} \label{alg-slice}
Suppose $K$ is an almost classical knot with Seifert surface $F$ such that the associated Seifert pair $(V^+, V^-)$ is null-concordant. Then for any $\om \neq 1$ unit complex number such that $\na^\pm_{K,F}(\om)\neq 0,$ we have $\widehat{\si}^\pm_\om(K)=0.$
Further, there exist polynomials
$f^\pm(t) \in \ZZ[t]$ such that $\na^\pm_{K,F}(t) = f^\pm(t) f^\pm(t^{-1})$.
\end{proposition}

The next theorem shows that every null-concordant Seifert pair can be realized by an almost classical knot which is slice.
\color{black} 

\begin{theorem} \label{thm:real_null}
Suppose $(V^+,V^-)$ is a pair of integral square $2g \times 2g$ matrices such that $V^- - V^+$ is skew-symmetric and $\det(V^- - V^+)=1$. Then $(V^+,V^-)$ is null-concordant if and only it occurs as the Seifert pair of an almost classical knot $K$ which is slice. 
\end{theorem}

\begin{proof}
Suppose $(V^+,V^-)$ is the Seifert pair associated to an almost classical
knot $K$ with Seifert surface $F$. If $K$ is slice, then as in the proof of Theorem \ref{thm:main-2}, it follows that $(V^+,V^-)$ is null-concordant. This proves the theorem in one direction.

To prove the converse, suppose $(V^+,V^-)$ is a null-concordant pair of integral square $2g \times 2g$ matrices
with $V^--V^+$ skew-symmetric and $\det(V^--V^+)=1$. We will construct an almost classical knot $K$ which is slice and admitting $(V^+,V^-)$ as a Seifert pair. In fact, it will follow from the construction that $K$ is actually ribbon.
 
Since $V^- - V^+$ is skew-symmetric and unimodular, the bilinear form $\be$ on the free module $U = \ZZ^{2g}$ 
defined by $\be(x,y) =  x \left( V^- - V^+ \right) y^\uptau$ for $x,y \in U$ is a non-degenerate skew form. 
For a submodule $L \subset U$, we define $$L^\perp = \{ y \in U \mid \beta(x,y)=0 \text{ for all $x \in L$} \}.$$ A submodule $L$ is called \emph{isotropic} if $L \subset L^\perp$, and
a basis $\{a_1,\ldots, a_g, b_1\ldots, b_g\}$ for $U$ is said to be \emph{symplectic} if $\be(a_i,b_j) = \de_{ij}$ and $\be(a_i,a_j)= 0=\be(b_i,b_j)$ for all $1 \leq i,j \leq g.$

We claim that there is a symplectic basis $\{a_1,\ldots,a_{g}, b_1,\ldots, b_g\}$ for $U$ such that the restrictions of $V^+$ and $V^-$ to the submodule generated by $a_1,\ldots,a_{g}$ are both trivial.

The claim is proved by induction on $g$. For $g=1$, null-concordance implies there is a basis $\{a,b\}$ such that $V^+$ and $V^-$ both vanish on $a.$ Then $a$ is isotropic for $\be$. Further, since $\det(V^- - V^+) =1,$ replacing $b$ with $-b$ if necessary, we can arrange that $\{a,b\}$ is a symplectic basis for $\be$. 

Now suppose it has been proved for free abelian groups of rank $2g-2$ and consider $U= \ZZ^{2g}.$ Null-concordance implies there is a primitive submodule $L \subset U$ of rank $g$ on which $V^+$ and $V^-$ vanish. Clearly $L$ is isotropic, in fact one can easily see that $L = L^\perp.$ Let $\{a_1, \ldots, a_g\}$  be a basis for $L$ and $M$ be the submodule  generated by $\{a_2, \ldots ,a_g\}.$ Then $L^\perp \subsetneq M^\perp,$ and the basis $\{a_1, \ldots, a_g\}$ can be extended to a basis for $M^\perp$ by adding one element $b_1 \in M^\perp$, which can be chosen so that $\be(a_1, b_1) = 1$ (since $V^--V^+$ is unimodular). 
Notice that $\be$ restricts to a non-degenerate form on the submodule generated by $\{a_1,b_1\}.$ Set $W = \left({\rm span}\{a_1,b_1\}\right)^\perp.$ Then $\be$ restricts to a non-degenerate skew form on $W$, which has rank $2g-2$, and $M$ is isotropic with respect to the restriction, and we apply induction.

Since it is enough to prove the theorem up to simultaneous unimodular congruence, by the claim we can arrange that  
\begin{equation} \label{eqn-null-conc}
V^\pm= \begin{bmatrix} \mathbf{0} & Q^\pm \\ R^\pm & S^\pm \end{bmatrix} \quad \text{ and } \quad
V^--V^+ = \begin{bmatrix} \mathbf{0} & I \\ -I & \mathbf{0} \end{bmatrix}.
\end{equation}
For such pairs, $V^-$ is determined by $V^+$. The proof of Theorem \ref{thm_realize} given in
Subsection \ref{subsec-real}
already shows how to construct a virtual band surface with prescribed Seifert matrix $V^+$, and we will explain how the assumption of null-concordance results in this surface bounding an almost classical knot that is slice. As a warm-up, we give the proof in case $g=1$.

For instance, suppose $$V^+ =\begin{bmatrix} 0 & q \\ r & s \end{bmatrix}$$ and $V^-=V^+ +H$. Clearly the Seifert pair $(V^+,V^-)$ is null-concordant, and the virtual band surface depicted in Figure \ref{base-case} with $p=0$ realizes the Seifert pair $(V^+, V^-)$.

\begin{figure}[ht]
\centering
\includegraphics[scale=1.05]{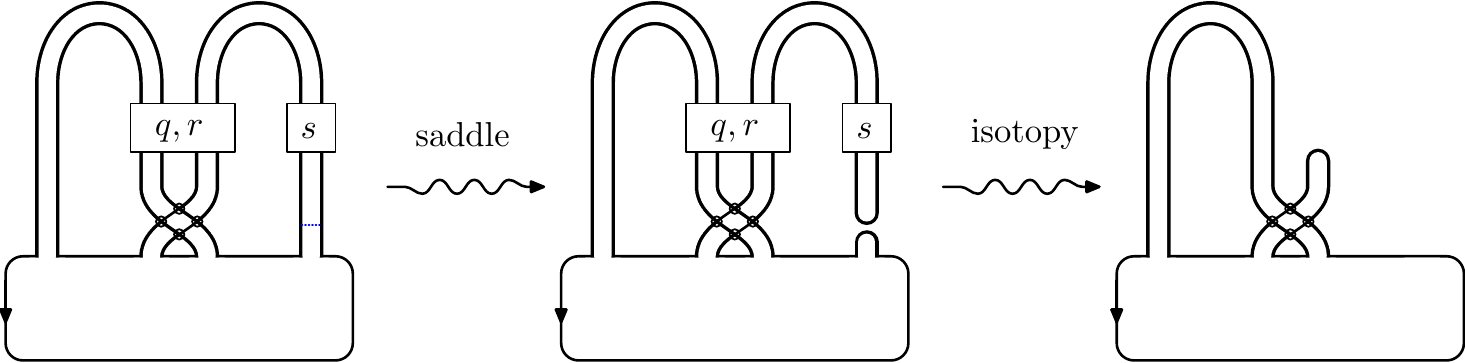}
\caption{A movie for the almost classical knot bounded by the virtual band surface.}
\label{ribbon-movie}
\end{figure}

We claim that the almost classical knot bounded by this surface is slice, and the movie  in Figure \ref{ribbon-movie} gives the concordance to the unknot. In fact, after performing one saddle at the base of the second band on the right, the resulting link is isotopic by a sequence of Reidemeister II move and virtual Reidemeister II moves to the trivial link with two components. Performing one death results in the unknot, and this produces a ribbon concordance 
to the unknot. This shows that the original knot is ribbon and completes the proof in the case $g=1$.

\begin{figure}[htb]
\includegraphics[scale=1.40]{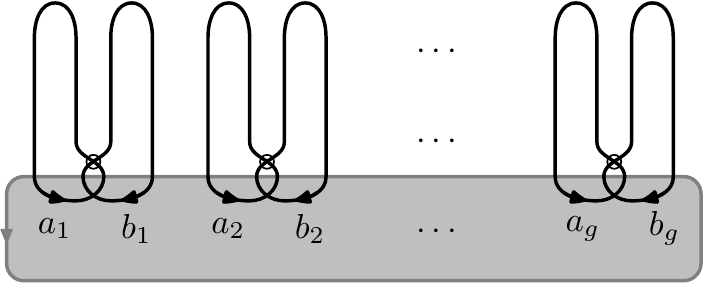}
\caption{Initial placement of arcs to realize a null-concordant pair $(V^+,V^-)$} \label{fig_slice_base}
\end{figure}

A virtual band surface can be constructed realizing any such null-concordant pair $(V^+,V^-)$ as in the proof of Theorem \ref{thm_realize}. Here, we label the bands according to Figure \ref{fig_slice_base}, where bands are drawn as lines. It depicts the virtual band surface with $2g$ bands whose boundary is the trivial knot. We can modify this surface to realize any Seifert pair $(V^+,V^-)$ satisfying \eqref{eqn-null-conc} by only altering the second set of bands $\{b_1,\ldots, b_g\}$ but leaving the first set of bands $\{a_1,\ldots, a_g\}$ trivial. 
Redrawing the second set of bands $\{b_1,\ldots, b_g\}$ to allow them to virtually link themselves and the first set of bands $\{a_1,\ldots a_g\}$ arbitrarily, and adding kinks to the second set of bands as needed, we can realize any null-concordant Seifert pair $(V^+,V^-)$ satisfying \eqref{eqn-null-conc}.

It remains to show the resulting virtual knot bounding this surface is slice, and the following lemma is useful.
\begin{lemma}
For any virtual band surface as in Figure \ref{band-induct}, if a saddle move is performed to one of its bands, then there is an isotopy of the resulting link which contracts the two surgered ends to the disk by a sequence of Reidemeister II and virtual Reidemeister II moves (as in Figure \ref{ribbon-movie}).
\end{lemma}

Performing $g$ saddle moves to each of bands $\{b_1,\ldots, b_g\}$ and applying the lemma, we obtain a $g+1$ component virtual unlink.   This describes a ribbon concordance to the unknot, and  it completes the proof of the theorem.  
\end{proof}

We close this section with the next result, which is in marked contrast to the situation for classical knots.

\begin{theorem} \label{thm:real_null2}
For any integral polynomial $\De(t)$ satisfying $\De(1)=1$, there exists a ribbon almost classical knot $K$ with $\De_K(t) \doteq \De(t).$
\end{theorem}

\begin{proof}
Since $\De(1)=1$, we can write $\De(t) = 1+ a_1(t-1) + \cdots + a_n(t-1)^n$
for some $a_1,\ldots, a_n \in \ZZ$ with $a_n \neq 0$.
Let $$J=\begin{bmatrix}
\,\mathbf{0}  &I \, \\
\,-I & \mathbf{0} \, 
\end{bmatrix},$$
where $I$ is the $n \times n$ identity matrix.
Consider the matrices $(V^+,V^-)$ of size $2n \times 2n$ written in block form:
$$
V^- = \begin{bmatrix}
\,\mathbf{0}  &A \, \\
\,\mathbf{0} & * \, 
\end{bmatrix}
\quad \text{ and } \quad 
V^+ = \begin{bmatrix}
\,\mathbf{0}  &A-I \, \\
\,I & * \, 
\end{bmatrix}.$$
Here, $*$ is an arbitrary
$n\times n$ integral matrix, and 
$$A = \begin{bmatrix}
\,0  &   &\cdots & -a_{n} \\
\,1 & 0 & & -a_{n-1} \\
\,\vdots &    \ddots & & \vdots \\
\,0&  \cdots &1  &-a_1
\end{bmatrix}$$
is the $n\times n$ companion matrix to $\De(t)$.

Obviously, $V^- -V^+ = J$, and the Seifert pair $(V^+,V^-)$ is evidently null-concordant. 
By Theorem \ref{thm:real_null}, there is an almost classical knot which is ribbon and which realizes the pair $(V^+,V^-)$. Furthermore, this knot has Alexander polynomial 
$$\De_K(t) \doteq \det\left(tV^--V^+\right) = \det\left((t-1)V^- +J\right)= \det\left((t-1)A +I\right) = \De(t).$$
\end{proof}

\section{Applications} \label{sec-5}
In this section, we apply Theorems \ref{thm:main} and \ref{thm:main-2} to the problem of determining the slice genus for every almost classical knot up to six crossings, and this step relies on computations of the Seifert matrices (see Table \ref{table-3}) and Turaev's graded genus (see \cite{Boden-Chrisman-Gaudreau-2017, Boden-Chrisman-Gaudreau-2017t}).
We show that the directed Alexander polynomials satisfy a skein relation and study  how the knot signature behaves under crossing changes. 

\subsection{Computations} \label{subsec-computations}

In this subection, we explain how to compute the Seifert matrices and signatures for almost classical knots. To start off, we take a homologically trivial knot in a thickened surface and apply Seifert's algorithm to construct a Seifert surface. This method was used to produce the Seifert matrices $V^\pm$ for almost classical knots with up to 6 crossings in Table \ref{table-3}, namely we applied Seifert's algorithm to the knots in surfaces depicted in Figure \ref{ACknot-diagrams} and  computed linking pairings.

In all computations, the top of the Carter surface is facing up, and the orientation of the Seifert surface determines the direction of the positive and negative push-offs. We follow the conventions of \cite{Boden-Gaudreau-Harper-2016} in this aspect, which means that we parameterize a regular neighborhood $N(F)$ of the Seifert surface as $F \times [-1,1]$ such that a small oriented meridian of $K$ enters $N(F)$ at $F \times \{ -1\}$ and exits at $F \times \{1\}.$ Our convention is that meridians are left-handed.

It is helpful to notice that the oriented smoothing of a knot in a thickened surface depends only on the underlying flat knot, which is the virtual knot up to crossing changes, and so we organize the almost classical knots into families according to their underlying flat knot. 
Notice further that the Seifert surfaces of two knots with the same underlying flat knot differ only in the types of half-twisted bands that are attached. For positive crossings, the attached band has a left-handed twist, and for negative crossings, the band has a right-handed twist.

\begin{example} \label{5-family}
Consider the almost classical knots 5.2012, 5.2025, 5.2133, and 5.2433, which all have the same underlying flat knot. Each occurs as a knot in a Carter surface of genus two and admits a Seifert surface of genus two. Two of the knots in this family are slice (5.2025 and 5.2133), and we will mainly focus on the other two. 


In this family of 5-crossing almost classical knots, 5.2012 is obtained by performing a crossing change to 5.2433, and each of 5.2025 and 5.2133 are obtained by a crossing change to 5.2012. Since all the knots in the family are related by crossing changes, one can use Conway's method to perform some of these computations. We will return to this in Example \ref{5-Conway} below.

\begin{figure}[h]
\def\svgwidth{0.45\textwidth} 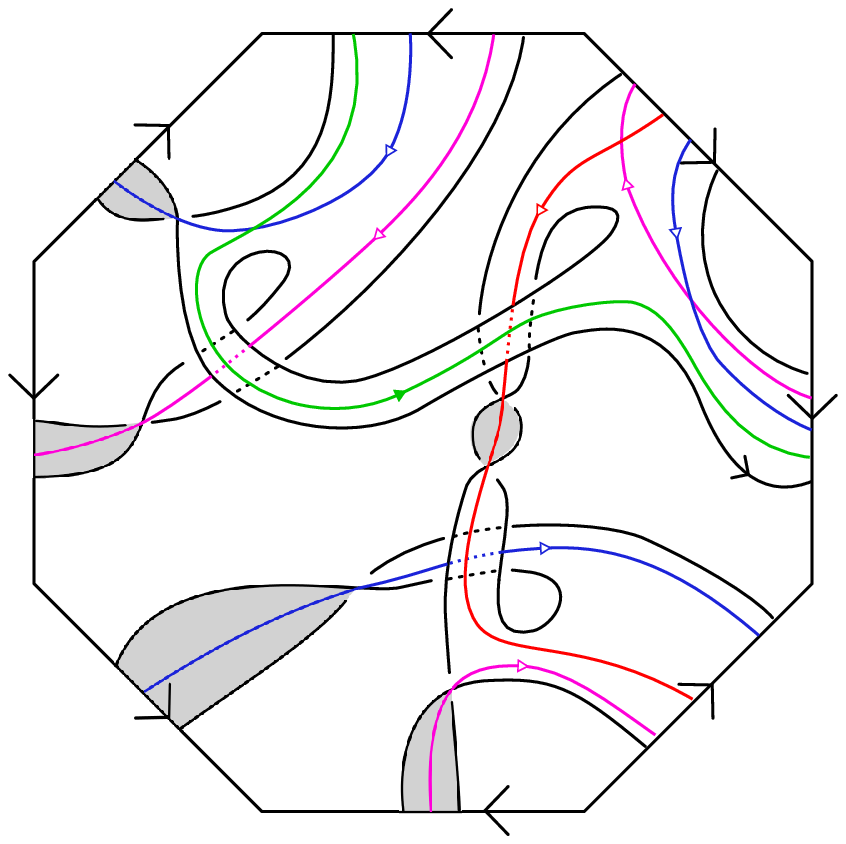
\qquad
\def\svgwidth{0.45\textwidth} 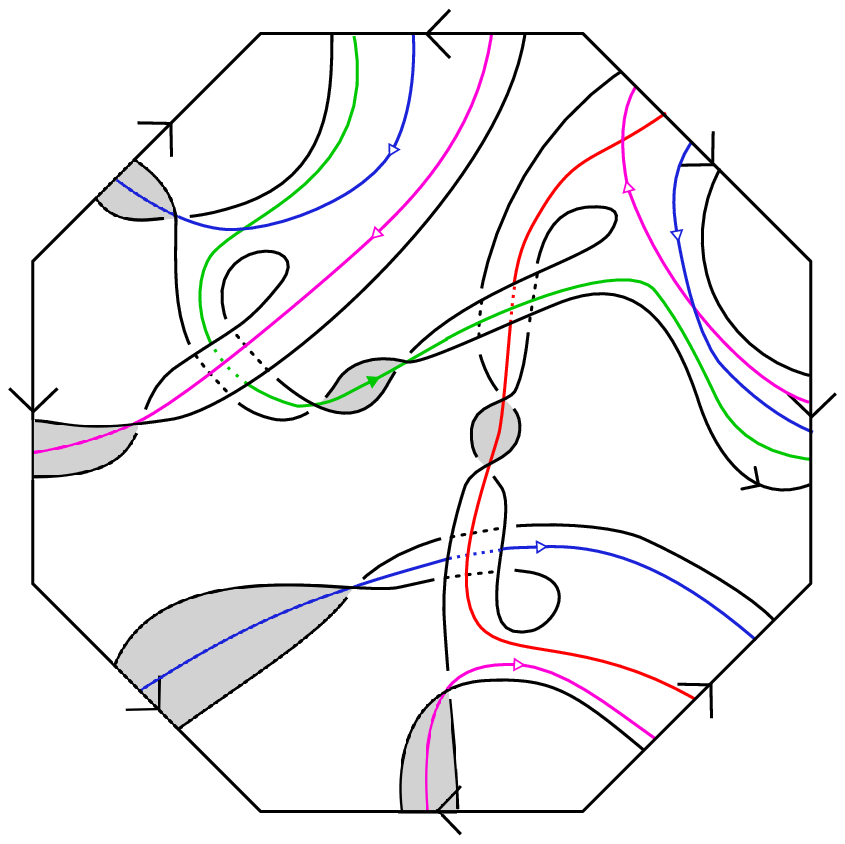

\caption{Seifert surfaces of the virtual knots 5.2012 (left) and 5.2433 (right).}
\label{5-2012}
\end{figure}

Consider the Seifert surfaces in Figure \ref{5-2012} for 5.2012 and 5.2433. Also depicted there is a basis $\{a,b,c,d\}$ of simple closed curves for $H_1(F)$.
The shaded part of $F$ is the positive side and the unshaded part is the negative side. Thus the positive push-offs of $a,b,c,d$ lie above $F$ along the shaded portion and below it along the unshaded portion. Recall further that if $J, K$ are two simple closed curves in $\Si \times I$ that do not intersect, then $\lk(J,K)$ is an algebraic count of only the crossings when $J$ goes over $K$ (this is because linking is taken in the relative homology of the pair $(\Si \times I, \Si \times \{1\}).$

For 5.2433, one obtains that 
$$V^+ = \begin{bmatrix} 1 & 0 & 0 & 1 \\ 0 & 1 & 0 & 0 \\ 1 & 0 & 1 & 0 \\ 0 & -1 & 0 & 1\end{bmatrix}
\quad \text{and} \quad
V^- = \begin{bmatrix} 1 & -1 & 0 & 1 \\ 1 & 1 & 1 & 0 \\ 1 & -1 & 1 & 1 \\ 0 & -1 & -1 & 1\end{bmatrix}.$$
Thus $\De_K(t) = t^2-2t+4-3t^{-1}+t^{-2}$ and $\si(K,F)=4.$
The directed Alexander polynomials are equal in this case and given by
$\na^\pm_{K,F}(t) = t^2-t+1-t^{-1}+t^{-2}$. This polynomial has roots  $e^{\pm 2 \pi i /5}, e^{\pm 2 \pi i 3/5}$ at the fifth roots of $-1$, and $\widehat{\si}^\pm_\om(K,F)$ takes values $\{0,2,4\}$.

For 5.2012, one obtains that 
$$V^+ =\begin{bmatrix}   0&0&-1&1 \\ 0&1&0&0 \\ 0&0&0&0 \\ 0&-1&0&1 \end{bmatrix}  
\quad \text{ and } \quad
V^- =\begin{bmatrix}  0&-1&-1&1 \\ 1&1&1&0 \\ 0&-1&0&1\\ 0&-1&-1&1 \end{bmatrix}.$$
Thus $\De_K(t) = t$ and $\si(K,F)=2.$
The directed Alexander polynomials are again equal and given by $\na^\pm_{K,F}(t) = t-1+t^{-1}$. This polynomial has roots $e^{\pm 2 \pi i /3}$ at the third roots of $-1$, and $\si^\pm_\om(K,F)$ takes values $\{0,2\}$. 
\end{example}

It should be noted that the Seifert matrices in Table \ref{table-3} are computed with just one choice of Seifert surface, and different choices would lead to possibly different signatures, $\om$-signatures, and directed Alexander polynomials. Passing to a second Seifert surface is useful in finding slice obstructions for a given almost classical knot, and we apply this method to show 5.2439, 6.72695, 6.77908, and 6.87548 are not slice in Subsection \ref{subsec:slice-o}.
\subsection{Conway's method} \label{subsec:Conway}
Conway developed methods for computing the Alexander polynomial and signatures of classical knots using crossing changes, and these methods carry over to give skein formulas for the directed Alexander polynomials and the signature of almost classical knots. (The Alexander polynomial $\De_K(t)$ of almost classical knots also satisfies a skein relation, see \cite[Theorem 7.11]{Boden-Gaudreau-Harper-2016} for more details.) We begin by reviewing the situation for classical knots.

Suppose $K_+$ and $K_-$ are classical knots that are identical except at one crossing, which is positive for $K_+$ and negative for $K_-$. Choosing Seifert surfaces for $K_+$ and $K_-$ that differ only in the half-twisted band at the crossing, and choose a basis for its first homology so the first generator passes through the crossing and all the other generators do not. It follows that the Seifert matrices $V(K_+)$ and $V(K_-)$ are the same except for the $(1,1)$ entry, in fact, we have
$$V(K_+) = \begin{bmatrix} a & u \\ v &W \end{bmatrix} \quad \text{ and } \quad V(K_-) = \begin{bmatrix} a+1 & u \\ v &W \end{bmatrix},$$
where $a$ is the self-linking of the first generator, $u$ and $v$ are row and column vectors, respectively, and $W$ represents the Seifert matrix for the two component link $K_0$ obtained from the $0$-smoothing of $K$.
Arguing as in \cite{Giller-1982}, it follows that the Alexander polynomials satisfy the skein relation
$$\De_{K_+}(t) - \De_{K_-}(t) = (t^{-1/2}-t^{1/2}) \De_{K_0}(t),$$
and the knot signatures of $K_+$ and $K_-$ differ by at most two. 
In fact, one can further show that  
\begin{equation} \label{eq-Conway-1}
\si(K_+) \leq \si(K_-) \leq \si(K_+) + 2,
\end{equation}
so either $\si(K_+) = \si(K_-)$ or $\si(K_+)+2 = \si(K_-)$. 

In general, for any non-singular symmetric real matrix, its signature modulo 4 is determined by the sign of its determinant. Applying this observation to the symmetrized Seifert matrix of a knot $K$, since it is a $2g \times 2g$ matrix with $\det(V+V^\uptau) = \De_K(-1)$, it follows that
\begin{equation} \label{eq-Conway-2}
\si(K) =
\begin{cases}
0 \mod 4 & \text{ if $\De_K(-1)>0$,} \\
2 \mod 4 & \text{ if $\De_K(-1)<0$.}
\end{cases}
\end{equation}
Thus one can determine $\si(K_+)$ from $\si(K_-)$ entirely by comparing the signs of $\De_{K_+}(-1)$ and $\De_{K_-}(-1)$; if the signs are the same then $\si(K_+)=\si(K_-)$, otherwise $\si(K_+)+2 = \si(K_-).$ 

Similar formulas hold for almost classical knots provided one uses compatible Seifert surfaces for $K_+, K_-,$ and $K_0$, which will be denoted $F_+, F_-,$ and $F_0$, respectively and which are assumed to be identical except for the half-twisted bands at the crossing. One must further assume that the symmetrized Seifert matrices for $K_+, K_-$, and $K_0$ are all non-singular, which is automatic for classical knots but which does not always hold for almost classical knots. Under those assumptions, for appropriate choices of generators, the Seifert matrices $V^\pm(K_+)$ and $V^\pm(K_-)$ are given by
$$V^+(K_+) = \begin{bmatrix} a & u \\ v &W^+ \end{bmatrix} \quad \text{ and }
\quad V^+(K_-) = \begin{bmatrix} a+1 & u \\ v &W^+ \end{bmatrix},$$
$$V^-(K_+) = \begin{bmatrix} a & u' \\ v' &W^- \end{bmatrix} \quad \text{ and }
\quad V^-(K_-) = \begin{bmatrix} a+1 & u' \\ v' &W^- \end{bmatrix}.$$
It follows that the directed Alexander polynomials satisfy the skein relation
\begin{equation}\label{eq-skein}
\na^\pm_{K_+, F_+}(t) - \na^\pm_{K_-, F_-}(t) = (t^{-1/2}-t^{1/2}) \na^\pm_{K_0, F_0}(t)
\end{equation}
and that equations \eqref{eq-Conway-1} and \eqref{eq-Conway-2} continue to hold, where in the second formula $\De_K(-1)$ is replaced by $\na^\pm_{K,F}(-1).$ (Notice that Lemma \ref{first-lem} implies that $\na^+_{K,F}(-1) = \na^-_{K,F}(-1).$) 

\begin{example} \label{5-Conway}
In this example, we apply Conway's method to determine the effect on the signature of the crossing change for the virtual knots 5.2012 and 5.2433 depicted in Figure \ref{5-2012}.
Note that 5.2433 has all negative crossings, and that 5.2012 is obtained by changing one of the crossings of 5.2433.   

Set $K_- = 5.2433$ and $K_+ = 5.2012$,  and let $F_-$ and $F_+$ be the associated Seifert surfaces. From Example \ref{5-family}, we have 
\begin{eqnarray*}
\na^\pm_{K_-,F_-}(t) &=& t^2-t+1-t^{-1}+t^{-2},\\
\na^\pm_{K_+,F_+}(t) &=& t-1+t^{-1}.
\end{eqnarray*}
Since $\na^\pm_{K_-,F_-}(-1) = 5 > 0$ has the opposite sign of $\na^\pm_{K_+,F_+}(-1) = -3<0$,  equations \eqref{eq-Conway-1} and \eqref{eq-Conway-2} imply that $\si(K_-,F_-)=\si(K_+,F_+) +2.$ This is consistent with the calculation given in  Example \ref{5-family}, where we determined that  $ \si(K_-,F_-) = 4$ and  $\si(K_+,F_+) = 2.$  
\end{example}

\subsection{Slice obstructions} \label{subsec:slice-o}
In this subsection, we apply the previous results  to the question of sliceness of almost classical knots. Combined with the graded genus  \cite{Boden-Chrisman-Gaudreau-2017, Boden-Chrisman-Gaudreau-2017t}, the computations are sufficient  to determine sliceness for all almost classical knots up to six crossings. These results are presented in Table \ref{table-2} at the end of the paper, which gives the graded genus, signature, and $\om$-signatures for these knots.  The almost classical knots that are known to be slice are also displayed with their slicings in Figure \ref{slice-Gauss}. 

All the other knots in the table can be seen to be non-slice, and we explain this now. In the majority of cases, this follows by applying Theorem \ref{thm:main}, using either  the signature or $\om$-signatures. When the signatures and $\om$-signatures fail to obstruct sliceness, we apply  Theorem \ref{thm:main-2} or the graded genus. For instance, the knots $K_1=4.108$ and $K_2=6.90172$ are  classical, and although their signatures and $\om$-signatures all vanish, their Alexander-Conway polynomials, which are $\De_{K_1}(t)=t-3+t^{-1}$ and $\De_{K_2}(t)=t^2-3t+5-3t^{-1}+t^{-2}$ do not factor as $f(t) f(t^{-1})$, and so Theorem \ref{thm:main-2} applies to show that neither $K_1$ nor $K_2$ is slice. 

In a similar way, one can use the directed Alexander polynomials to obstruct sliceness for the knots  6.87857 and 6.90194. For  $K=6.87857$, the directed Alexander polynomials are $\na^+_{K,F}(t)=-2t+4-2t^{-1}$ and   $\na^-_{K,F}(t)=-t+6-t^{-1}$, neither of which factors as $f(t) f(t^{-1})$, so Theorem \ref{thm:main-2} applies to show it is not slice. Likewise, $K=6.90194$ has down polynomial $\na^-_{K,F}(t)=-t^2-3t+8-3t^{-1}+t^{-2},$ which does not factor so $K$ is not slice.

The five almost classical knots 6.77905, 6.77985, 6.78358, 6.85091, and 6.90232 all have graded genus $\vartheta(K) =1$. Since $\vartheta(K)$ is a concordance invariant of virtual knots \cite{Boden-Chrisman-Gaudreau-2017}, it follows that none of them are slice. 

In some cases, we obstruct sliceness using invariants derived from a second Seifert surface $F'$. The new surface is obtained from the standard one by connecting it to a parallel copy of the Carter surface by a small tube. For instance, for $K=5.2439$, the new Seifert matrices are 
$$V^- = \begin{bmatrix}-1 & 1 & 0 & 0 & 0 & 1 \\ -1 & 0 & 1 & 0 & 1 & 0 \\ 0 & 0 & 1 & 0 & 0 & 0 \\ 0 & 0 & -1 & 1 & 0 & 0 \\ 0 & 0 & 0 & 0 & 0 & 1 \\ 0 & 0 & 0 & 0 & -1 & 0
\end{bmatrix} \qquad
V^+ = \begin{bmatrix}-1 & 0 & 0 & 0 & 0 & 1 \\ 0 & 0 & 0 & 0 & 1 & 0 \\ 0 & 1 & 1 & -1 & 0 & 0 \\ 0 & 0 & 0 & 1 & 0 & 0 \\ 0 & 0 & 0 & 0 & 0 & 0 \\ 0 & 0 & 0 & 0 & 0 & 0
\end{bmatrix}.$$
From this, it follows that $\si(K,F')=2$, and Theorem \ref{thm:main} applies to show $K$ is not slice. 
For $K = 6.77908$, using a second Seifert surface $F'$, we find that $\widehat{\si}^\pm_\om(K,F')$ takes values $\{-2,0\}$, and Theorem \ref{thm:main} again shows it is not slice.

In the case of $K = 6.72695$, the new Seifert surface $F'$ has $\na^\pm_{K,F'}(t)=-t+3-t^{-1}$, which does not factor, and Theorem \ref{thm:main-2} applies and gives an obstruction to sliceness. Likewise for $K = 6.87548$, using a second Seifert surface $F'$, we have $\na^+_{K,F'}(t)=-t^2-t+5-t^{-1}-t^{-2}$
and $\na^-_{K,F'}(t)=-t+3-t^{-1}$, and neither factors, so Theorem \ref{thm:main-2} applies again to show it is not slice.

Table \ref{table-2} shows, with fewer details, a slice obstruction for each of the remaining non-slice almost classical knots with up to six crossings, and Seifert pairs for all almost classical knots are given in Table \ref{table-3}.


\subsection{Slice genera} \label{subsec-tables}
In this subsection, we return to the question \cite{Dye-Kaestner-Kauffman-2014} mentioned in the introduction on the virtual slice genus of classical knots, which we rephrase optimistically below as a conjecture. We begin with a brief review of the slice genus for virtual knots.

The (smooth) slice genus $g_s(K)$ of a virtual knot $K$ was  introduced in Subsection \ref{subsec-VKC} and is related to the virtual unknotting number of $K$ (see \cite{Boden-Chrisman-Gaudreau-2017}). Computations of the slice genus for many virtual knots with up to six crossings are given in \cite{Boden-Chrisman-Gaudreau-2017, Rushworth-2017}, and a table of these computations can be found online \cite{Boden-Chrisman-Gaudreau-2017t}. 

The \emph{topological} slice genus $g_s^{\rm top}(K)$ is defined as the minimum genus over all topological locally flat embedded oriented surfaces $F$ in $W \times I$ with $\partial F=K$, where $W$ is a 3-manifold with $\partial W = \Si$ and $K$ is a representative knot in $\Si \times I$ (cf. Definition \ref{defn-conc}).
The following conjecture was posed as an open problem by Dye, Kaestner, and Kauffman \cite{Dye-Kaestner-Kauffman-2014}.

\begin{conjecture} \label{conj-DKK}
For any classical knot, its smooth and topological slice genus as a virtual knot agree with its smooth and topological slice genus as a classical knot.
\end{conjecture}

Using $g_4(K)$ and $g_4^{\rm top}(K)$ to denote the smooth and topological slice genera as a classical knot,
the conjecture asserts that, if $K$ is classical, then
\begin{equation} \label{conj-DKK}
g_s(K) =g_4(K) \quad \text{ and } \quad g_s^{\rm top}(K)=g_4^{\rm top}(K).
\end{equation}

Using KnotInfo \cite{Knotinfo}, one can use the signatures and Rasmussen invariants to verify this conjecture for classical knots with up to 12 crossings. This step implicitly uses the fact that the signatures and Rasmussen invariant extend to the virtual setting. The first is proved here, and the second is a consequence of \cite{Dye-Kaestner-Kauffman-2014}.

The next result provides a summary of our findings.

\begin{proposition}
For classical knots with up to 10 crossings, Conjecture \ref{conj-DKK} is true for both the smooth and topological slice genera.
 
For classical knots with 11 crossings,
Conjecture \ref{conj-DKK} on the topological slice genus is true in all cases. The conjecture on the smooth slice genus is true with just two possible exceptions: $11a_{211}$ and $11n_{80}.$

For classical knots with 12 crossings,
Conjecture \ref{conj-DKK} on the topological slice genus is true with 15 possible exceptions. The conjecture on the smooth slice genus is true with 37 possible exceptions. The possible exceptions are listed below in equations \eqref{exc-list1}, \eqref{exc-list2},\eqref{exc-list3}, and \eqref{exc-list4}.

\end{proposition}

\begin{proof}
The main result in \cite{Boden-Nagel-2016}, which holds for both smooth and topological concordance, implies that 
equation \eqref{conj-DKK} holds for any classical knot with smooth  slice genus equal to one. Thus, to confirm the conjecture, we only need to consider knots with $g_4(K) \ge 2$.

For classical knots with up to ten crossings, for all but five cases, either $g_4(K) \leq 1$ or $g_4(K) = |\si(K)|/2$. This confirms the conjecture in all cases except for the knots:
$$10_{139}, 10_{145}, 10_{152}, 10_{154} \text{ and } 10_{161}.$$ 
Interestingly, each one has $g_4(K) \neq g_4^{\rm top}(K)$ and satisfies $g_4(K) = |s(K)|/2$ and $g_4^{\rm top}(K) = |\si(K) |/2.$ Combining our signature results with the results of \cite{Dye-Kaestner-Kauffman-2014} on Rasmussen's invariant, it follows that \eqref{conj-DKK} holds for these five knots. We conclude that  Conjecture \ref{conj-DKK} holds for all knots with up to 10 crossings.

In a similar way, one can check the conjecture on all the classical knots with 11 crossings, and it can be confirmed for all but the following six knots, namely   
$$11a_{211},  11n_{9}, 11n_{31}, 11n_{77}, 11n_{80}  \text{ and } 11n_{183}.$$
The smooth slice genus is unknown for $11n_{80},$ and the other five have $g_4(K) \neq g_4^{\rm top}(K).$ Four of them satisfy $g_4(K) = |s(K)|/2$ and $g_4^{\rm top}(K) = |\si(K) |/2$, so \eqref{conj-DKK} holds for these four. This confirms the conjecture  for classical knots with 11 crossings with two possible exceptions: $11a_{211}$ and $11n_{80}.$

Similar methods apply to 12 crossing knots. For example,
the topological slice genus $g_4^{\rm top}(K)$  has been computed for all 12-crossing knots except for the following seven \cite{Lewark-McCoy-2017}, \cite{Knotinfo}:
\begin{equation} \label{exc-list1}
12a_{244}, 12a_{810}, 12a_{905}, 12a_{1142}, 12n_{549}, 12n_{555} \text{ and } 12n_{642}.
\end{equation}
For the other knots with 12-crossings, one can check that they all satisfy $g_4^{\rm top}(K) \leq 1$ or $g_4^{\rm top}(K) = |\si(K)|/2$, except for the
following  eight  \cite{Knotinfo}:
\begin{equation} \label{exc-list2}
12a_{787}, 12n_{269}, 12n_{505}, 12n_{598}, 12n_{602}, 12n_{694}, 12n_{749} \text{ and } 12n_{756}.
\end{equation}
This confirms that $g_4^{\rm top}(K)=g_s^{\rm top}(K)$ for  classical knots with 12 crossings with 15 possible exceptions.

The smooth slice genus $g_4(K)$ has been computed for all 12-crossing knots except for the following 20 \cite{Lewark-McCoy-2017}, \cite{Knotinfo}:
\begin{equation} \label{exc-list3}
12a_{\{153, 187, 230, 317, 450, 570, 624, 636, 905, 1189, 1208\}} \text{ and }  12n_{\{52, 63, 225, 239, 512, 555, 558, 665, 886\} }.
\end{equation}
For the other knots with 12-crossings, one can check that they all satisfy $g_4(K) \leq 1$ or $g_4(K) = |s(K)|/2$, except for the
following 17  \cite{Knotinfo}:
\begin{equation} \label{exc-list4}
12a_{\{244, 255, 414, 534, 542, 719, 787, 810, 908, 1118, 1142, 1185\}} \text{ and }  12n_{\{269, 505, 598, 602, 756\}}.
\end{equation}
This confirms that $g_4(K)=g_s(K)$ for  classical knots with 12 crossings with 37 possible exceptions.
\end{proof}

For almost classical knots up to six crossings, our methods were not able to determine the slice genus in four cases: $$6.90115, 6.90146, 6.90150, \text{ and } 6.90194.$$ For the first three, the signature tells us that  $g_s(K) \geq 1$, and for 6.90194, Theorem \ref{thm:main-2}  implies that
 $g_s(K) \geq 1$. Further, by the results of \cite{Boden-Chrisman-Gaudreau-2017}, we know that $g_s(K) \leq 2$ for all four.
In the classical case, for stubborn knots such as $8_{18}$ with vanishing signature, signature function, and Rasmussen invariant, the slice genus can be investigated by other techniques, such as the $T$-genus and the triple point method \cite{Murakami-Sugishita-1984}. It would be interesting to generalize these results to virtual knots, and in their recent paper \cite{Fedoseev-Manturov-2017}, Fedoseev and Manturov define a slice criterion for free knots using triple points of free knot cobordisms. Their work represents a promising development toward realizing this goal.

At the end of this paper, we have included a number of figures and tables of almost classical knots with up to six crossings.  Figure \ref{slice-Gauss} shows all the slice almost classical knots, with the slicing indicated as a saddle move on the Gauss diagram. Table \ref{table-1} gives their Alexander-Conway polynomials and Table \ref{table-2} shows their graded genus, signature, $\om$-signatures, and slice genus. Figure \ref{ACknot-diagrams} gives realizations of each almost classical knot as a knot in a thickened surface. The knot diagrams determine Seifert surfaces in the usual way, and the resulting pairs of Seifert matrices are listed in Table \ref{table-3}. Note that the Tables \ref{table-1}, \ref{table-2}, and \ref{table-3} include classical knots.


\section{Parity, projection, and concordance} \label{sec-2}
In this section, we introduce Manturov's notion of parity projection and Turaev's construction of lifting of knots in surfaces.  
The main result is that parity projection preserves concordance, and it is established  by interpreting parity projection in terms of lifting knots to coverings.

\subsection{Parity and projection} \label{PandP}
A parity is a collection of functions $\{f_D\}$, one for each Gauss diagram $D$ in the \emph{diagram category}, which is a category $\sD$ whose objects are Gauss diagrams and whose morphisms are compositions of Reidemeister moves. Given $D \in {\rm ob}(\sD)$, the function $f_D$ is a map from the chords of $D$ to the set $\{0, 1\}.$ Chords $c$ with $f_D(c)=1$ are called \emph{odd}, and those with $f_D(c)=0$ are called \emph{even}. 
The collection $\{f_D \mid D \in {\rm ob}(\sD) \}$ of functions is required to satisfy the following:

\subsection*{Parity axioms}
Suppose $D$ and $D'$ are related by a single Reidemeister move.
\begin{enumerate} \addtocounter{enumi}{-1}
\item The parity of every chord not participating in the Reidemeister move does not change.
\item If $D$ and $D'$ are related by a Reidemeister I move which eliminates the chord $c_0$ of $D$, then, $f_D(c_0)=0$.  
\item If $D$ and $D'$ are related by a Reidemeister II move which eliminates the chords $c_1$ and $c_2$ of $D$, then, $f_{D}(c_1)=f_D(c_2)$.
\item If $D$ and $D'$ are related by a Reidemeister III move, then the parities of the three chords involved in the Reidemeister III move do not change. Moreover, the three parities involved in the move are either all even, all odd, or exactly two are odd.  
\end{enumerate}

In \cite{Manturov-2010}, this is referred to as ``parity in the weak sense''. Readers interested in more details on parity are referred to \cite{Manturov-2010}, \cite{Ilyutko-Manturov-Nikonov-2011} and \cite{Nikonov-2016}. 

Given a parity $f$, there is a map on Gauss diagrams called \emph{parity projection} and denoted $P_f$. For a Gauss diagram $D$, its image $P_f(D)$ under parity projection is the Gauss diagram obtained from $D$ by eliminating all of its odd chords. 

\begin{proposition}[\cite{Manturov-2010}] \label{prop-closedparity}
If $D$ and $D'$ are two Gauss diagrams  equivalent through Reidemeister moves, then their parity projections $P_f(D)$ and $P_f(D')$ are also equivalent through Reidemeister moves.
\end{proposition}
This result is a direct consequence of the parity axioms. Although parity projection is defined in terms of the underlying Gauss diagram, Proposition \ref{prop-closedparity} implies that it is well-defined as a map on virtual knots. 

Since each application of parity projection $P_f$ removes chords from the Gauss diagram $D$, we have $P_f^{k+1}(D) =P^k_f(D)$ for $k$ sufficiently large.  Using this observation, we define $P_f^\infty(D)=\lim_{k\to \infty}P^k_f(D)$, and we call $P_f^\infty$ the \emph{stable projection} with respect to the parity $f$. Although stable projection is defined on the level of the Gauss diagrams, it gives rise to a well-defined map on virtual knots.

\subsection{Gaussian parity} \label{gaussianparity}

We now consider Gaussian parities, which are defined in terms of the indices of the chords in a Gauss diagram.

\begin{definition}
Let $n >1$ be an integer and $D$ a Gauss diagram. The \emph{mod $n$ Gaussian parity} is defined by setting, for any chord $c$ of $D$, 
$$f_n(c) = \begin{cases} 0 & \text{if $\ind(c) = 0$ mod $n$, and} \\ 1 &  \text{otherwise.}\end{cases}$$
\end{definition}
 
Here, we  denote mod $n$ Gaussian parity by $f_n$, its associated parity projection by $P_n$, and stable projection with respect to $f_n$ by $P_n^\infty$. For example, the mod $n$ parity projection $P_n$ acts on a Gauss diagram $D$ by removing all chords whose index is nonzero modulo $n$, and the image of $P^\infty_n$ consists of all Gauss diagrams $D$ with $\ind(c) = 0 \mod n$ for every chord $c$ of $D$. A virtual knot satisfies this condition if and only if it can be represented as a knot $K$ in a thickened surface $\Si \times I$ which is homologically trivial in $H_1(\Si, \ZZ/n)$  (see  \S5 of \cite{Boden-Gaudreau-Harper-2016}).

Another very useful parity is the total Gaussian parity, which is defined next.

\begin{definition}
The \emph{total Gaussian parity} is defined by setting, for any chord $c$ of $D$,
$$f_0(c) = \begin{cases} 0 & \text{if $\ind(c) = 0$, and} \\ 1 &  \text{otherwise.}\end{cases}$$ 
\end{definition}

Alternatively, if $D$ is a Gauss diagram, one can define the total Gaussian parity by setting $f_0(c)=\lim_{n\to \infty}f_n(c)$. Indeed, taking $n$ to be larger than the number of chords of $D$, one can easily see that $f_n(c)=0$ if and only if $\ind(c)=0.$

Throughout this paper, we  denote total Gaussian parity by $f_0$, its associated parity projection by $P_0$, and the stable projection with respect to $f_0$ by $P_0^\infty.$ 
For example, the parity projection $P_0$ acts on a Gauss diagram $D$ by removing all chords whose index is nonzero, and the image of $P_0^\infty$ consists of Gauss diagrams $D$ with $\ind(c) = 0$ for every chord $c$ of $D$, i.e., almost classical knots. 

\subsection{Coverings and Parity}  
In this subsection, we show how to interpret Gaussian parity projection of a virtual knot $K$ represented as a knot in a thickened surface in terms of Turaev's \emph{lifting} along an abelian covering of the surface \cite{Turaev-2008-a}.
Using this approach, we prove that if two virtual knots $K_0$ and $K_1$ are concordant, then their images $P_n(K_0)$ and $P_n(K_1)$ under parity projection are concordant (see Theorems \ref{thm_parity_conc} and \ref{thm_st_parity_conc}). 

In general, if $\pi \co \widetilde{\Si} \to \Si$ is a covering of surfaces and  $K$ is a knot in $\Si \times I,$ then Turaev observed that one can lift $K$ by taking $\widetilde{K}$ to be the knot in $\widetilde{\Si} \times I$ given as a connected component of the preimage $\pi^{-1}(K)$. The equivalence class of $\widetilde{K}$ can be shown to be independent of the choice of connected component, and one can define invariants of $K$ in terms of invariants of the lifted knot $\widetilde{K}$.  

We review this construction and relate the knots obtained by lifting along abelian covers to those obtained by projection with respect to Gaussian parity. This correspondence gives a natural topological interpretation of parity projections, and we use it to prove that parity projection preserves concordance. We begin with Turaev's construction for lifting knots along covers and explain how it is related to parity projection.

Let $\Si$ be a closed oriented surface, $p\co \Si\times I \to \Si$ the projection, and $K$ a knot in $\Si\times I$.
Fix an integer $n>1$ and consider the homology class $[p_*(K)]\in H_1(\Si;\ZZ/n)$. Its 
 Poincar\'{e} dual determines an element in $H^1(\Si;\ZZ/n) \cong \Hom_{\ZZ}(H_1(\Si),\ZZ/n)$, and  precomposing with the Hurewicz map, we obtain a homomorphism $\psi_K \co \pi_1(\Si,c) \to\ZZ/n$.
The proof of the next result is standard and left to the reader.

\begin{lemma} \label{lemm_left} 
For $\al \in \pi_1(\Si,c)$, $\psi_K(\al)=[\al] \cdot [K] \pmod n$.
\end{lemma}
 
The kernel of $\psi_K$ yields a regular covering space $\Si^{(n)} \to \Si$ of finite index. Moreover, the map sends the homotopy class of $K$ to $0$. Hence, $K$ lifts to a knot $K^{(n)}$ in $\Si^{(n)} \times \RR$. As any lift $K^{(n)}$ of $K$ may be obtained from any other lift by a diffeomorphism
of $\Si$, all lifts of $K$ represent the same virtual knot. The knot $K^{(n)}$ is called a \emph{lift} of $K$.
 
\begin{example} If the map $\psi_K \co \pi_1(\Si,c) \to\ZZ/n$ is trivial, then $\Si^{(n)}=\Si$ and $K^{(n)}=K$. If the group of covering transformations of $\Si^{(n)} \to \Si$ is non-trivial and at each crossing point of $K$ the over-crossing arc and under-crossing arc lift to different sheets, then $K^{(n)}$ is a simple closed curve on $\Si^{(n)}$ representing the trivial virtual knot.  
\end{example}

The next lemma relates the lift $K^{(n)}$ of a virtual knot $K$ to its image $P_n(K)$ under mod $n$ parity projection. 

\begin{lemma} \label{lemma_tur} Let $K$ be a knot in $\Si \times I$ and $K^{(n)}$  the lift of $K$ to the covering space $\Si^{(n)} \times I$. Then the Gauss diagram  $D^{(n)}$  of $K^{(n)}$ is obtained from the Gauss diagram $D$ of $K$ by deleting the chords $c$ with $\ind(c)\neq 0 \pmod n $. In particular, as virtual knots, the lift $K^{(n)}$ of $K$ is equivalent to its image $P_n(K)$ under parity projection.
\end{lemma}

\begin{proof} Let $c \in \Si$ be a double point of $p_*(K)$, where $p \co \Si \times I \to \Si$. Applying the oriented smoothing at $c$ gives knots $K'$ and $K''$ on $\Si \times I$,  and using equation \eqref{eq:intersect} and the fact that $\al \cdot \al=0$ holds for any $\al \in H_1(\Si)$, we see that
$$\ind(c) = [p_*(K')]\cdot [p_*(K'')] =[p_*(K')] \cdot [p_*(K)].$$

Now viewing $p_*(K')$ as an element of $\pi_1(\Si,c)$, the lifting criterion and Lemma \ref{lemm_left} imply that $p_*(K')$ lifts to a closed curve in the covering space $\Si^{(n)}$ precisely when $\ind(c)=[p_*(K')] \cdot [p_*(K)] = 0 \pmod n$. On the other hand, if $c$ is a crossing with $\ind(c) \neq 0 \pmod n$,  then $p_*(K')$ does not lift to a closed curve in the covering space $\Si^{(n)} \to \Si$, hence this crossing does not appear in the lifted knot $K^{(n)}$. Thus the Gauss diagram of $K^{(n)}$ is obtained from $D$ by deleting those chords $c$ with $\ind(c) \neq 0 \pmod n$. 
\end{proof}

\begin{figure}[ht]
\includegraphics[scale=0.82]{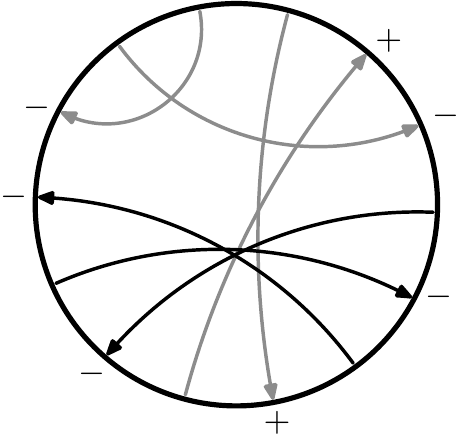}  \qquad
\def\svgwidth{0.52\textwidth}  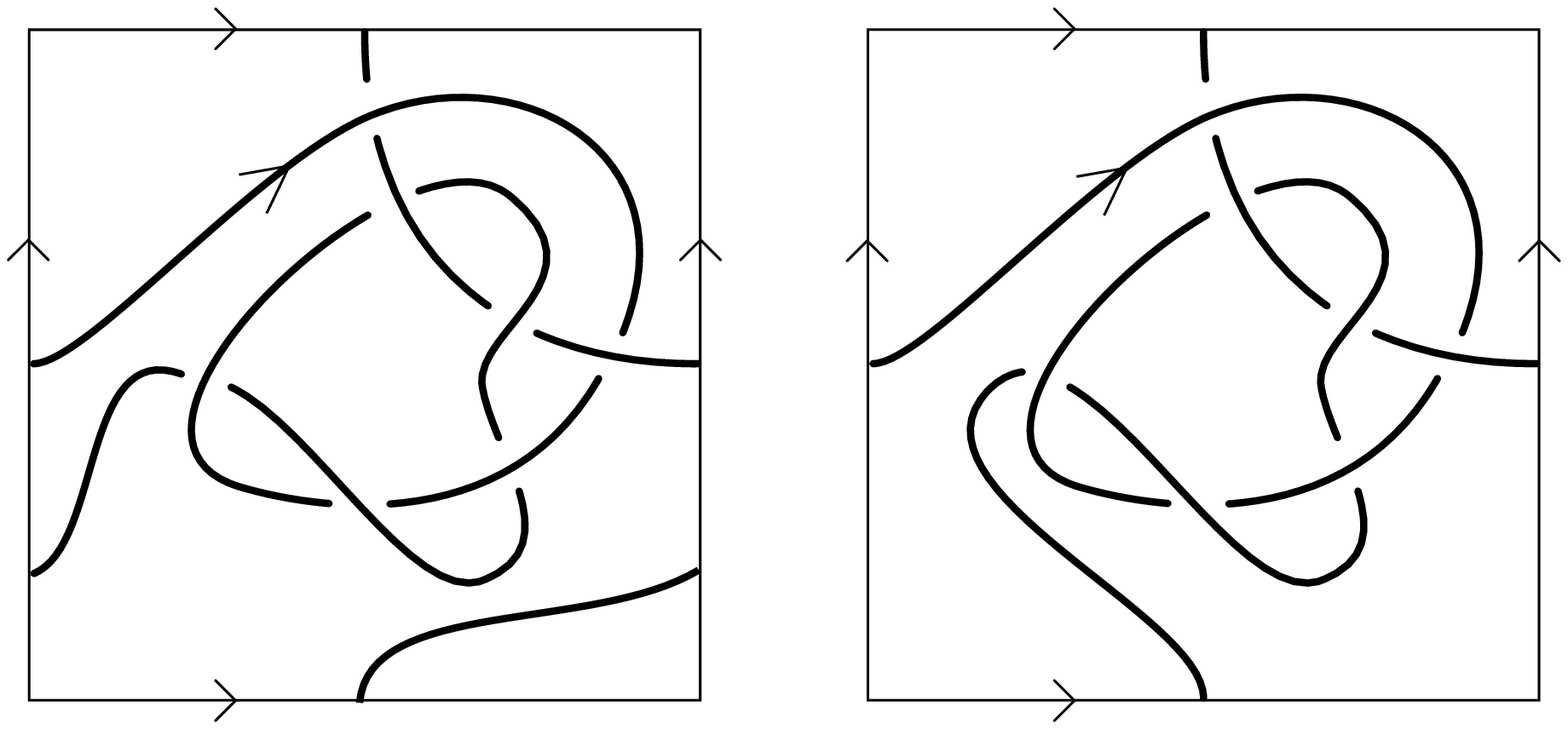
  \caption{A Gauss diagram with two representative knots in $T^2 \times I$.}
\label{covering2}
\end{figure}

\begin{example}
Starting with the Gauss diagram $D$ on the left of Figure \ref{covering2}, we find a representative knot in a surface. Its depiction is not unique, as the surface can be stabilized, and Dehn twists can be applied, which alter the homology class represented by the knot diagram. For $D$, two such representatives are depicted on the right of Figure \ref{covering2}. They are inequivalent as knots in thickened surfaces, but have the same underlying Gauss diagram. Let $K\subset T^2 \times I$ be the knot in the 2-torus depicted on the left in Figure \ref{covering2}, and let $\al,\be \in H_1(T^2)$ be the generators for  its first homology. Notice that $[K]= \al$,  hence, the lift of $K$ to the 2-fold cover is obtained by gluing two copies of $T^2$ along $\al$ and identifying the remaining opposite sides. The lifted knot $K^{(2)}$ is depicted in Figure \ref{covering3}, and one can check that the Gauss diagram $D^{(2)}$ for the lift  $K^{(2)}$ is obtained by forgetting the grey arrows of $D$. 

\begin{figure}[ht]
\def\svgwidth{0.5\textwidth} 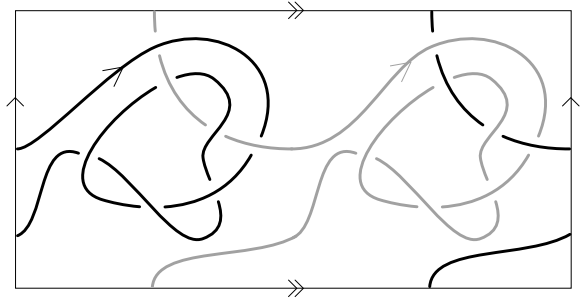
  \caption{Lifting the knot to a two-fold cover}
\label{covering3}
\end{figure}
\end{example}

The following lemma, due to Turaev, shows how concordance behaves under lifting knots along coverings.

\begin{lemma}[Lemma 2.1.1, \cite{Turaev-2008-a}] \label{lemma_turaev_211} 
Suppose $K_i$ is a knot in the thickened surface $\Si_i \times I$ for $i=0,1$. If $K_0$ and $K_1$ are concordant, then the lifted knots $K_0^{(n)}$ and $K_1^{(n)}$ are also concordant.  
\end{lemma}

The following is the main result in this section, and it implies that parity projection preserves concordance. 
\begin{theorem} \label{thm_parity_conc} 
Suppose $K_i$ is a virtual knot for $i=0,1$ and $n\geq 0$ is an integer with $n\neq 1.$ If $K_0$ is concordant to $K_1$, then their images $P_n(K_0)$ and $P_n(K_1)$ under parity projection are concordant.
\end{theorem}

\begin{proof}
For $n \geq 2$, the theorem is an immediate consequence of combining Lemmas \ref{lemma_tur} and \ref{lemma_turaev_211}.  The case $n=0$ follows from the observation that, for a Gauss diagram with fewer than $n$ chords, $P_0(D) =P_n(D)$.
\end{proof}

\begin{example} Consider the virtual knot $K=6.89907$, whose
Gauss diagram appears on the left in Figure \ref{GD-6-89907}. In \cite{Boden-Chrisman-Gaudreau-2017t}, we show that it has virtual slice genus $g_s(K) \leq 1.$ Applying parity projection with respect to total Gaussian parity, the resulting knot $K'= P_0(K)$ has Gauss diagram on the right in Figure \ref{GD-6-89907}. We see that $K'$ is the reverse of 4.107, and in \cite{Boden-Chrisman-Gaudreau-2017} we show that 4.107 has graded genus $\vartheta(K') =1.$
It follows that $K'$ is not slice, and Theorem \ref{thm_parity_conc} implies that $K$ is also not slice. Hence $g_s(K)=1.$

\begin{figure}[ht]
\includegraphics[scale=1.00]{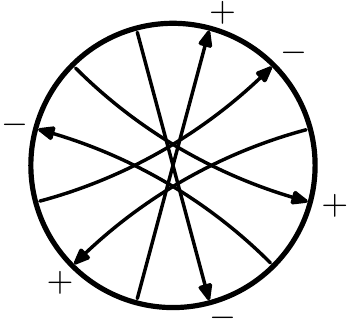}  \qquad
\includegraphics[scale=1.00]{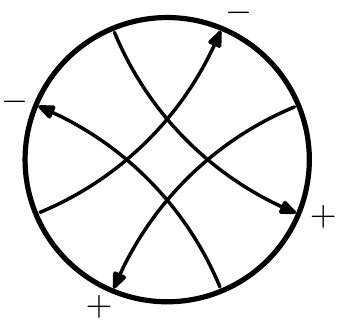}
\caption{Gauss diagrams of $6.89907$ and the reverse of $4.107$.}
\label{GD-6-89907}
\end{figure}
\end{example}

The next result follows by repeated application of Theorem \ref{thm_parity_conc}. 
\begin{theorem} \label{thm_st_parity_conc} 
If $K_0$ and $K_1$ are virtual knots and $K_0$ is concordant to $K_1$, then their images  $P^\infty_n(K_0)$ and $P^\infty_n(K_1)$ under stable projection are concordant for any $n \geq 2$ and $n=0.$
\end{theorem}

In particular, taking $n=0$ in Theorem \ref{thm_st_parity_conc} allows us to reduce questions about concordance of virtual knots to questions about concordance of almost classical knots. In terms of the  Tristram-Levine signatures defined for almost classical 
knots in Section \ref{sec-3}, Theorem \ref{thm_st_parity_conc} shows that they lift to define slice obstructions for all virtual knots. 

We give an application of Theorem \ref{thm_parity_conc} to show that the polynomial
invariants introduced in  \cite{Jeong-2016} and  \cite{Im-Kim-2017} are concordance invariants of virtual knots.

In \cite{Jeong-2016}, Jeong introduced the zero polynomial $Z_K(t)$, and in \cite{Im-Kim-2017}, Im and Kim give a sequence of polynomial invariants $Z^n_K(t)$ for $n$ a non-negative integer. They note that for $n=0, Z^0_K(t) = Z_K(t),$ and they give a formula for computing $Z^n_K(t)$ in terms of Gauss diagrams (see Definition 3.1 of \cite{Im-Kim-2017}).

It is clear from that formula that $Z^n_K(t)$ is equal to the index polynomial applied to $P_n(K)$, where $P_n$ denotes mod $n$ Gaussian parity projection. Theorem 7 of \cite{Boden-Chrisman-Gaudreau-2017} shows that the writhe polynomial $W_K(t)$ and Heinrich-Turaev polynomial $P_K(t)$ are concordance invariants. The following, which concerns the polynomials invariants $Z^n_K(t)$  and  $Z_K(t)$ introduced in \cite{Im-Kim-2017} and \cite{Jeong-2016}, respectively, is an immediate consequence of combining the above observations with Theorem \ref{thm_parity_conc}. 

\begin{corollary} 
The polynomials $Z^n_K(t)$ of Im and Kim are concordance invariants for virtual knots. In particular, the zero polynomial $Z_K(t)$ of Jeong is a concordance invariant for virtual knots.
\end{corollary}

Theorem \ref{thm_parity_conc} says that the projection $P_n$ with respect to Gaussian parity $f_n$ preserves concordance, and it is an interesting question to determine which other parities $f$ have the property that projection $P_f$ preserves concordance.
\subsection{Iterated lifts of knots}
In this subsection, we present analogues of Lemma \ref{lemma_tur} and Theorem \ref{thm_parity_conc} for iterated lifts and iterated projections. Given a knot $K$ in the thickened surface $\Si \times I$ and a finite sequence of integers $n_1,\ldots,n_r$, with $n_i \ge 2$ for $1 \le i \le r$, the iterated lift of $K$ is the knot 
$$K^{(n_1,\ldots,n_r)} =\left(\cdots\left(\left(K^{(n_1)}\right)^{(n_2)}\right)\cdots\right)^{(n_r)}$$ 
in the thickening of the iterated covering space
$$\Si^{(n_1,\ldots,n_r)}=\left(\cdots\left(\left(\Si^{(n_1)}\right)^{(n_2)}\right)\cdots\right)^{(n_r)}.$$
The next result is obtained  by repeated application of Lemma \ref{lemma_tur} and Theorem \ref{thm_parity_conc}.

\begin{proposition} \label{prop_parity_conc}
Suppose $n_1,\ldots,n_r$ are integers with $n_i \ge 2$ for $1 \le i \le r$. 
\begin{enumerate}
\item[(i)] As virtual knots, the iterated lift $K^{(n_1,\ldots,n_r)}$ of $K$ is equivalent to the image $P_{n_k} \circ \cdots \circ P_{n_1}(K)$ under iterated application of parity projection. 
\item[(ii)] If $K_0$ and $K_1$ are virtual knots and $K_0$ is concordant to $K_1$, then their images $P_{n_r} \circ \cdots \circ P_{n_1}(K_0)$ and $P_{n_r} \circ \cdots \circ P_{n_1}(K_1)$ under iterated projections are concordant.
\end{enumerate}
\end{proposition}

Proposition \ref{prop_parity_conc} raises the question of commutativity of iterated projections, and the simplest instance is whether  $K^{(n,m)} \stackrel{?}{=} K^{(m,n)}$, or equivalently whether $P_{m} \circ P_n(K) \stackrel{?}{=} P_{n} \circ P_m(K).$ More generally, one can ask whether the iterated lifts $K^{(n_1,\ldots,n_r)}$ are independent of the order $n_1, \ldots, n_r$. To that end, let $N^{(n_1,\ldots,n_r)}$ denote the normal subgroup of $\pi_1(\Si,c)$ defining the covering space $\Si^{(n_1,\ldots,n_r)} \to \Si$. The relationship between the lifts of $K$ is entirely determined by the structure of the lattice of normal subgroups of $\pi_1(\Si,c)$. Thus, the question of commutativity of iterated lifts is really a question about the structure of the lattice of normal subgroups. 
 
Consider, for example, the portion of the lattice of regular covering spaces in the commutative diagram below. Here  $n,m \ge 2$ are integers, and $q$ is any non-zero multiple of $\text{lcm}(n,m)$. The expression on each arrow corresponds to the index of the covering projection. For simplicity of exposition, we will assume that the indices of the covering maps are as indicated, although this need not be the case in general. Note that since $m|q$ and $n|q$ there are natural maps $\ZZ/q \to\ZZ/n$ and $\ZZ/q \to\ZZ/m$. It follows from the definition of $\psi_K$ that $N^{(q)} \subset N^{(n)} \cap N^{(m)}$. Hence, $\Si^{(q)} \to \Si^{(n)}$ and $\Si^{(q)} \to \Si^{(m)}$ are also covering projections. 

\[
\xymatrix{
\Si^{(n,m)} \ar[dr]_{m} &  & \Si^{(q)} \ar[dd]_{q} \ar[dr]^{q/m} \ar[dl]_{q/n}& & \Si^{(m,n)} \ar[dl]_{n} \\
                         &  \Si^{(n)} \ar[dr]_{n}&  & \Si^{(m)} \ar[dl]^{m}&  \\
                         &                       & \Si & & 
}
\] 
Furthermore, the homotopy class of $K$ is an element of each of $N^{(n)}$, $N^{(m)}$, and $N^{(q)}$. Since the diagram commutes, it follows from the lifting criterion that $K^{(q)}$ is a lift of both $K^{(n)}$ and $K^{(m)}$. The following theorem shows how to obtain the Gauss diagram of $K^{(q)}$ from that of either $K^{(n)}$ or $K^{(m)}$. In this instance, it follows that $K^{(n,m)}=K^{(m,n)}.$

\begin{theorem} The Gauss diagram $D^{(q)}$ of $K^{(q)}$ is obtained from the Gauss diagram $D^{(m)}$ of $K^{(m)}$ by deleting the chords whose index in $K$ are $= 0 \pmod m$ and $\neq 0 \pmod q$. Similarly, $D^{(q)}$ is obtained from the Gauss diagram $D^{(n)}$ of $K^{(n)}$.
\end{theorem}
\begin{proof} The Gauss diagram $D^{(q)}$ is obtained from $D$ by deleting those chords having index $\neq 0 \pmod q$, and $D^{(m)}$ is obtained from $D$ by deleting those chords that have index $\neq 0 \pmod m$. Since $m|q$, the only chords of $D$ that need to be deleted from $D^{(m)}$ to obtain $D^{(q)}$ are those having index $= 0 \pmod m$ and $\neq 0 \pmod q$.
\end{proof}

On the other hand, by Lemma \ref{lemma_tur}, the Gauss diagram $D^{(m,n)}$ can be obtained from $D^{(m)}$ by deleting chords in $D^{(m)}$ with index  $\neq 0 \pmod n$. We may similarly obtain $D^{(n,m)}$ from $D^{(n)}$. In the general case, it is not always true that $D^{(m,n)}=D^{(n,m)}$, and in certain cases, it may turn out that the iterated covering $\Si^{(m,n)} \to \Si$ is nonabelian, as explained in the next proposition.

\begin{proposition} The group $\Ga$ of covering transformations of $\Si^{(m,n)} \to \Si$ is a group extension of $\ZZ/n$ by $\ZZ/m$. In other words, there is short exact sequence
\[
\xymatrix{
1\ar[r] &\ZZ/n \ar[r] & \Ga \ar[r] &\ZZ/m  \ar[r] & 1.
}
\]
\end{proposition}
\begin{proof} This follows immediately from the third isomorphism theorem, since $$\frac{\pi_1(\Si,c)}{N^{(m)}} \cong \frac{\pi_1(\Si,c)/N^{(m,n)}}{N^{(m)}/N^{(m,n)}}.$$
\end{proof}

\begin{figure}[ht]
\includegraphics[scale=0.6]{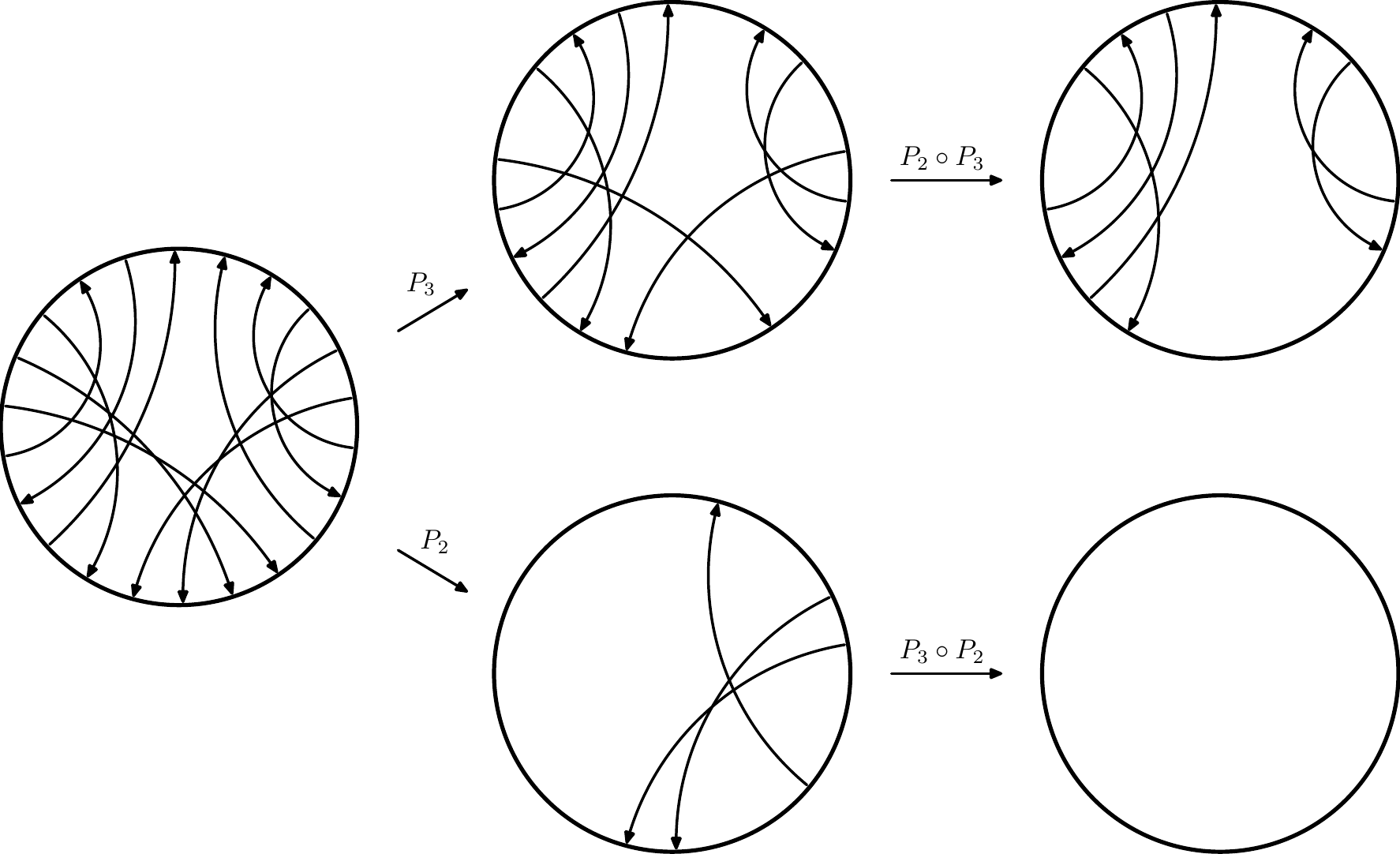}  
\caption{In the Gauss diagram above, all chords have negative sign. For the diagram $D$, notice that $P_2\circ P_3(D) \neq P_3 \circ P_2(D)$.}
\label{GD-iterated}
\end{figure}

\begin{example} 
Figure \ref{GD-iterated} shows a Gauss diagram $D$ such that $ P_2(P_3(D)) \neq P_3(P_2(D)).$
One can easily check that the four diagrams $P_2(D), P_3(D), P_2(P_3(D)),$ and $P_3(P_2(D))$ all represent distinct virtual knots. Notice that $P_2(P_3(D)) = P_6(D)$, thus $K^{(3,2)}$ is a knot in the abelian 6-fold cover. On the other hand, $K^{(2,3)}$ is a knot in the non-abelian cover whose associated covering group $G$ is isomorphic to the dihedral group of order 6 and sits in the short exact sequence
$$ 1 \to \ZZ/3 \to G \to \ZZ/2 \to 1.$$ 
\end{example}

\subsection*{Acknowledgements} We would like to thank Dror Bar Natan, Matthias Nagel, Andrew Nicas and Robb Todd for useful discussions, and we are grateful to Homayun Karimi for alerting us to an error in an earlier version of this paper.  H. Boden was supported by a grant from the Natural Sciences and Engineering Research Council of Canada,
M. Chrisman was supported by a Monmouth University Creativity and Research Grant, and
R. Gaudreau was supported by a scholarship from the National Centre of Competence in Research SwissMAP.

\newpage
\bigskip
\bigskip

\begin{figure}[H]  
\hspace{-0.6cm}\includegraphics[scale=0.92]{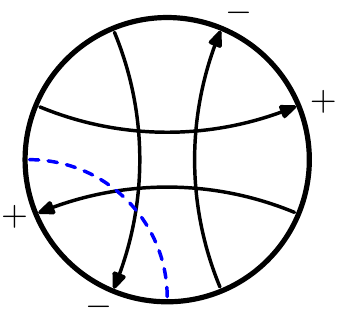} \hspace{0.8cm}  
\includegraphics[scale=0.92]{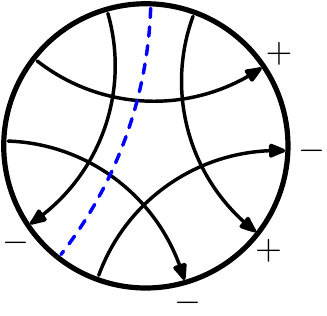} \hspace{0.8cm}
\includegraphics[scale=0.92]{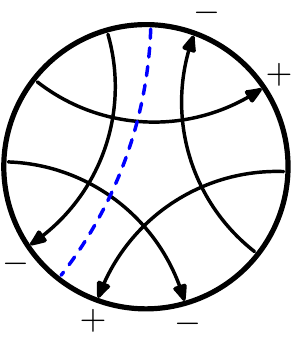} \hspace{0.8cm}
\includegraphics[scale=0.92]{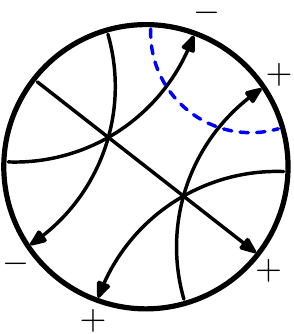}   
 
\smallskip 
\hspace{-.2cm} 4.99 \hspace{2.8cm} 5.2025 \hspace{2.8cm} 5.2133 \hspace{2.8cm} 5.2160 

\bigskip 
\hspace{-0.2cm}\includegraphics[scale=0.92]{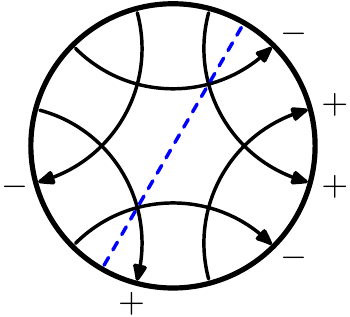} \hspace{0.4cm}
\includegraphics[scale=0.92]{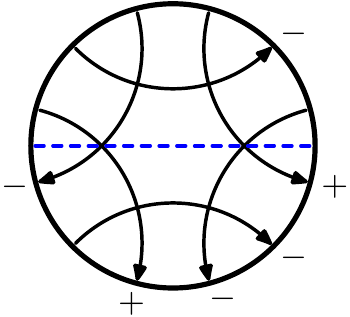}  \hspace{0.4cm}
\includegraphics[scale=0.92]{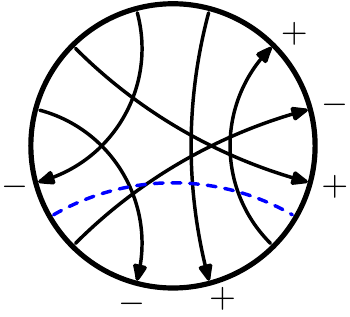} \hspace{0.4cm}
\includegraphics[scale=0.92]{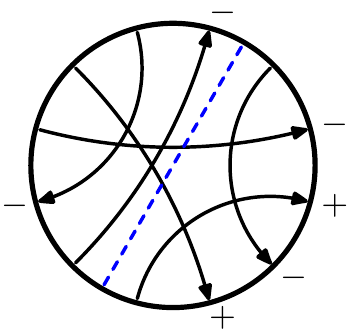}  

\smallskip 
\hspace{-0.4cm} 6.72557 \hspace{2.4cm} 6.72975 \hspace{2.4cm} 6.73007  \hspace{2.4cm} 6.73583

\bigskip
\includegraphics[scale=0.92]{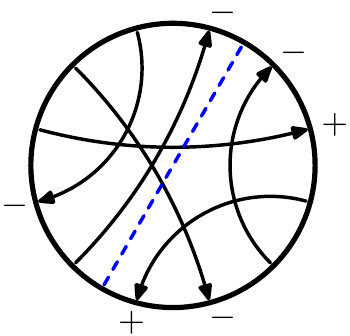} \hspace{0.4cm}
\includegraphics[scale=0.92]{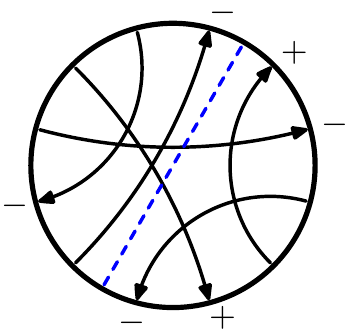} \hspace{0.4cm}
\includegraphics[scale=0.92]{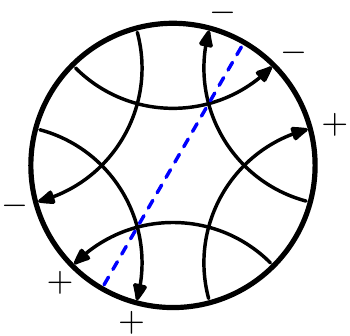} \hspace{0.4cm}
\includegraphics[scale=0.92]{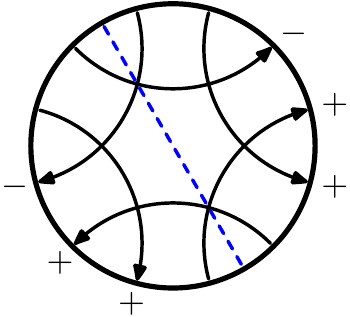} 

\smallskip 
6.75341 \hspace{2.4cm} 6.75348 \hspace{2.4cm} 6.77833  \hspace{2.5cm} 6.77844

\bigskip 
\includegraphics[scale=0.92]{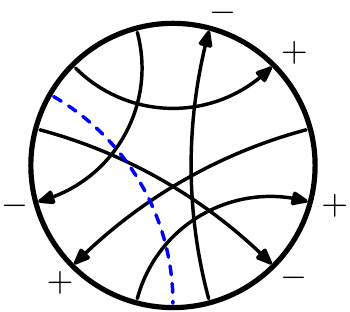} \hspace{0.6cm}
\includegraphics[scale=0.92]{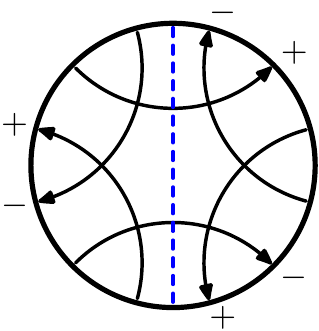} \hspace{0.6cm}
\includegraphics[scale=0.92]{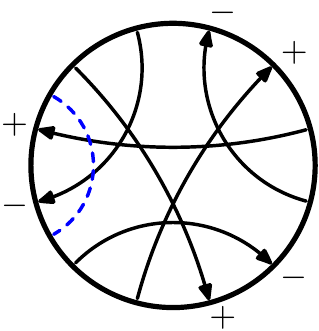} 

\smallskip 
\hspace{0.2cm} 6.79342 \hspace{2.4cm} 6.87269  \hspace{2.4cm} 6.87319   

\bigskip 

\includegraphics[scale=0.92]{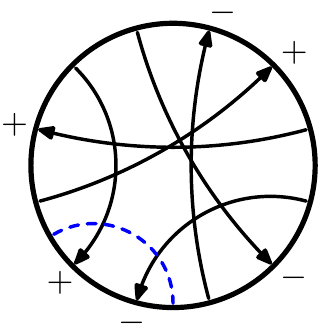} \hspace{0.6cm}
\includegraphics[scale=0.92]{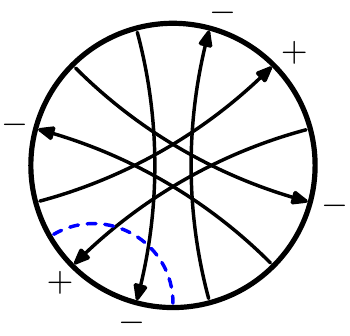}

\smallskip 
6.89623 \hspace{2.6cm} 
6.89815 

\bigskip
\caption{{\sc  Slice Gauss diagrams of almost classical knots}}
\label{slice-Gauss}
\end{figure}

\begin{table}[H] 

\begin{tabular}{cc}
\begin{minipage}{0.50\textwidth}
\begin{tabular}{|c|c|}
\hline
Virtual &  Alexander    \\
Knot &  Conway polynomial    \\
\hline \hline
{\bf 3.6} & $t-1+t^{-1}$\\ \hline
4.99  & $2-t^{-1}$ \\ \hline
4.105 & $2t-2+t^{-1}$ \\ \hline
{\bf 4.108}  & $t-3+t^{-1} $ \\  \hline
5.2012 & $t$ \\  \hline
5.2025 & $t$ \\  \hline
5.2080 & $1$ \\  \hline
5.2133 & $2-t^{-1}$ \\  \hline
5.2160 & $t-1+t^{-1} $ \\  \hline
5.2331 & $t^2-1+t^{-1} $ \\  \hline
5.2426 & $(t-1+t^{-1})^2$ \\  \hline
5.2433 & \;\; $t^2-2t +4 -3t^{-1} +t^{-2} $ \;\; \\  \hline
{\bf 5.2437} & $2t - 3 + 2t^{-1} $ \\  \hline
5.2439 & $t^2-2t +3 -t^{-1}$ \\  \hline
{\bf 5.2445} & $ t^2 - t + 1 - t^{-1} + t^{-2}$\\  \hline
6.72507 & $t$ \\  \hline
6.72557 &$t$ \\  \hline
6.72692 & $1 $ \\  \hline
6.72695 & $t$ \\  \hline
6.72938 & $t^2-t+1  $ \\  \hline
6.72944 & $2t-1  $ \\  \hline
6.72975 & $t$ \\  \hline
6.73007 & $1$ \\  \hline
6.73053 & $t-1+t^{-1} $ \\  \hline
6.73583 & $1$ \\  \hline
6.75341 & $2-t^{-1} $ \\  \hline
6.75348 & $-t^2+2t$ \\  \hline
6.76479 & $t-1+t^{-1} $ \\  \hline
6.77833 & $t^2-t+1 $ \\  \hline
6.77844 & $t^2-t+1 $ \\  \hline
6.77905 & $-t+3-t^{-1}$ \\  \hline
6.77908 & $2t-2+t^{-1} $ \\  \hline
6.77985 & $t-1+t^{-1} $ \\  \hline
6.78358 & $-t+3-t^{-1}$ \\  \hline
6.79342 & $-t+3-t^{-1}$ \\  \hline
6.85091 & $1+t^{-1}-t^{-2} $ \\  \hline
6.85103 & $t-1+2t^{-1}-t^{-2} $ \\  \hline
\;6.85613\; & $t-2+2t^{-1} $  \\  \hline

\end{tabular}
\end{minipage}

\begin{minipage}{0.50\textwidth}
 \begin{tabular}{|c|c|} 
\hline
Virtual &  Alexander   \\
Knot  &  Conway polynomial   \\ 
\hline \hline
6.85774 & $t - 1 + t^{-2} $  \\  \hline
6.87188 & $2t - 3+3 t^{-1}-t^{-2}$ \\  \hline
6.87262 & $t^3 - 2t^2 + 3t - 2 +t^{-1}$ \\  \hline
6.87269 & $4-4t^{-1} + t^{-2}$ \\  \hline
6.87310 & $t^2 -t +2 -2t^{-1} +t^{-2} $ \\  \hline
6.87319 & $3-3t^{-1}+t^{-2}$ \\  \hline
6.87369 & $t-2+3t^{-1}-t^{-2} $ \\  \hline
6.87548 & $-t^2+2t +1-t^{-1} $ \\  \hline
6.87846 &  $t-1+2t^{-1}-t^{-2}$ \\  \hline
6.87857 & $-t+4-2t^{-1} $ \\  \hline
6.87859 & $3-3t^{-1}+t^{-2} $ \\  \hline
6.87875 & $t+1-2t^{-1}+t^{-2} $ \\  \hline
6.89156 &  $2t-1-t^{-1}+t^{-2} $ \\  \hline
{\bf 6.89187} & $(t-1+t^{-1})^2$ \\  \hline
{\bf 6.89198} & $(t-1+t^{-1})^2$ \\  \hline
6.89623 & $2-2t^{-1}+t^{-2} $ \\  \hline
6.89812 & $t^2 -2+2t^{-1} $ \\  \hline
6.89815 & $2t^{-1} - t^{-2} $ \\  \hline
6.90099 & $t - t^{-1} + t^{-3} $ \\  \hline
6.90109 & \;\; $2t^2 - 4t +5 - 3t^{-1} + t^{-2}$ \;\; \\  \hline
6.90115 & $-t^2 + 4t -5+4t^{-1} - t^{-2}$ \\  \hline
6.90139 & $3t^2 - 6t +7 - 4t^{-1} + t^{-2}$ \\  \hline
6.90146 & $t^3 - 5t^2 + 9t - 5 + t^{-1}$ \\  \hline
6.90147 & $t^2-3t+6-5t^{-1}+2t^{-2} $ \\  \hline
6.90150 & $-t^2 +5t-6+4t^{-1}-t^{-2} $ \\  \hline
6.90167 & $t^2-2t+4 -4t^{-1}+2t^{-2} $ \\  \hline
{\bf 6.90172} & $t^2-3t+5-3t^{-1}+t^{-2} $ \\  \hline
6.90185 & $3t^2-6t+6-3t^{-1}+t^{-2}$ \\  \hline
6.90194 & $t^2-4t+8-5t^{-1}+t^{-2}$ \\  \hline
6.90195 & $2t^2-3t+3-2t^{-1}+t^{-2} $ \\  \hline
{\bf 6.90209} & $t^2-3t+3-3t^{-1}+t^{-2} $ \\  \hline
6.90214 & $3t-4+2t^{-1} $ \\  \hline
6.90217 & $-t^2+4t-3+t^{-1} $ \\  \hline
6.90219 & $2t - 3 + 3t^{-1} - t^{-2}$ \\  \hline
{\bf 6.90227} & $2t-3+2t^{-1}$ \\  \hline
6.90228 & $4t - 6 + 3t^{-1}$ \\  \hline
6.90232 & $-2 t + 6 - 4t^{-1} + t^{-2}$ \\  \hline
\; 6.90235 \; & $t -3+5t^{-1}-3t^{-2}+t^{-3} $ \\  \hline  
\end{tabular}
\end{minipage}
\end{tabular}
\bigskip
\caption{The Alexander--Conway polynomials of almost classical knots. Boldface font is used for classical knots, e.g. ${\bf 3.6} = 3_1$.}
\label{table-1}
\end{table}

\begin{table}[H] 

\begin{tabular}{cc}
\begin{minipage}{0.5\textwidth}
\begin{tabular}{|c|c|l|c|}
\hline

Virtual & Graded & \;\; Signatures  & Slice \\ 
knot & genus & $\quad \si \quad \{\widehat{\si}^\pm_\om\}$ 
& genus  \\ 
\hline 
\hline
{\bf 3.6}   & 0 & \quad 2 $\quad \{0,2\}$& 1 \\ \hline
4.99   & 0 & \quad 0 $\quad \{0\}$& 0 \\ \hline
4.105  & 1 & \quad 2 $\quad \{0, 2\}$  & 1 \\ \hline
{\bf 4.108}  & 0 & \quad 0 $\quad \{0\}$ & 1\\ \hline
5.2012  & 1 & \quad 2 $\quad \{0, 2\}$ & 1 \\ \hline 
5.2025  & 0 & \quad 0 $\quad \{0\}$ & 0 \\ \hline 
5.2080  & 1 & \quad 2 $\quad \{0, 2\}$ & 1 \\ \hline
5.2133  & 0 & \quad 0 $\quad \{0\}$ & 0 \\ \hline 
5.2160  & 0 & \quad 0 $\quad \{0\}$ & 0 \\ \hline 
5.2331  & 1 & \quad 2 $\quad \{2\}$ & 1 \\ \hline
5.2426  & 1 & \quad 4 $\quad \{0, 2, 4\}$ & 2 \\ \hline 
5.2433  & 2 & \quad 4 $\quad \{0, 2, 4\}$ & 2 \\ \hline
{\bf 5.2437} & 0 & \quad 2 $\quad \{0, 2\}$ & 1 \\ \hline
5.2439  & 0 & \quad 0 $\quad \{0\}$  & 1 \\ \hline
{\bf 5.2445} & 0 & \quad 4 $\quad \{0, 2, 4\}$ & 2 \\ \hline 
6.72507 & 1 & \quad 2 $\quad \{0, 2\}$ & 1 \\ \hline 
6.72557 & 0 & \quad 0 $\quad \{0\}$ & 0 \\ \hline 
6.72692 & 1 & \quad 2 $\quad \{2\}$  & 1 \\ \hline 
6.72695 & 0 & \quad 0 $\quad \{0\}$ & 1 \\ \hline 
6.72938 & 1 & \quad 2 $\quad \{0, 2\}$ & 1 \\ \hline
6.72944 & 0 & \quad 2 $\quad \{0, 2\}$ & 1 \\ \hline
6.72975 & 0 & \quad 0 $\quad \{0\}$ & 0 \\ \hline
6.73007 & 0 & \quad 0 $\quad \{0\}$ & 0 \\ \hline 
6.73053 & 1 & \quad 2 $\quad \{0, 2\}$ & 1 \\ \hline 
6.73583 & 0 & \quad 1 $\quad \{1\}$ & 0 \\ \hline 
6.75341 & 0 & \quad 0 $\quad \{0\}$ & 0 \\ \hline
6.75348 & 0 & \quad 1 $\quad \{0\}$ & 0 \\ \hline 
6.76479 & 1 & \quad 2 $\quad \{2\}$  & 1 \\ \hline
6.77833 & 0 & \quad 0 $\quad \{0\}$ & 0 \\ \hline
6.77844 & 0 & \quad 0 $\quad \{0\}$ & 0 \\ \hline 
6.77905 & 1 & \quad 1 $\quad \{1\}$  & 1\\ \hline 
6.77908 & 0 &  \, \!\!\! $-1$ $\quad \{-1, 0\}$ & 1 \\ \hline 
6.77985 & 1 &  \, \!\!\! $-1$ $\quad \{-1\}$ & 1 \\ \hline 
6.78358 & 1 & \quad 0 $\quad \{0\}$ & 1 \\ \hline 
6.79342 & 0 & \quad 0 $\quad \{0\}$ & 0 \\ \hline
6.85091 & 1 & \quad 0 $\quad \{0\}$  & 1 \\ \hline
6.85103 & 0 & \, \!\!\! $-1$ $\quad \{0\}$ & 1 \\ \hline 
6.85613 & 1 & \quad 2 $\quad \{0, 2\}$  & 1 \\ \hline
\end{tabular}
\end{minipage}

\begin{minipage}{0.5\textwidth}
 \begin{tabular}{|c|c|l|c|} 
\hline
Virtual & Graded &\;\; Signatures & Slice \\ 
knot & genus & $ \quad \si  \quad \{\widehat{\si}^\pm_\om\}$
& genus\\ \hline 
\hline
6.85774  & 1 & \quad 2 $\quad \{0, 2\}$  & 1 \\ \hline 
6.87188  & 1 & \quad 2 $\quad \{0, 2\}$ & 1 \\ \hline 
6.87262  & 2 & \quad 2 $\quad \{0 ,2\}$  & 2 \\ \hline
6.87269  & 0 & \quad 0 $\quad \{0\}$ & 0 \\ \hline
6.87310  & 1 & \quad 4 $\quad \{2, 4\}$  & 2 \\ \hline
6.87319  & 0 & \quad 0 $\quad \{0\}$ & 0 \\ \hline 
6.87369  & 1 & \,\!\! $-2$ $\quad \{-2, 0\}$ & 1 \\ \hline 
6.87548  & 0 & \quad 0 $\quad \{0\}$ & 1 \\ \hline 
6.87846  & 1 & \quad 2 $\quad \{0, 2\}$ & 1 \\ \hline 
6.87857  & 1 & \quad 0 $\quad \{0\}$ & 1 \\ \hline 
6.87859  & 1 & \quad 0 $\quad \{-2, 0\}$ & 1 \\ \hline
6.87875  & 1 & \,\!\! $-1$ $\quad \{-2, 0\}$  & 1 \\ \hline 
6.89156  & 1 & \quad 2 $\quad \{0, 2\}$ & 1 \\ \hline 
{\bf 6.89187} & 0 & \quad 4 $\quad \{0, 4 \}$ & 2 \\ \hline 
{\bf 6.89198} & 0 & \quad 0 $\quad \{0\}$ & 0 \\ \hline
6.89623  & 0 & \quad 0 $\quad \{0\}$ & 0 \\ \hline
6.89812  & 1 & \quad 2 $\quad \{0, 2\}$ & 1 \\ \hline 
6.89815  & 0 & \quad 0 $\quad \{0\}$ & 0 \\ \hline 
6.90099  & 2 & \quad 2 $\quad \{0, 2\}$  & 1 \\ \hline 
6.90109  & 1 & \quad 4 $\quad \{0, 2, 4\}$ & 2 \\ \hline 
6.90115  & 0 & \quad 2 $\quad \{0, 2\}$  & 1 or 2 \\ \hline 
6.90139  & 1 & \quad 4 $\quad \{0, 2, 4\}$ & 2 \\ \hline
6.90146  & 0 & \quad 2 $\quad \{0, 2\}$ & 1 or 2 \\ \hline 
6.90147  & 2 & \quad 4 $\quad \{2, 4\}$ & 2 \\ \hline 
6.90150  & 0 & \quad 2 $\quad \{0, 2\}$ & 1 or 2 \\ \hline 
6.90167  & 1 & \quad 4 $\quad \{0, 2, 4 \}$  & 2 \\ \hline 
{\bf 6.90172} & 0 & \quad 0 $\quad \{0\}$ & 1 \\ \hline 
6.90185  & 2 & \quad 4 $\quad \{0, 2, 4\}$  & 2 \\ \hline 
6.90194  & 0 & \quad 0 $\quad \{0\}$ & 1 or 2 \\ \hline
6.90195  & 1 & \quad 4 $\quad \{0, 2, 4\}$  & 2 \\ \hline
{\bf 6.90209} & 0 & \quad 2 $\quad \{0, 2\}$ & 1 \\ \hline
6.90214  & 1 & \quad 2 $\quad \{0, 2\}$& 1 \\ \hline 
6.90217  & 1 & \quad 2 $\quad \{0, 2\}$ & 1 \\ \hline
6.90219  & 1 & \quad 2 $\quad \{0, 2\}$ & 1 \\ \hline
{\bf 6.90227} & 0 & \quad 0 $\quad \{0\}$ & 0 \\ \hline 
6.90228  & 1 & \quad 2 $\quad \{0, 2\}$ & 1 \\ \hline
6.90232  & 1 & \quad 0 $\quad \{0\}$ & 1 \\ \hline 
6.90235  & 0 & \,\!\! $-2$ $\quad \{-2, 0, 2\}$ & 1 \\ \hline 

\end{tabular}
\end{minipage}
\end{tabular}
\vspace{5mm}
\caption{The graded genus, signature, $\om$-signatures, and slice genus of almost classical knots. Boldface font is used for classical knots.}
\label{table-2}
\end{table}

 \newpage
 
 {\tiny

\begin{figure}[H]
\caption{{\sc Diagrams of almost classical knots in surfaces}}
\label{ACknot-diagrams}
\bigskip
\def\svgwidth{2.9cm}   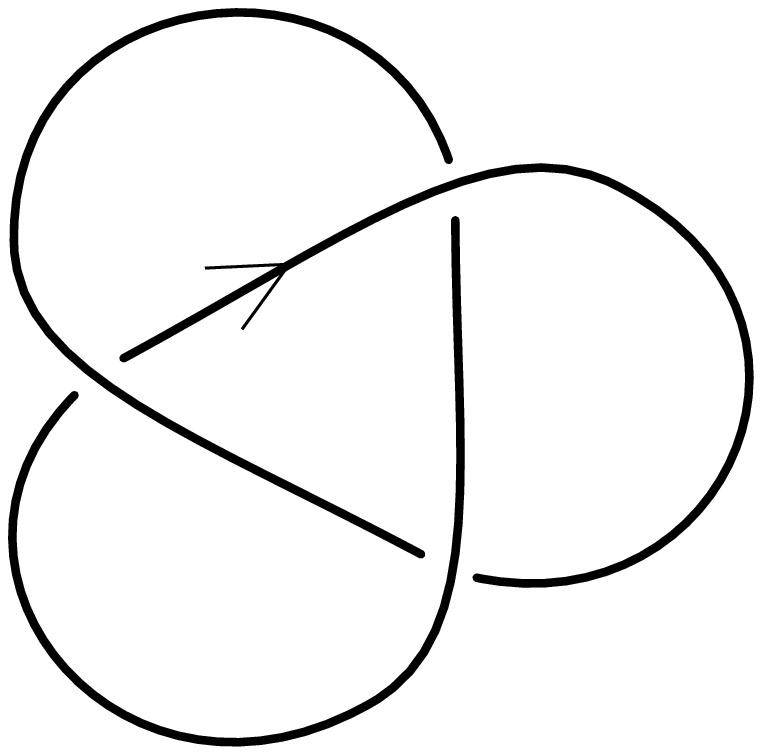 
\def\svgwidth{2.9cm}   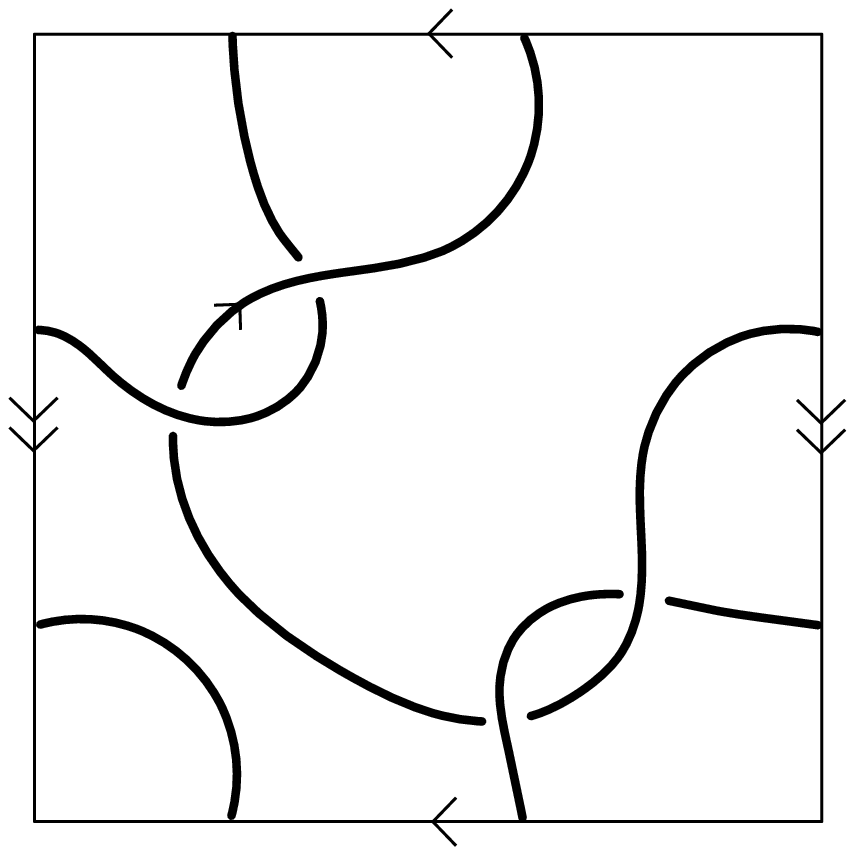 
\def\svgwidth{2.9cm}   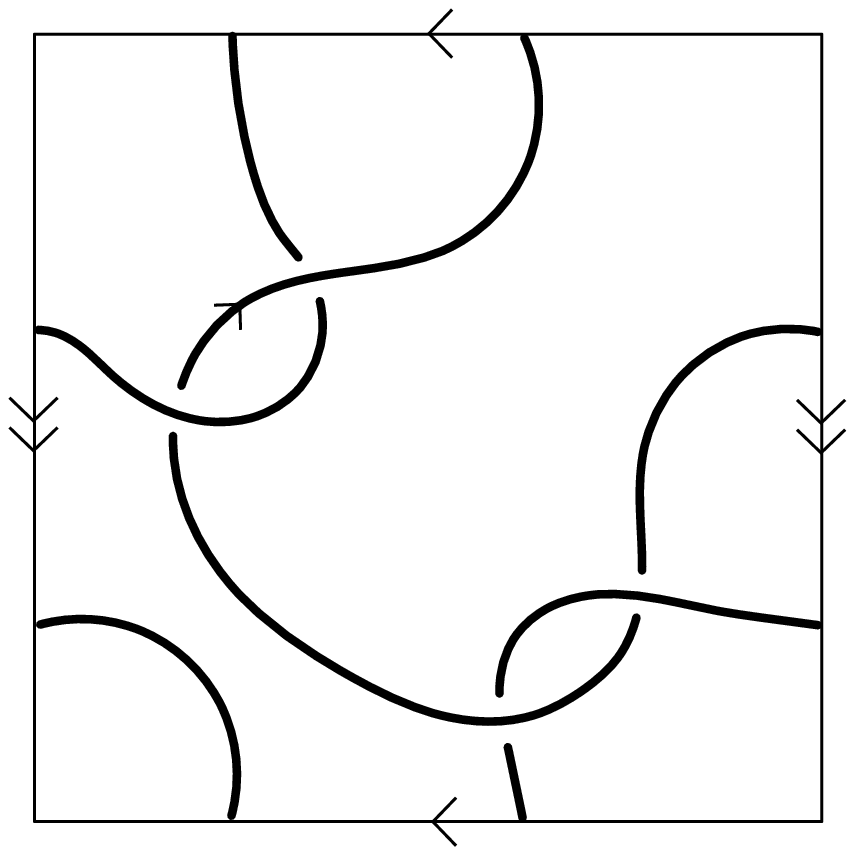
\def\svgwidth{2.9cm}   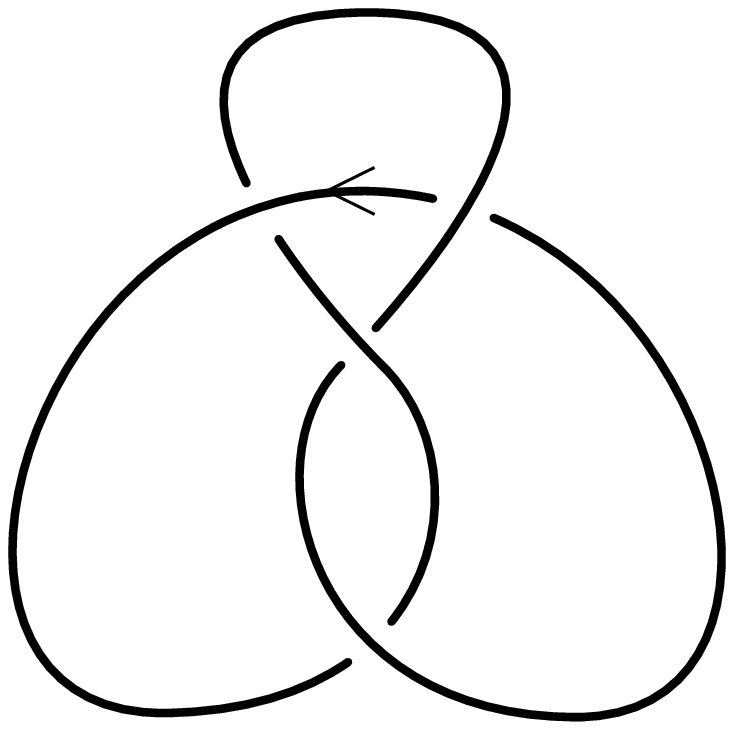
\def\svgwidth{2.9cm}   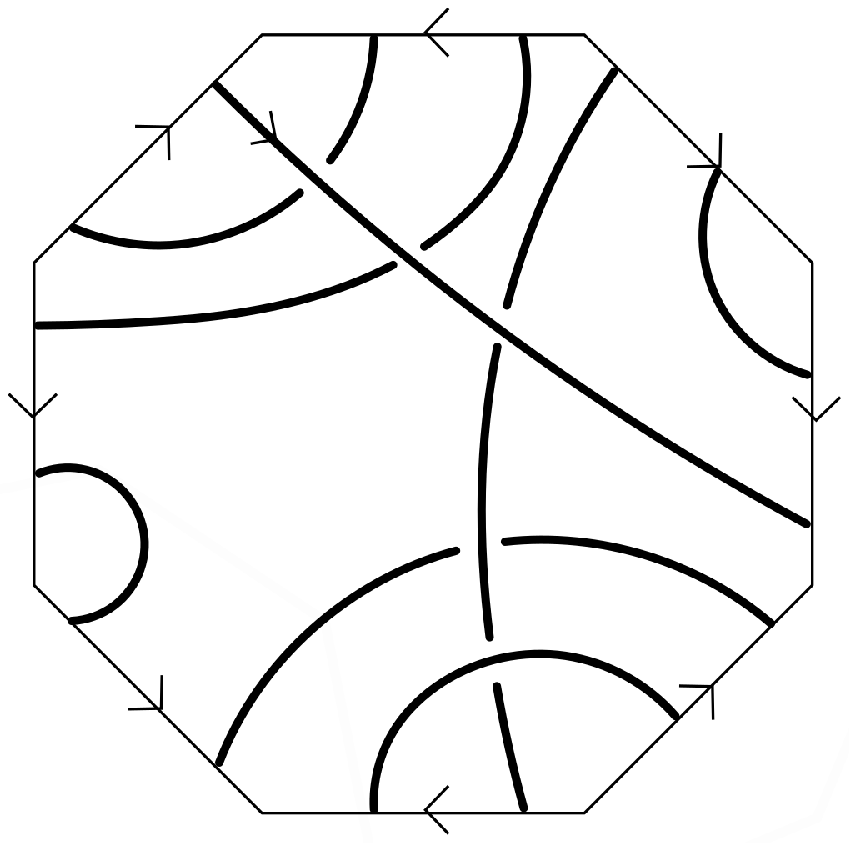 

\def\svgwidth{2.9cm}   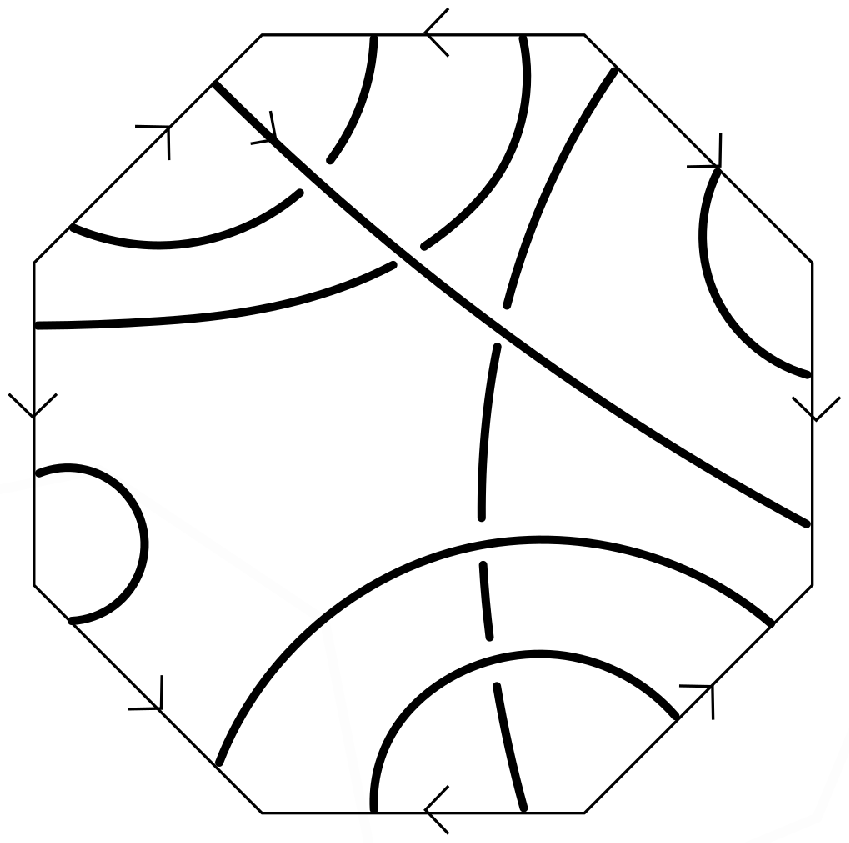
\def\svgwidth{2.9cm}   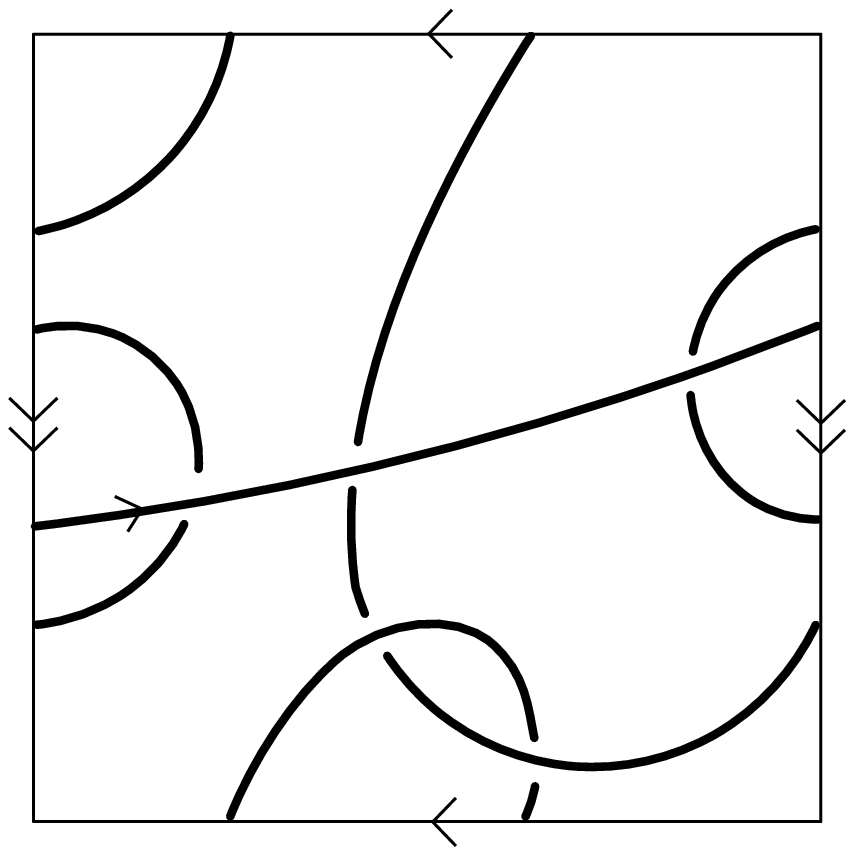
\def\svgwidth{2.9cm}   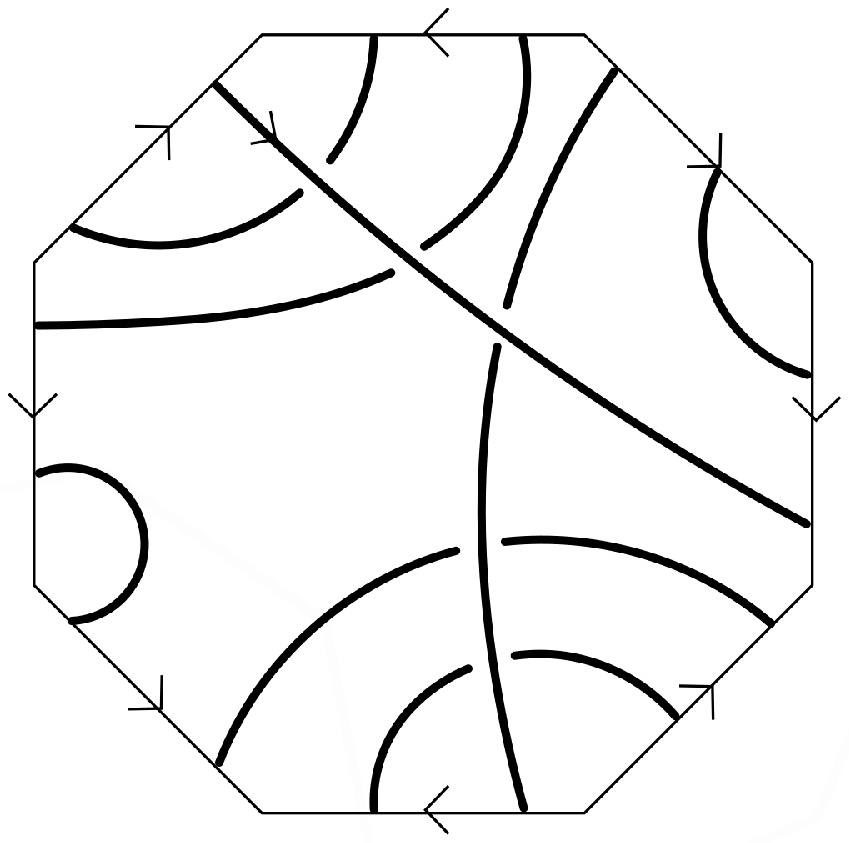
\def\svgwidth{2.9cm}   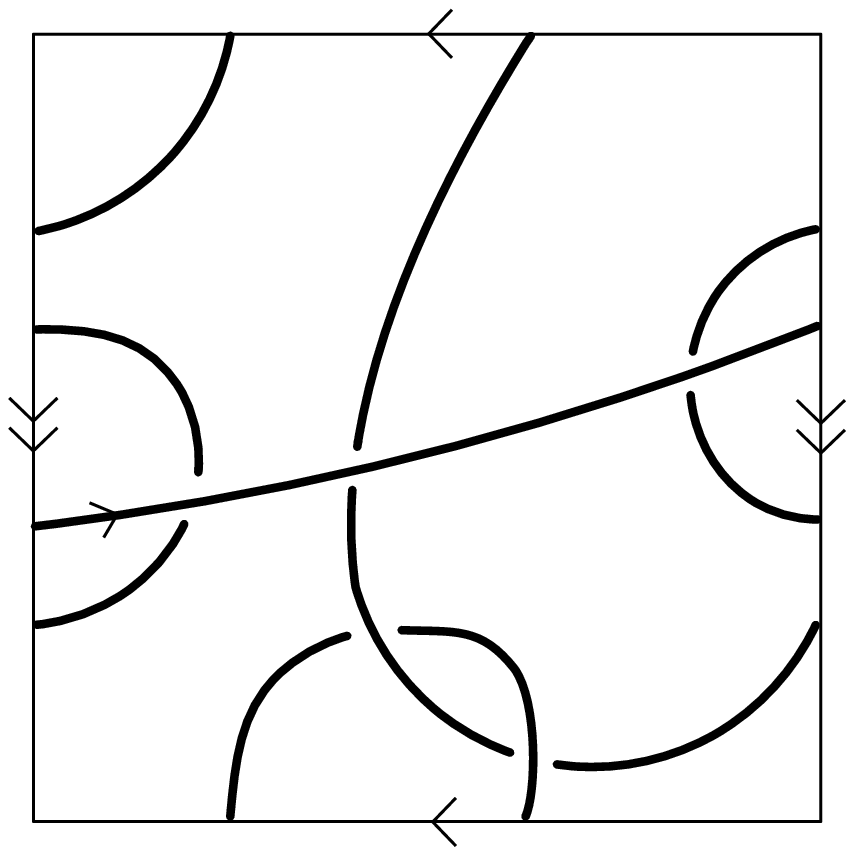
\def\svgwidth{2.9cm}   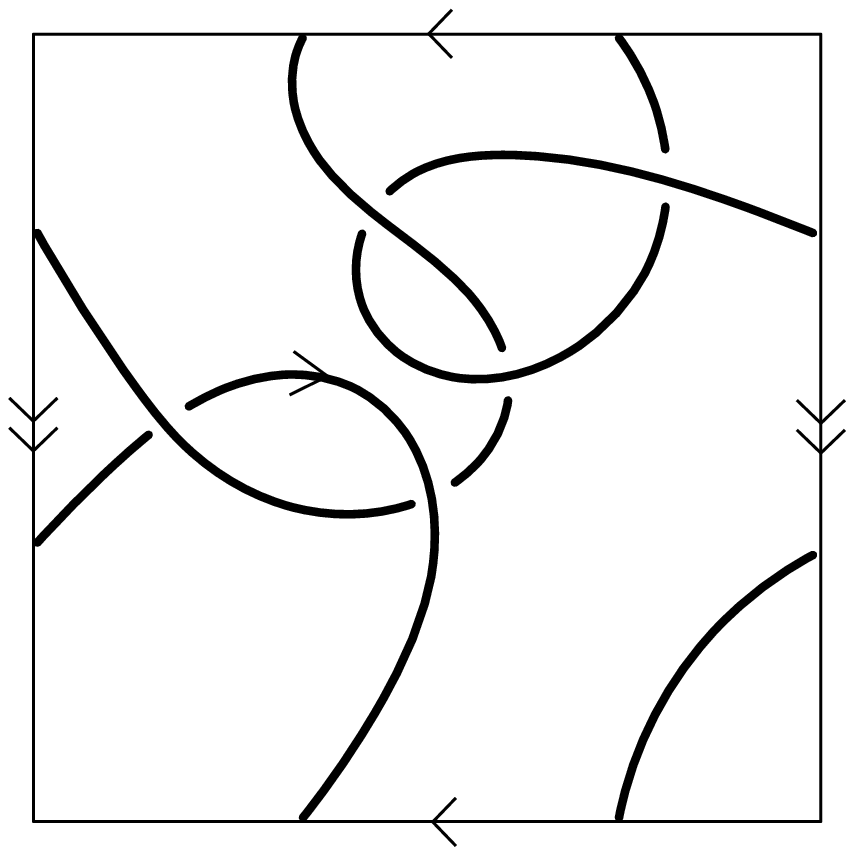

\def\svgwidth{2.9cm}   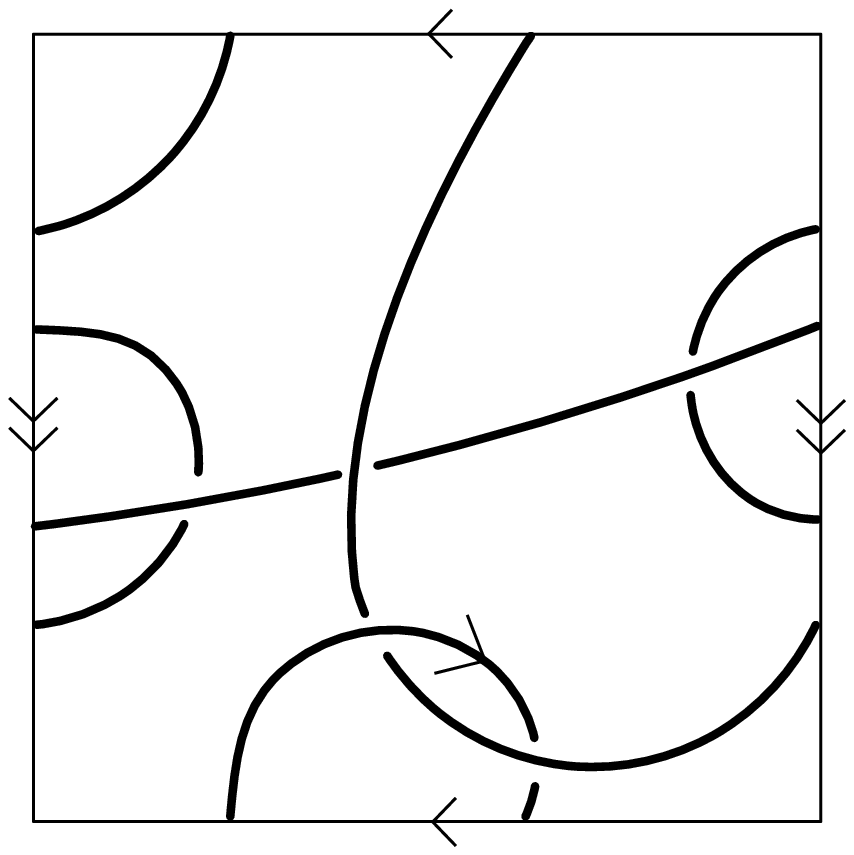
\def\svgwidth{2.9cm}   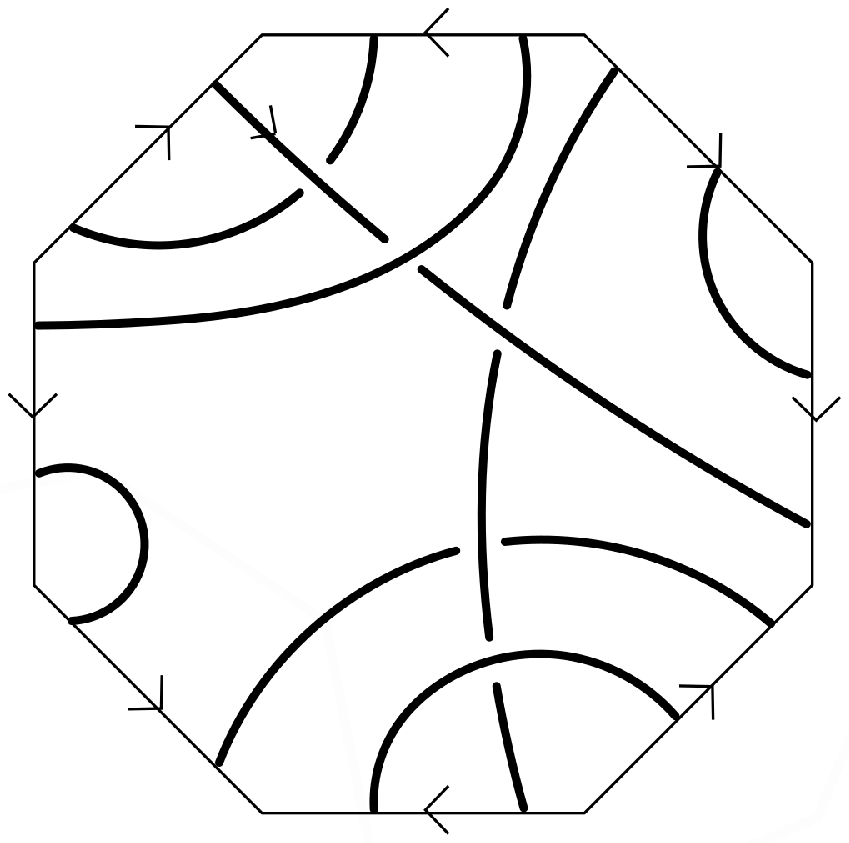
\def\svgwidth{2.9cm}   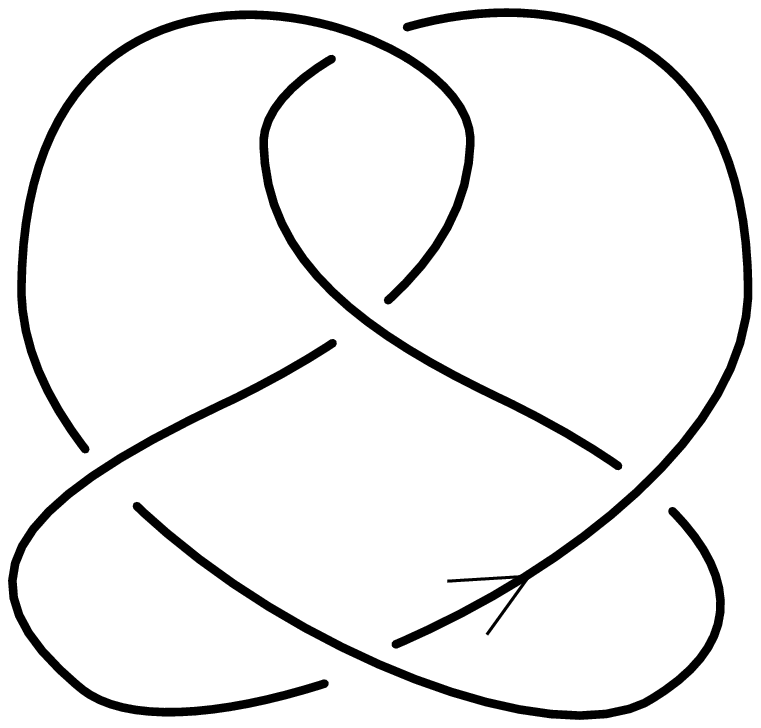 
\def\svgwidth{2.9cm}   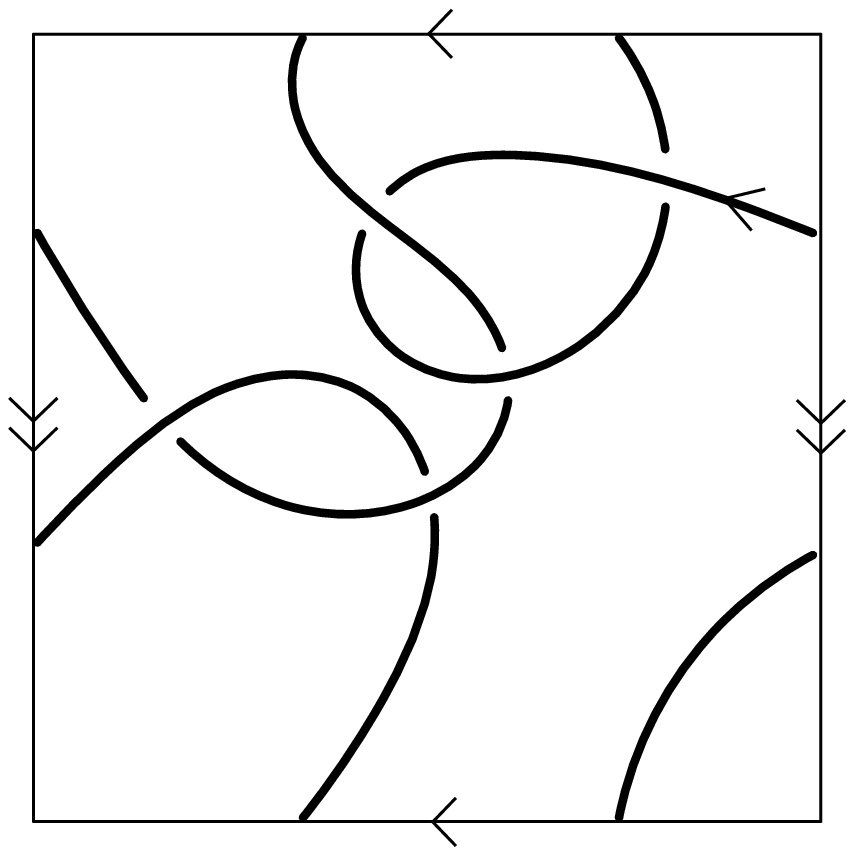
\def\svgwidth{2.9cm}   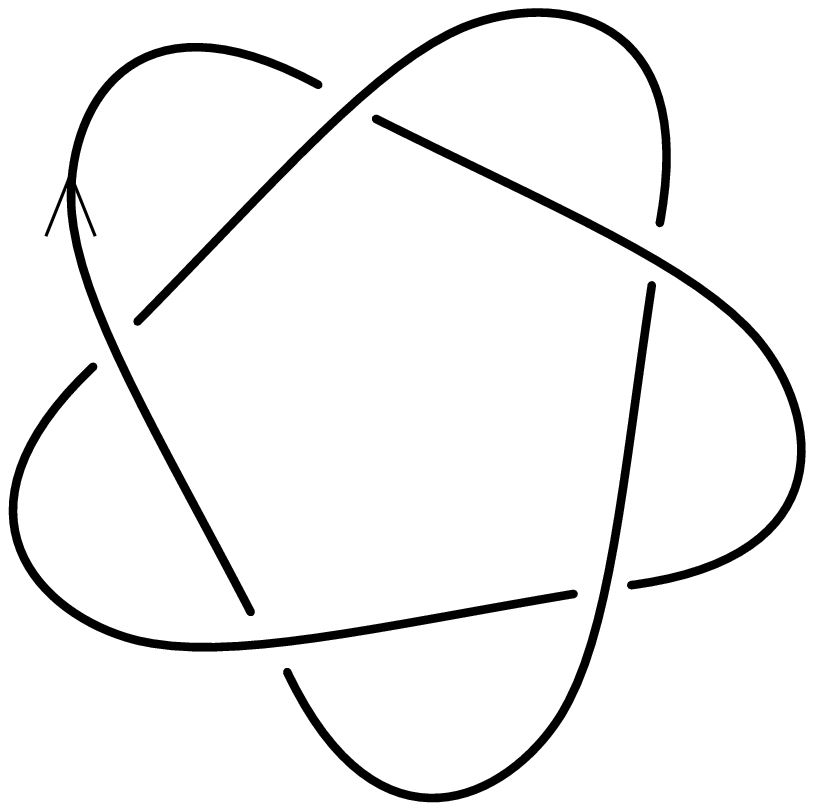 

\def\svgwidth{2.9cm}    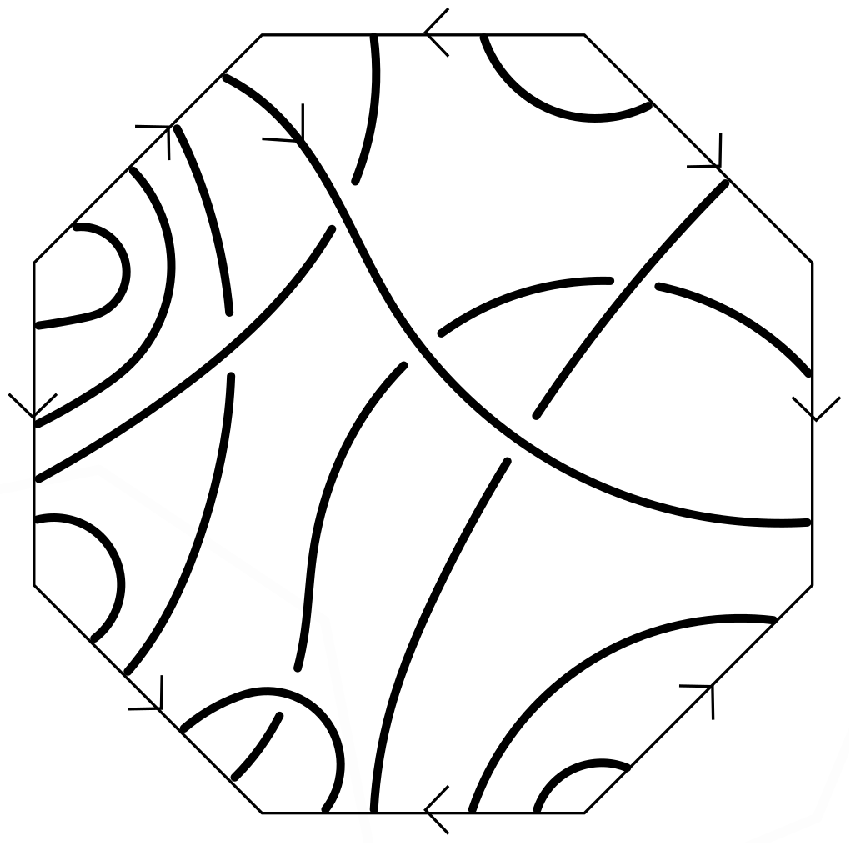
\def\svgwidth{2.9cm}    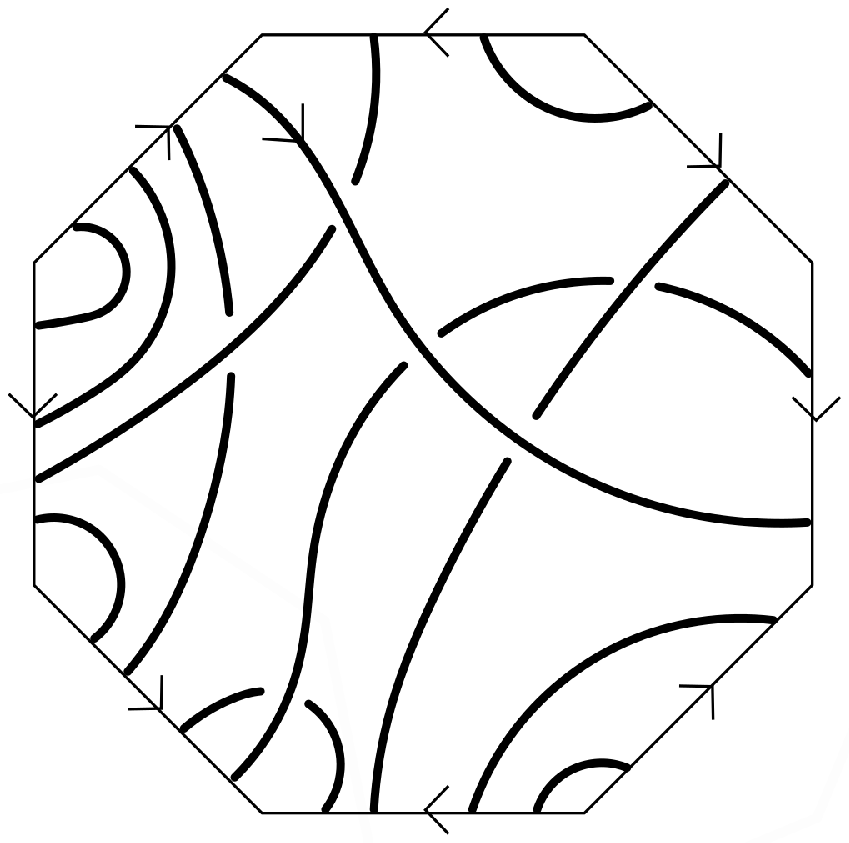
\def\svgwidth{2.9cm}   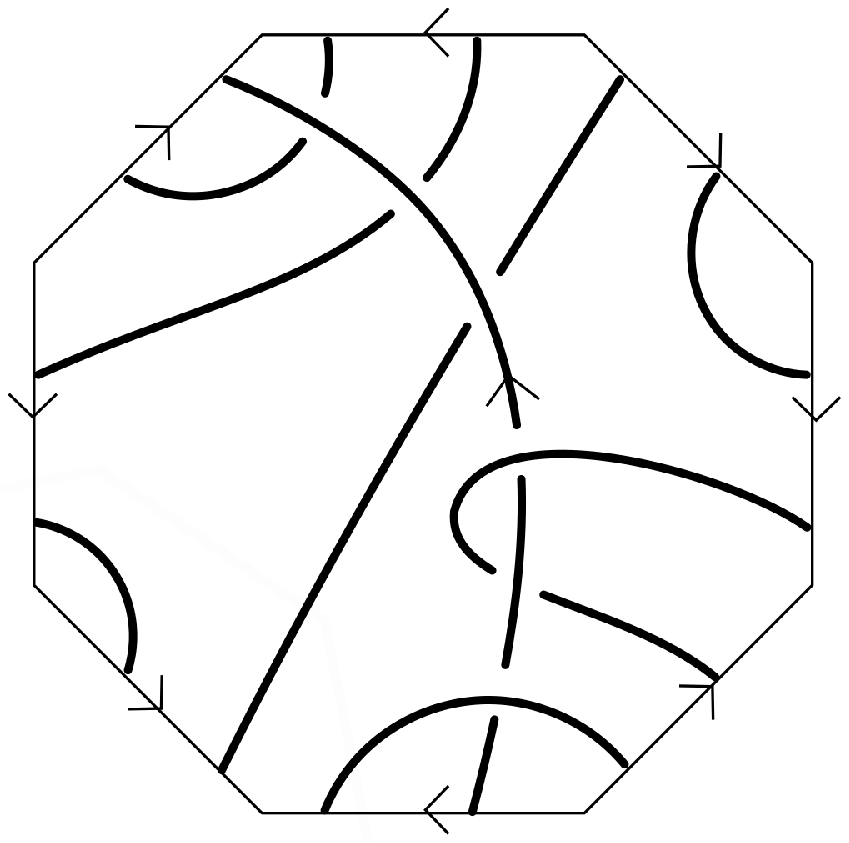
\def\svgwidth{2.9cm}   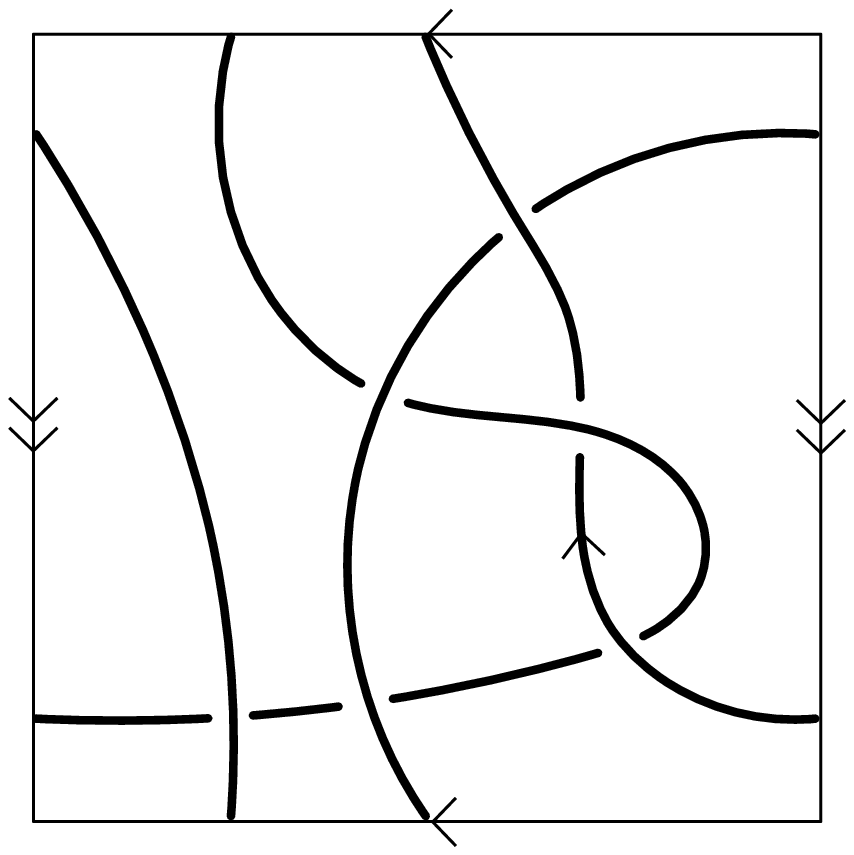
\def\svgwidth{2.9cm}    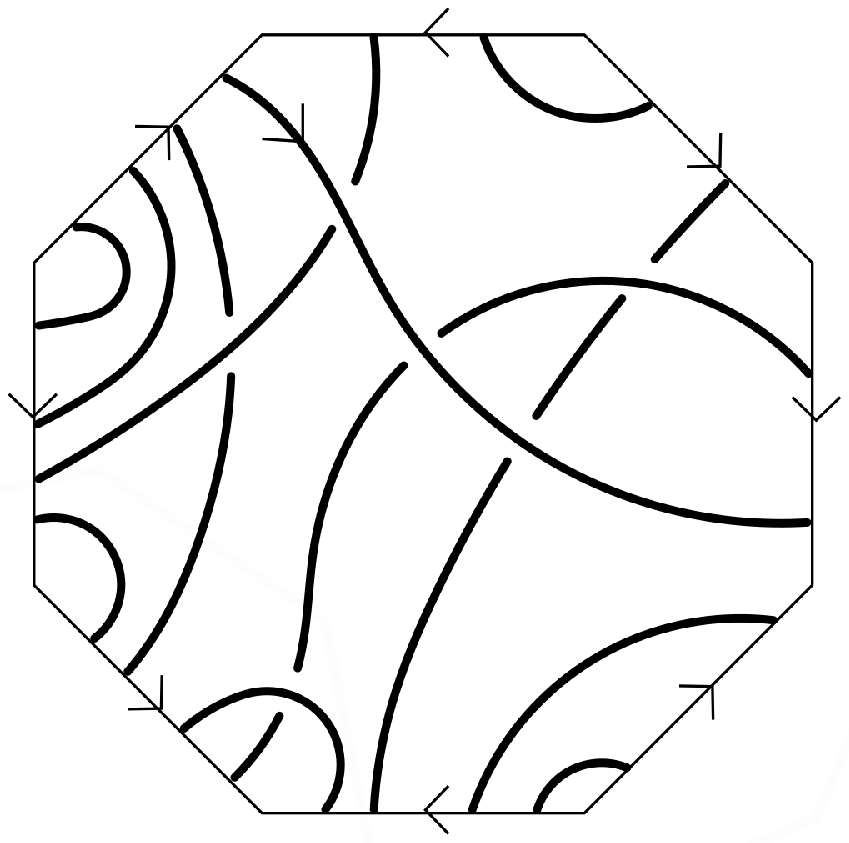

\def\svgwidth{2.9cm}    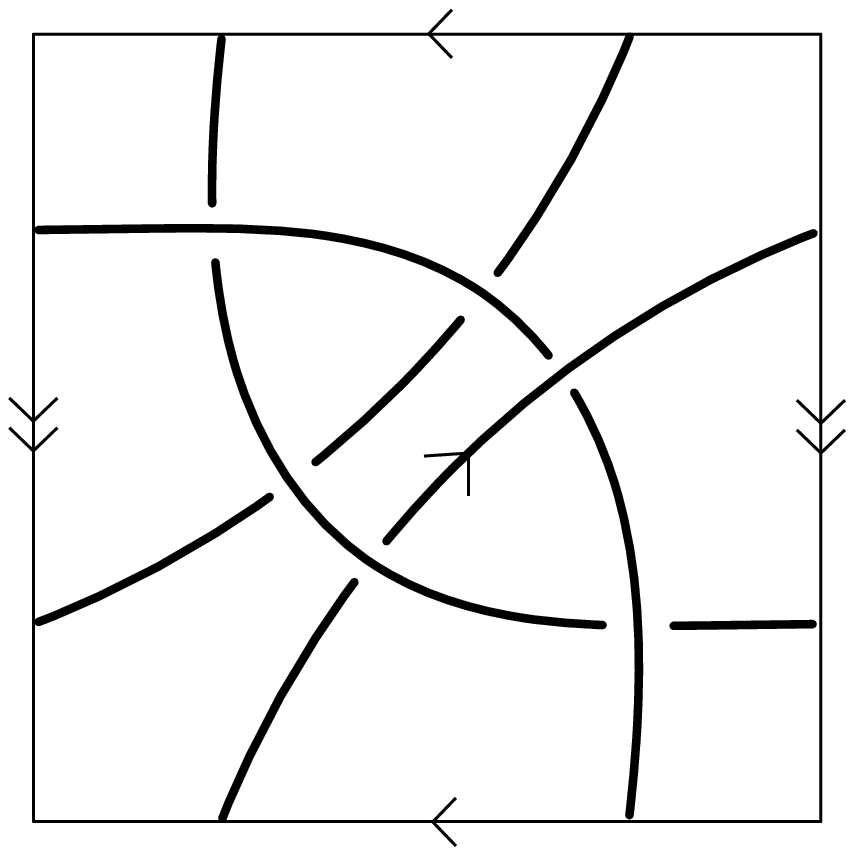
\def\svgwidth{2.9cm}    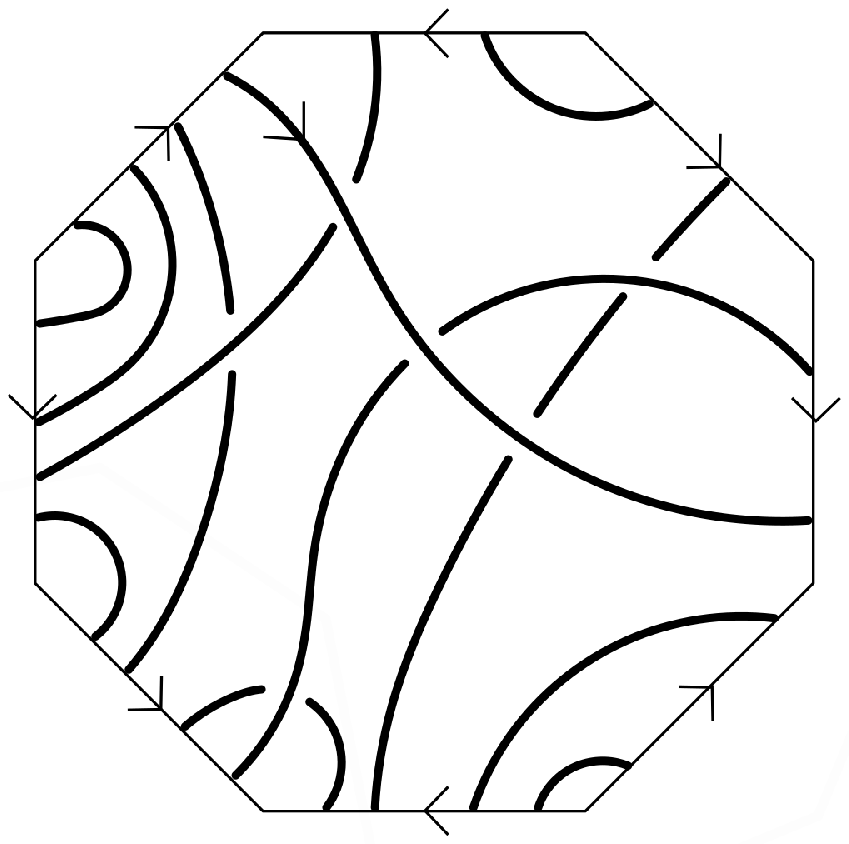
\def\svgwidth{2.9cm}    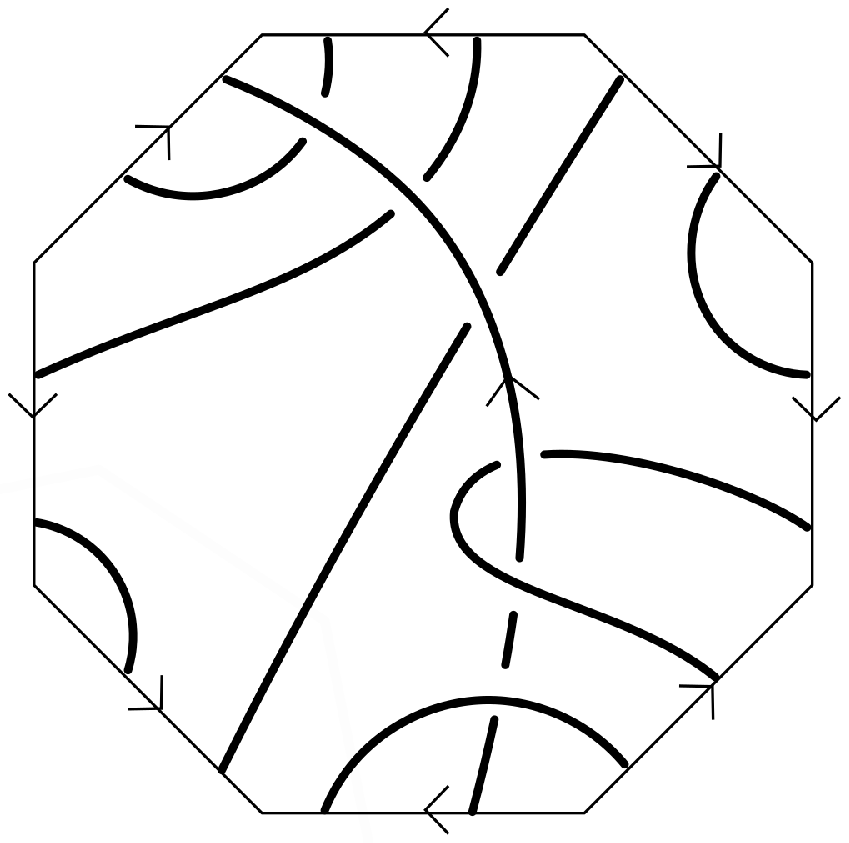
\def\svgwidth{2.9cm}    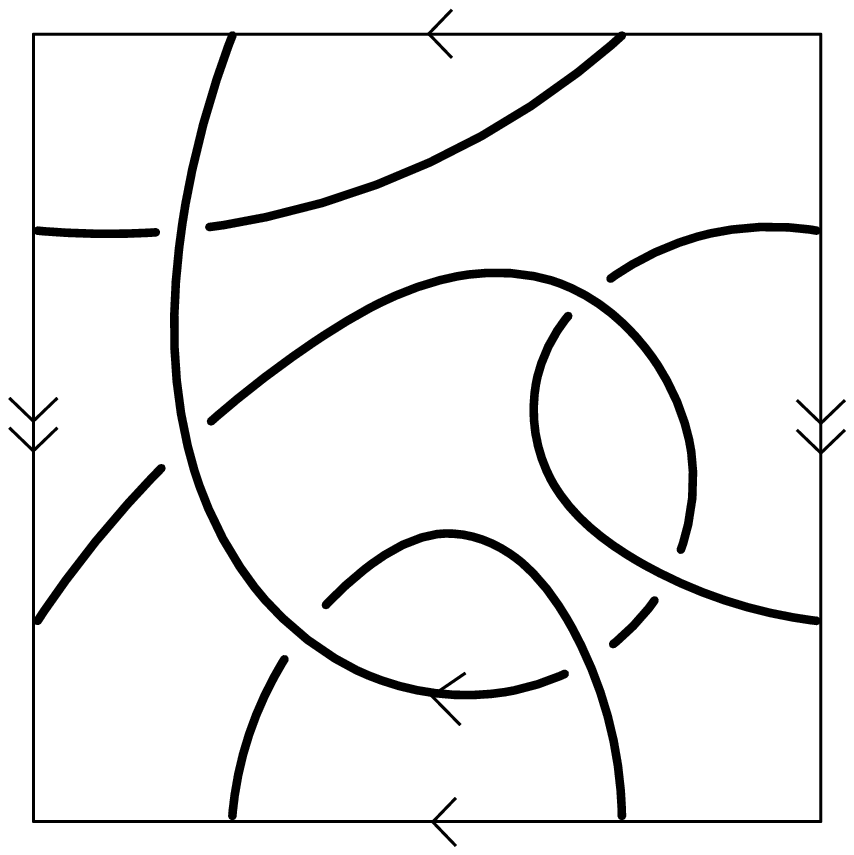
\def\svgwidth{2.9cm}    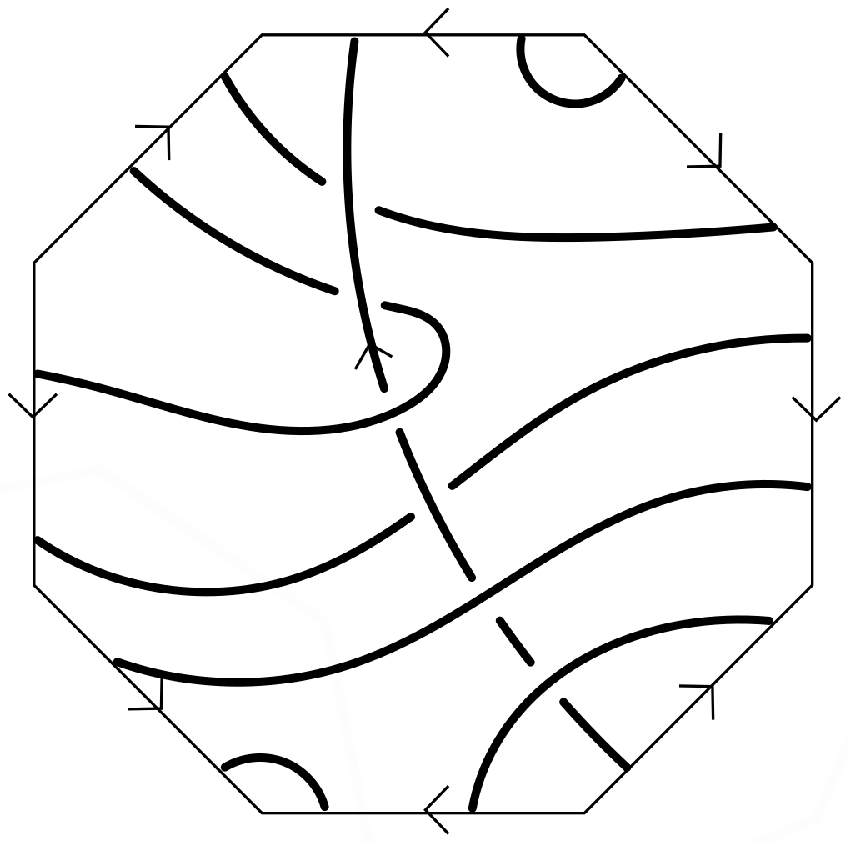

\def\svgwidth{2.9cm}   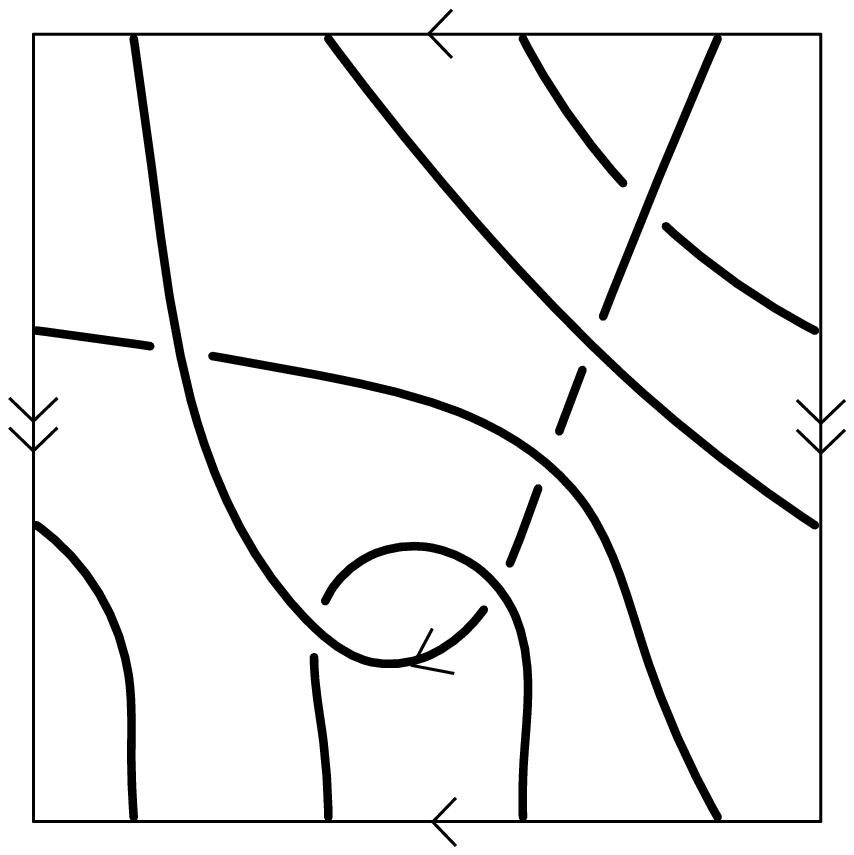
\def\svgwidth{2.9cm}  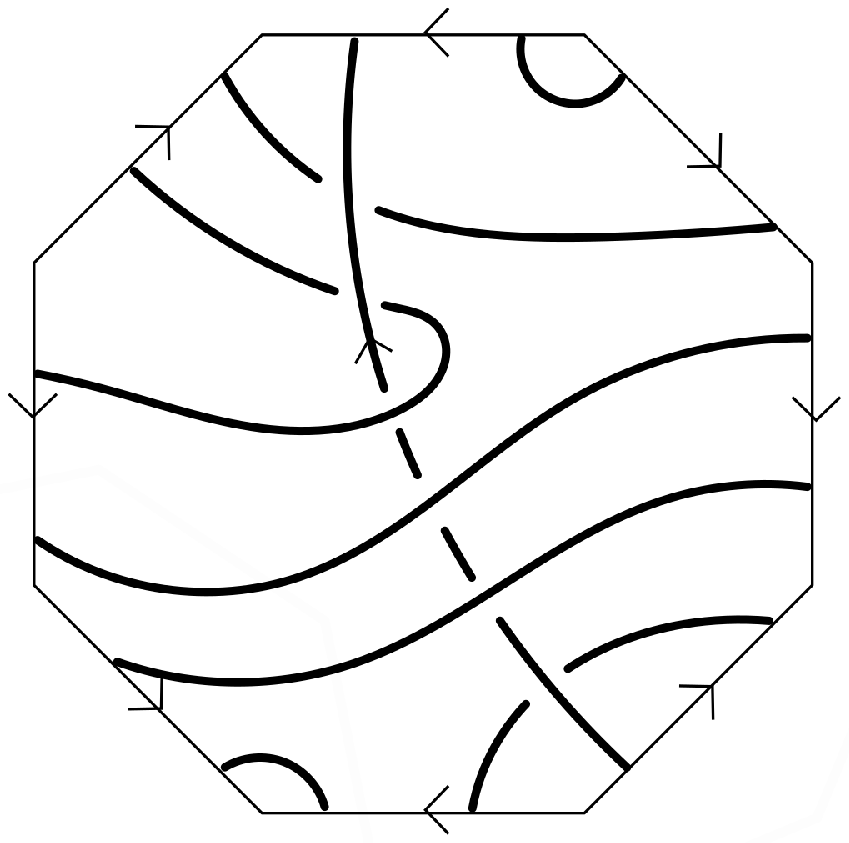
\def\svgwidth{2.9cm}  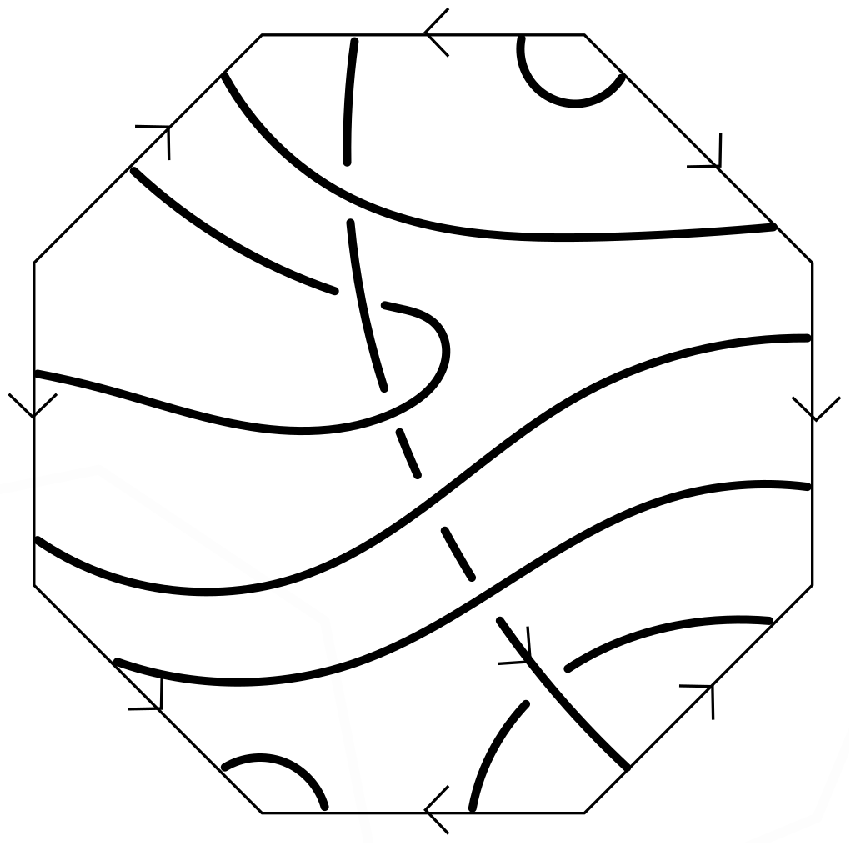
\def\svgwidth{2.9cm}  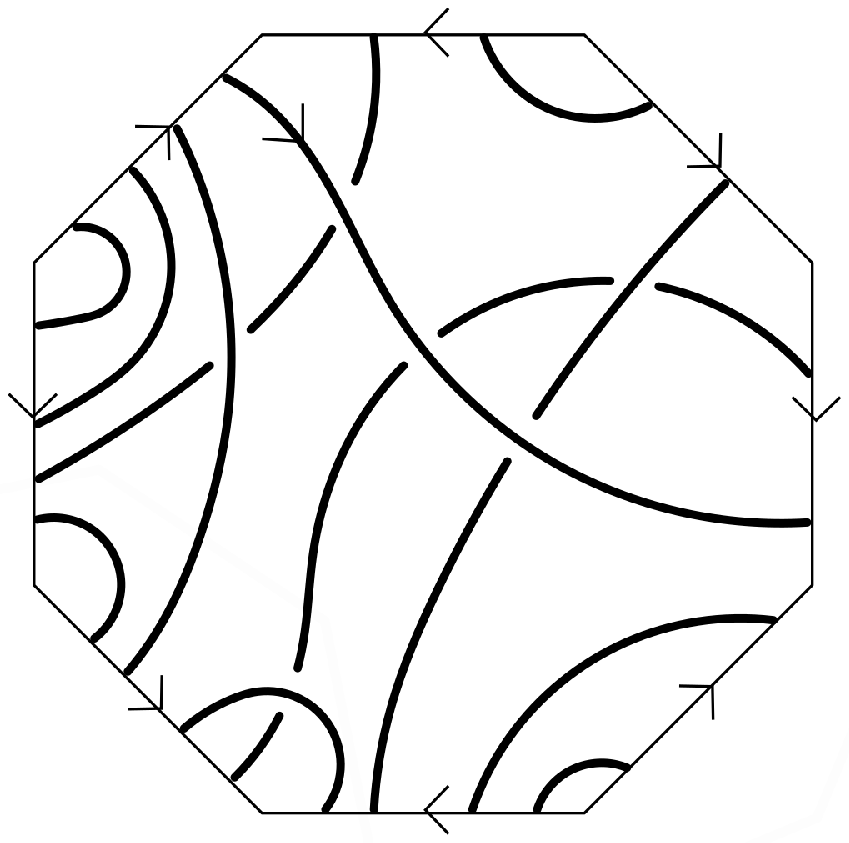
\def\svgwidth{2.9cm}  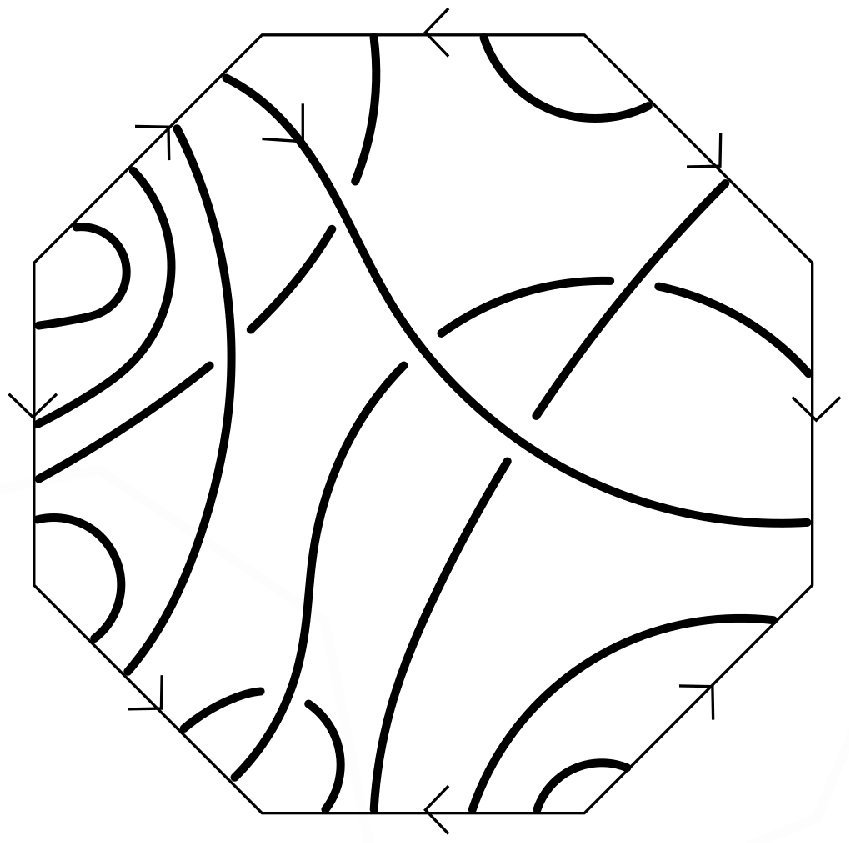

\def\svgwidth{2.9cm}   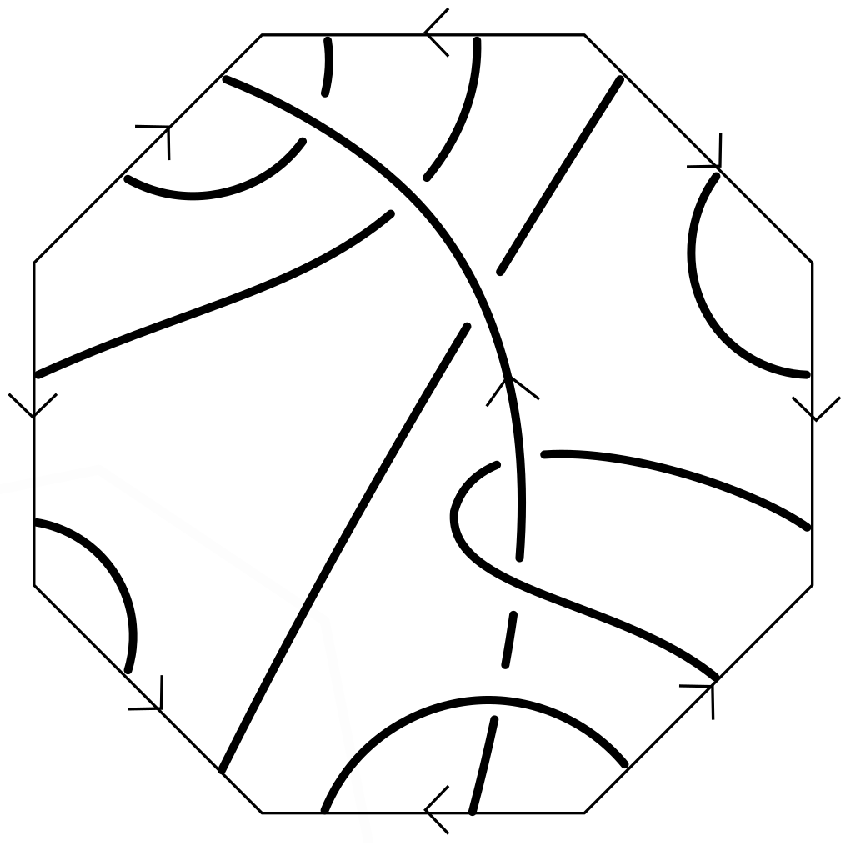
\def\svgwidth{2.9cm}  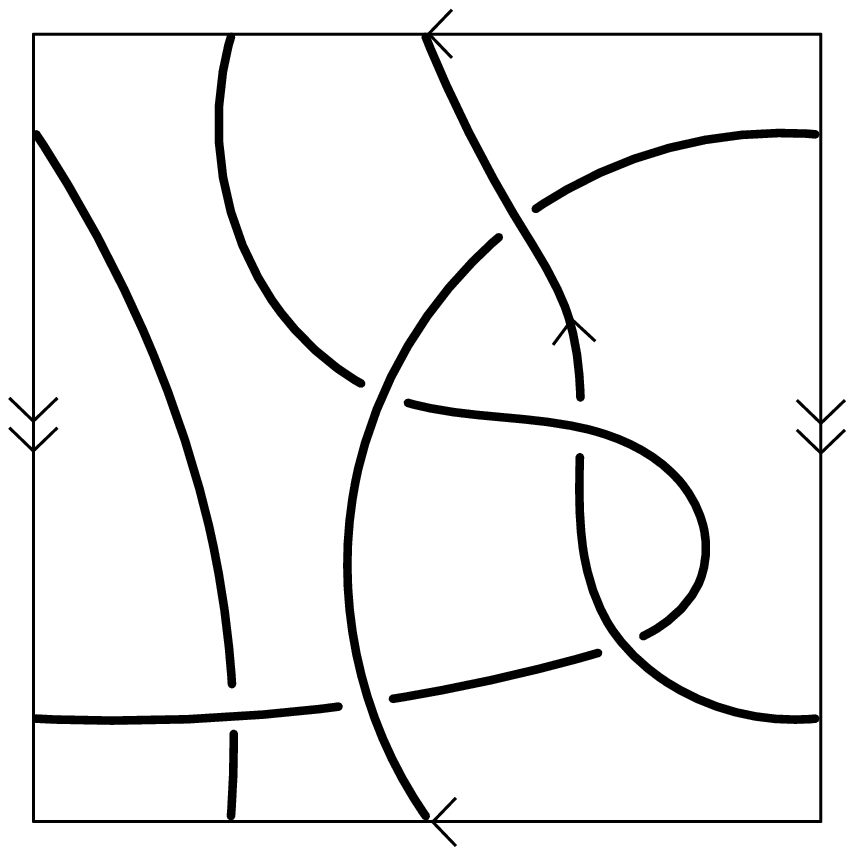
\def\svgwidth{2.9cm}   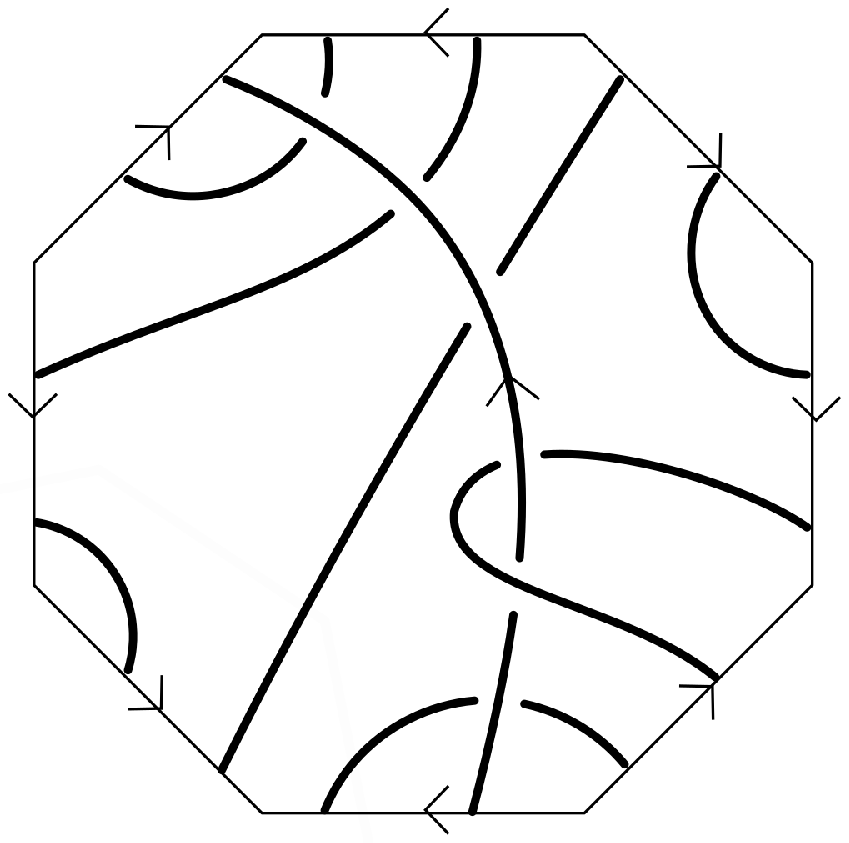
\def\svgwidth{2.9cm}  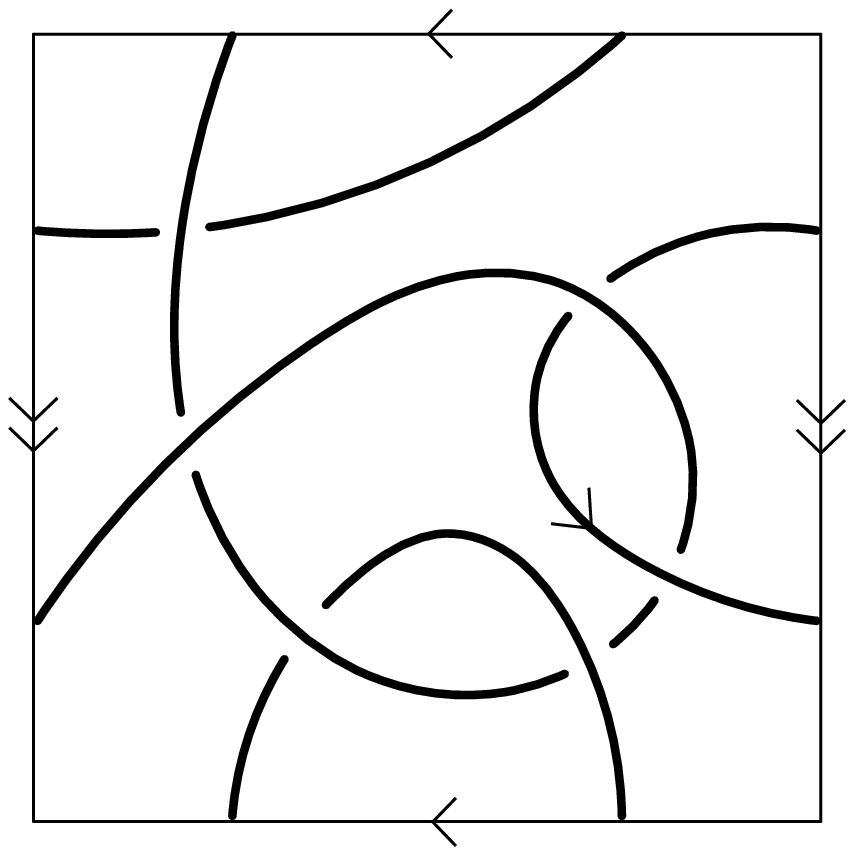
\def\svgwidth{2.9cm}   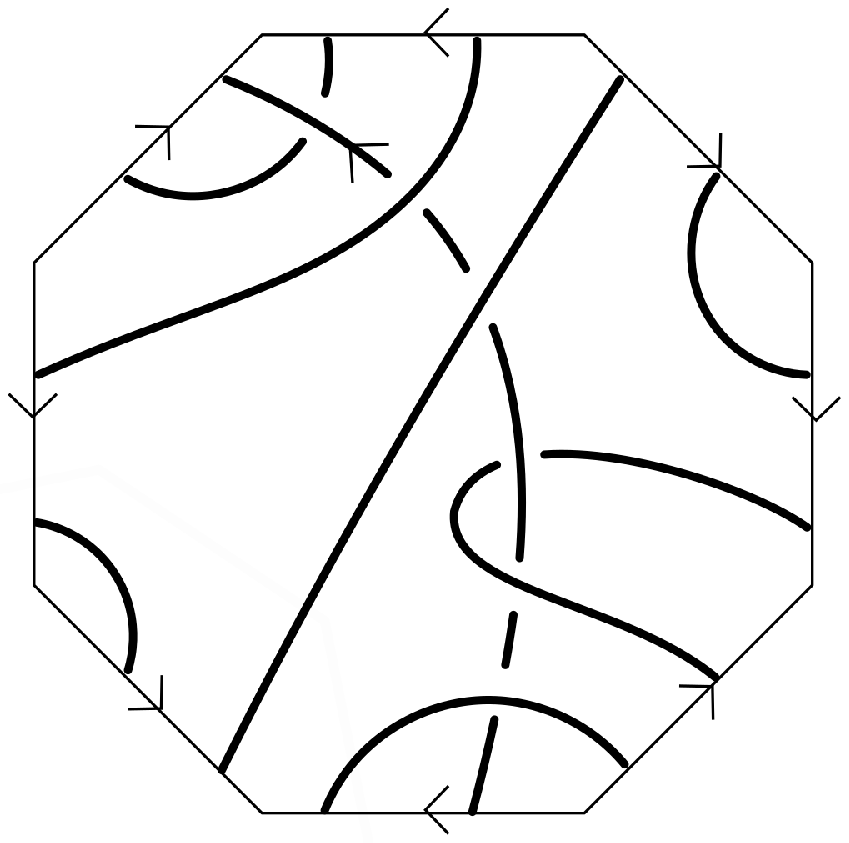

\end{figure}

\begin{figure}[H]
\def\svgwidth{2.9cm}  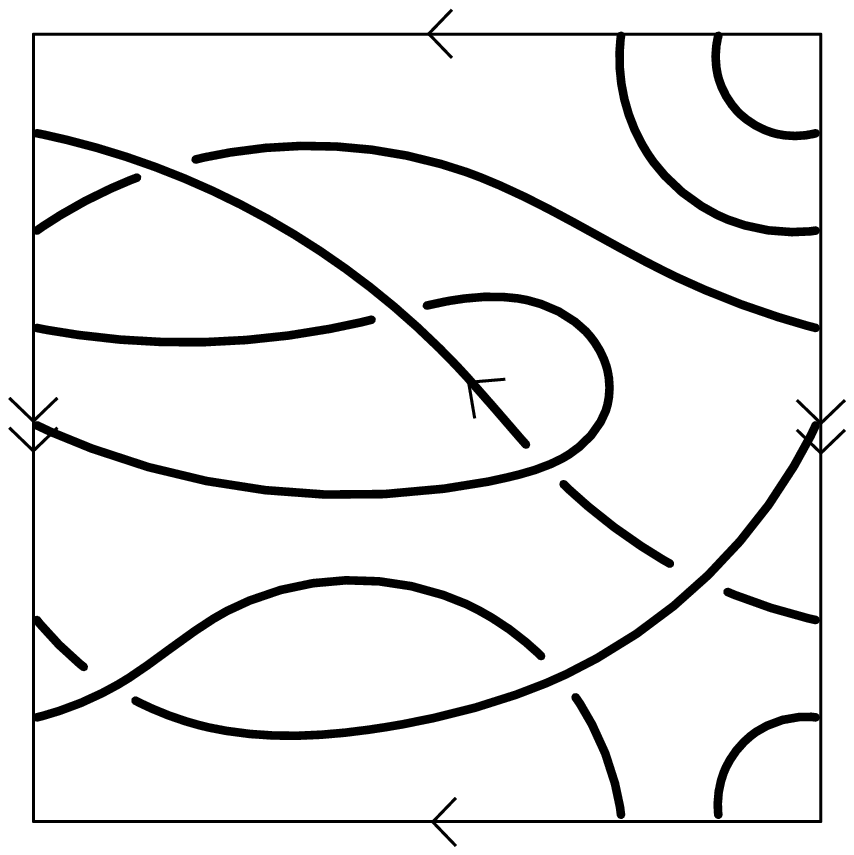
\def\svgwidth{2.9cm}  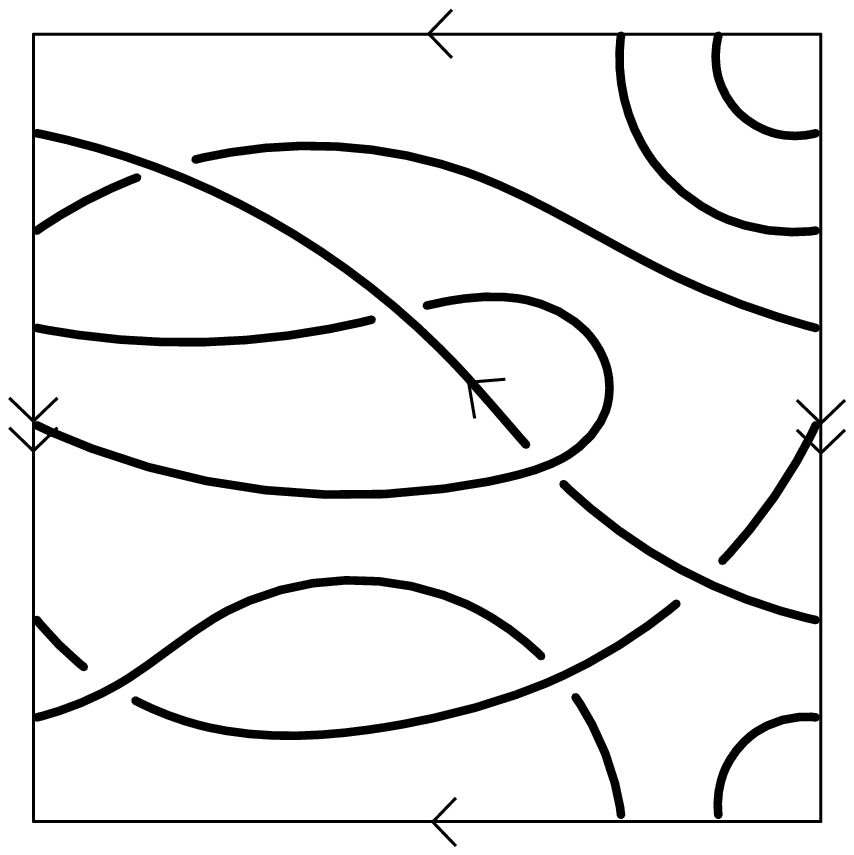
\def\svgwidth{2.9cm}   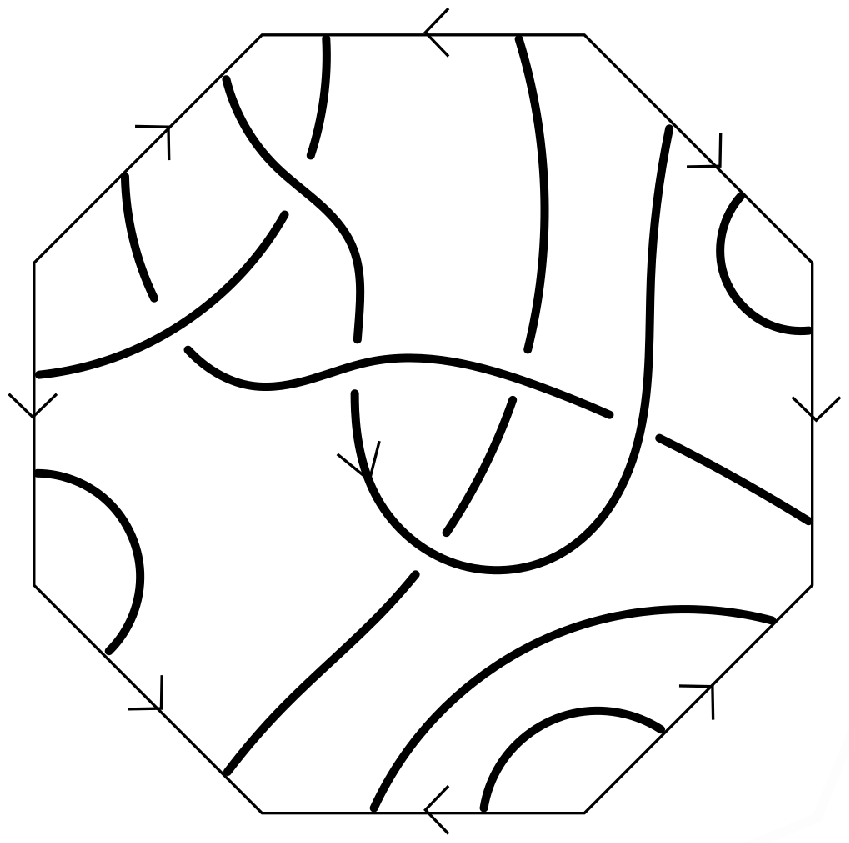
\def\svgwidth{2.9cm}   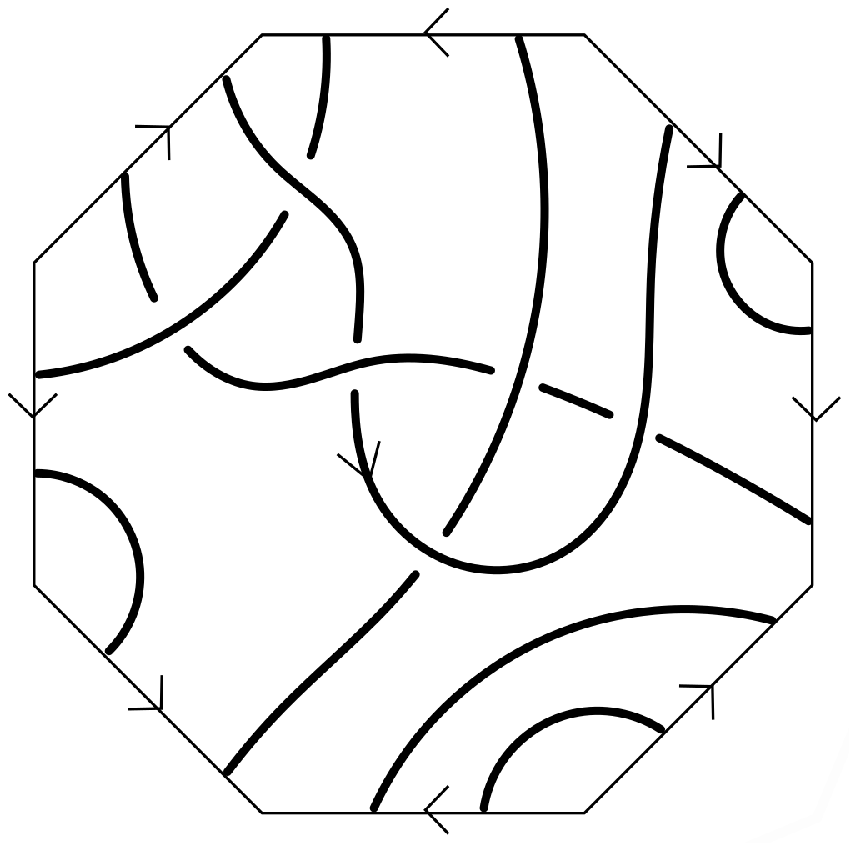
\def\svgwidth{2.9cm}   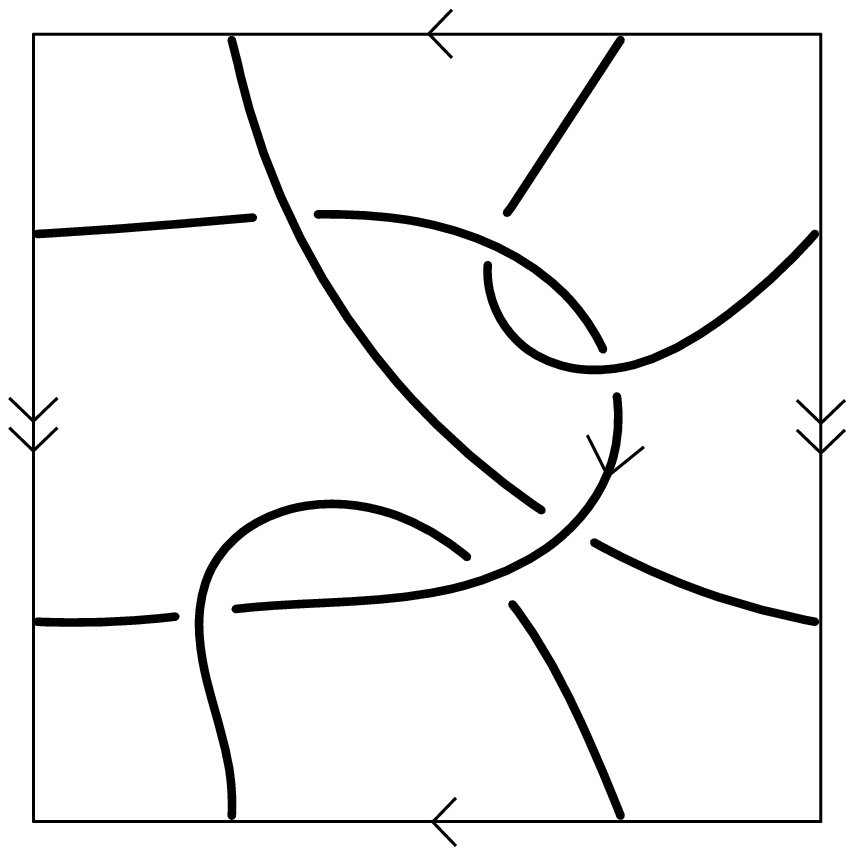

\def\svgwidth{2.9cm}  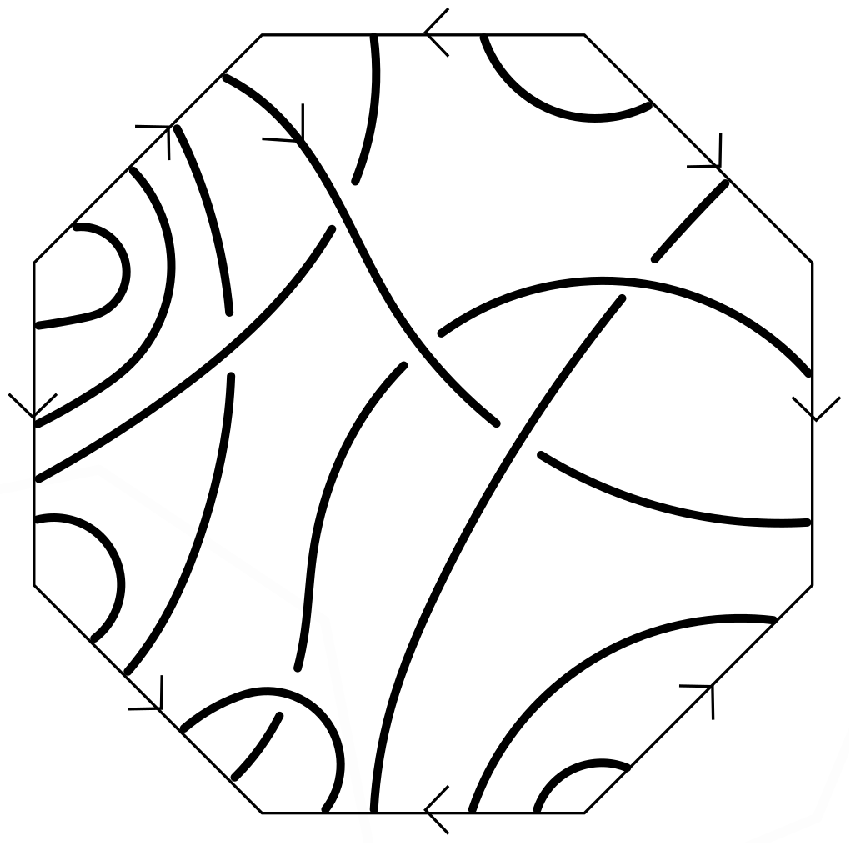
\def\svgwidth{2.9cm}   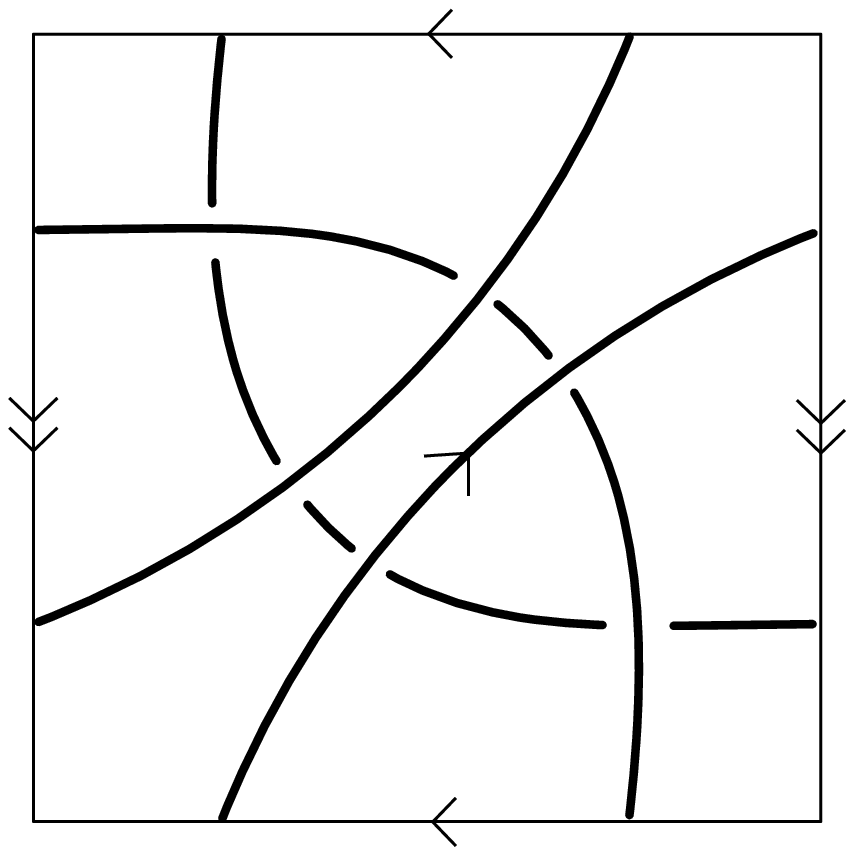
\def\svgwidth{2.9cm}  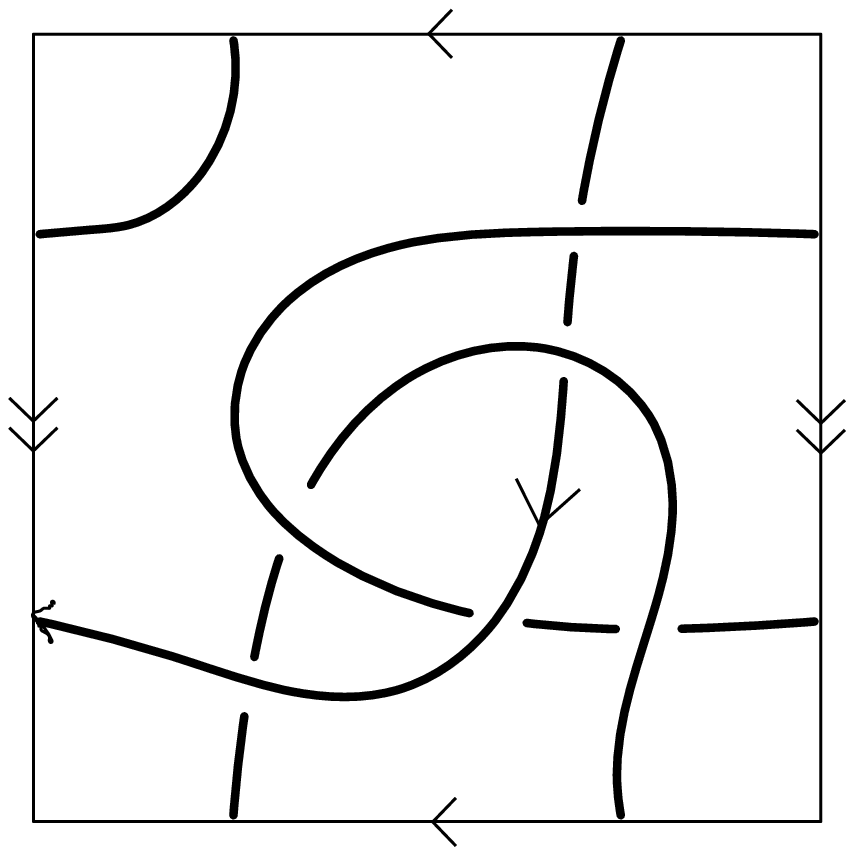
\def\svgwidth{2.9cm}   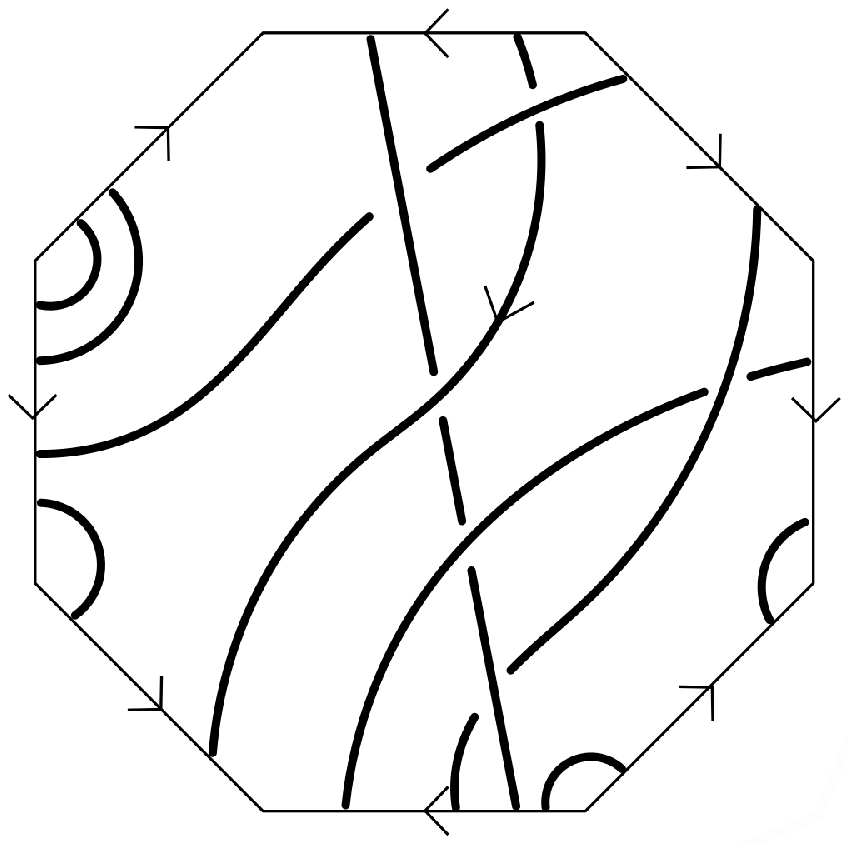
\def\svgwidth{2.9cm}   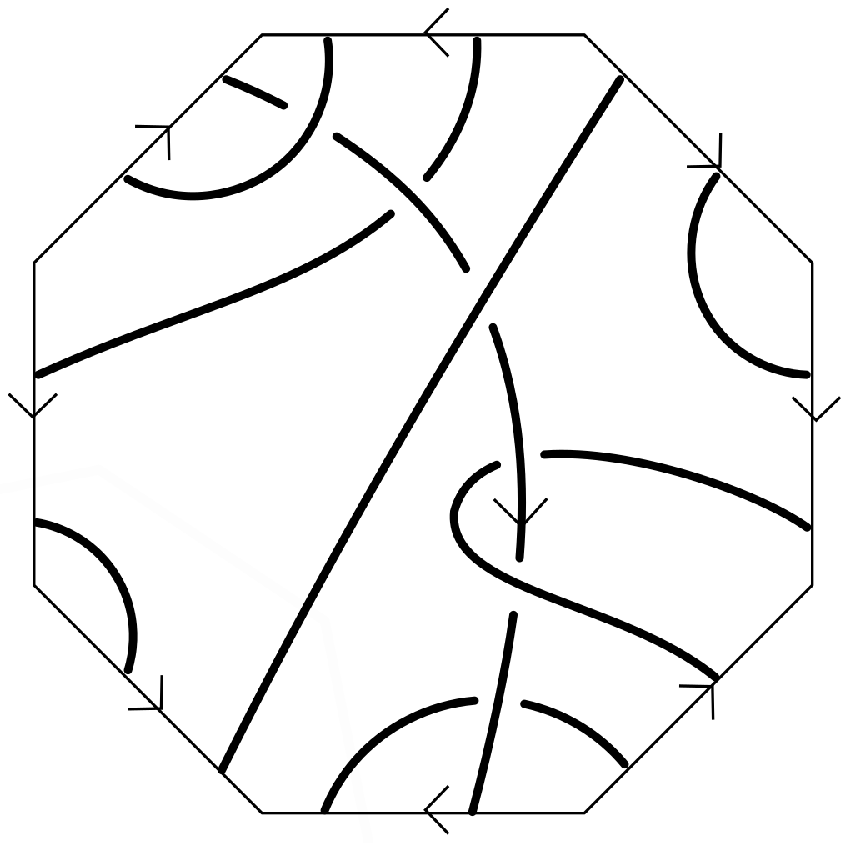

\def\svgwidth{2.9cm}   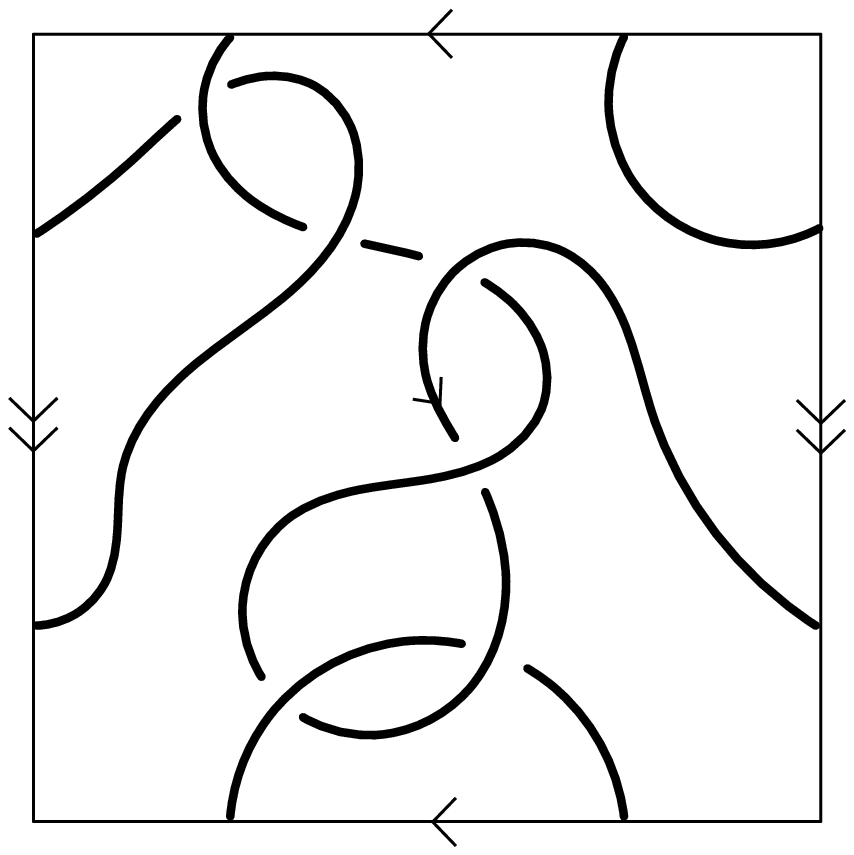
\def\svgwidth{2.9cm}  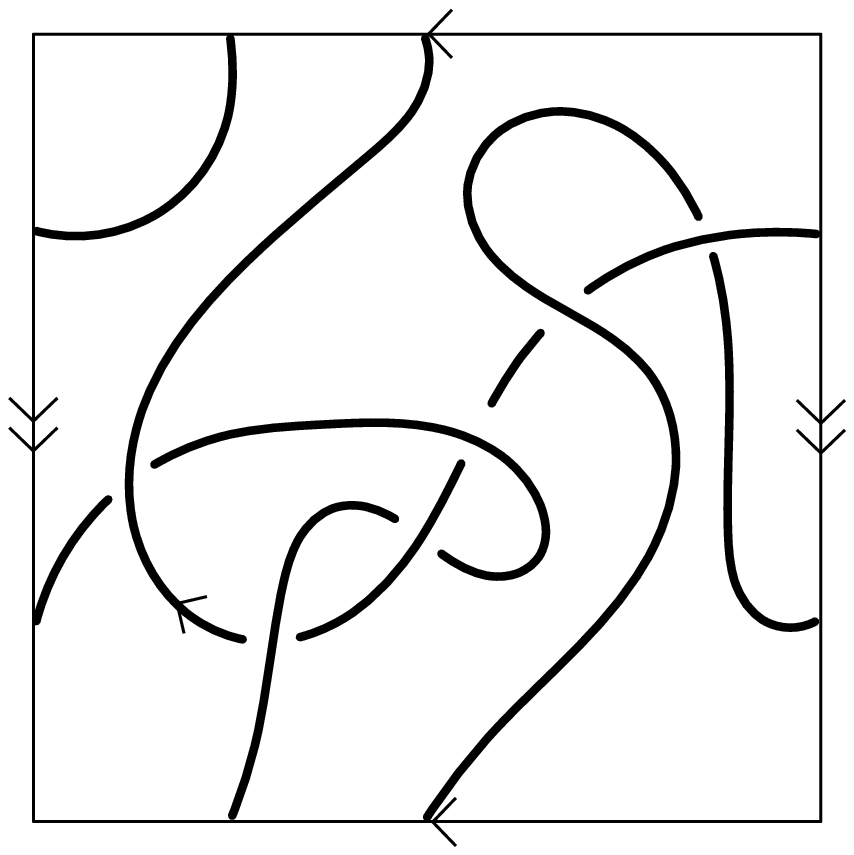
\def\svgwidth{2.9cm}  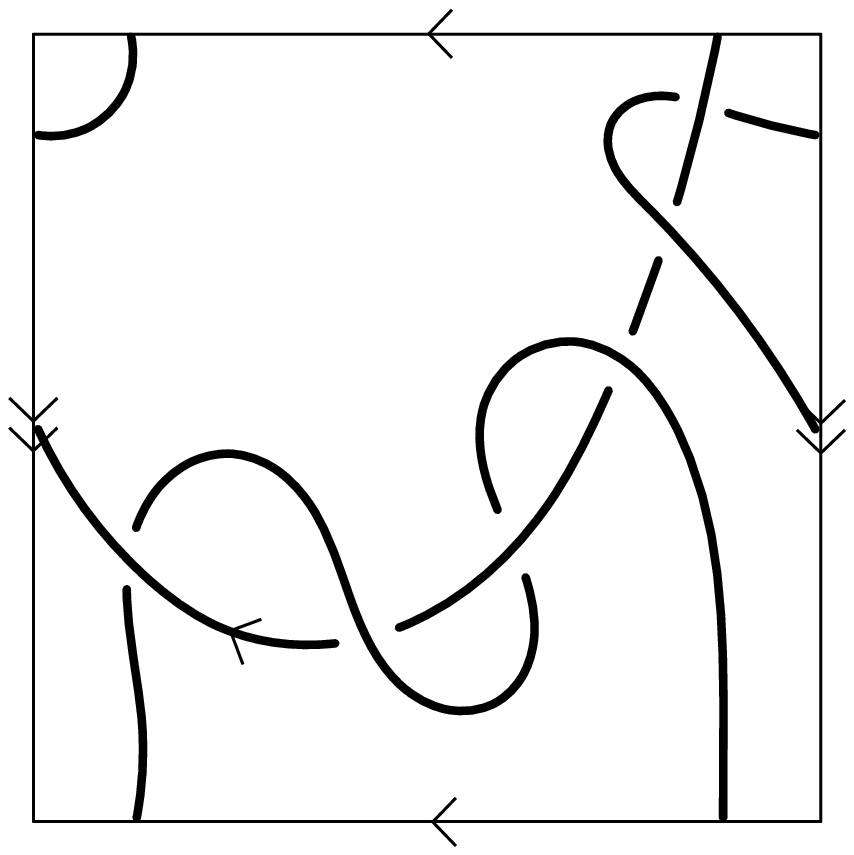
\def\svgwidth{2.9cm}   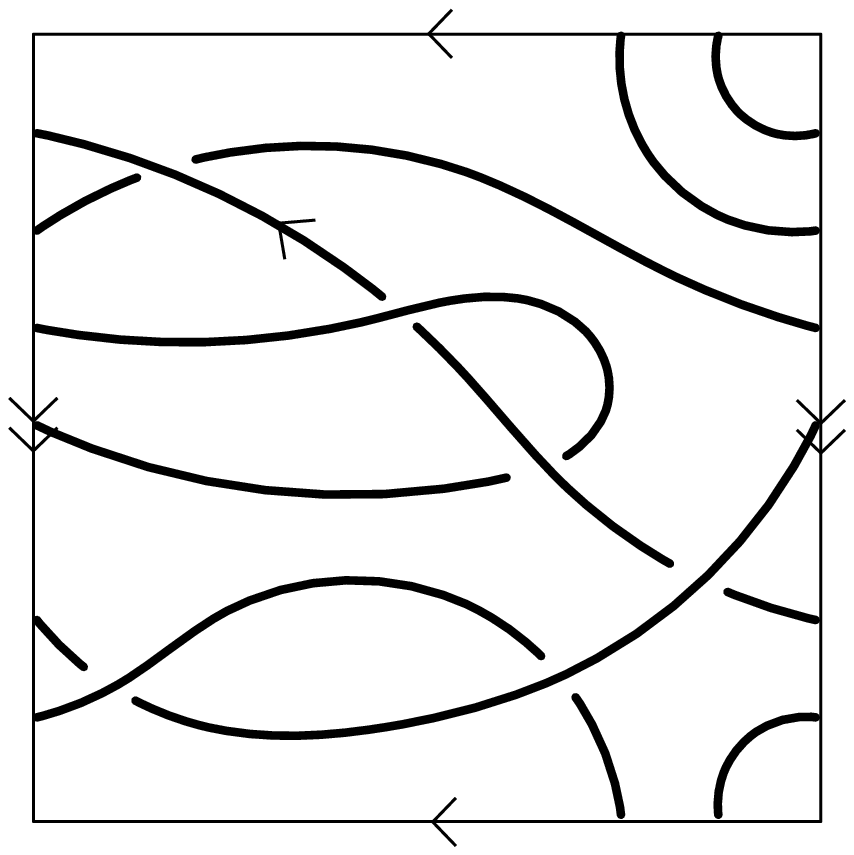
\def\svgwidth{2.9cm}   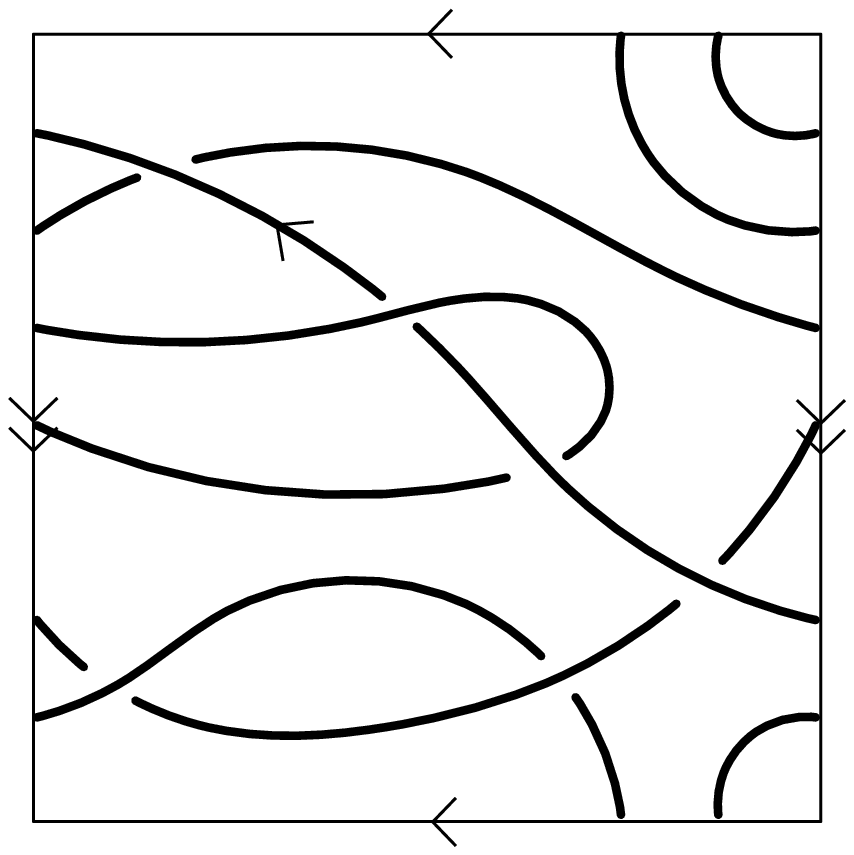

\def\svgwidth{2.9cm} 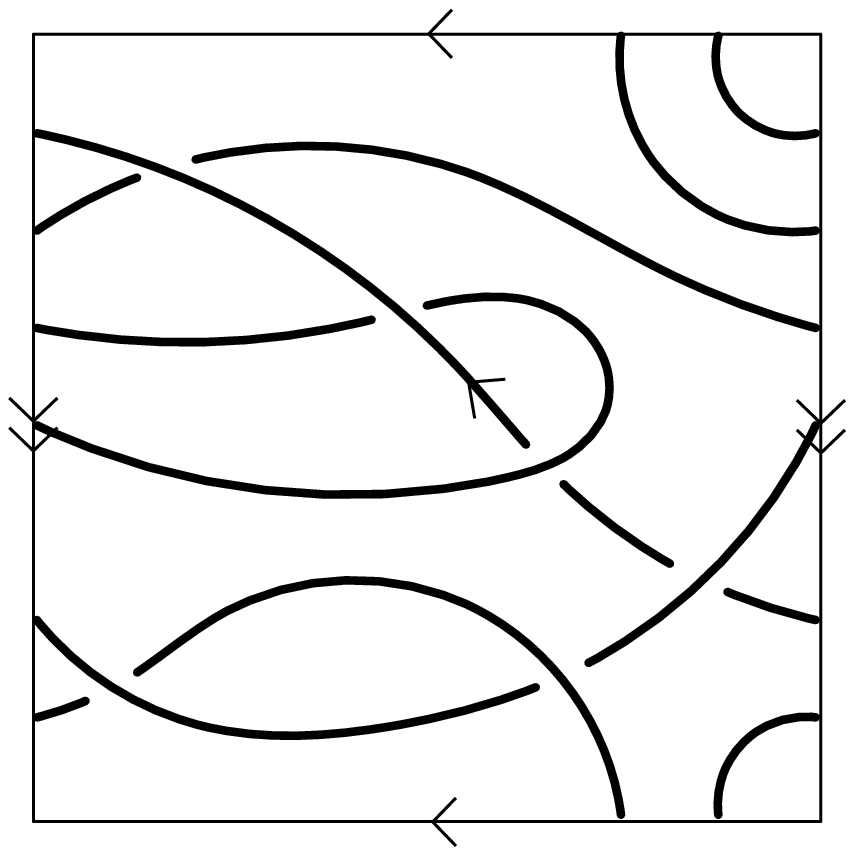
\def\svgwidth{2.9cm}   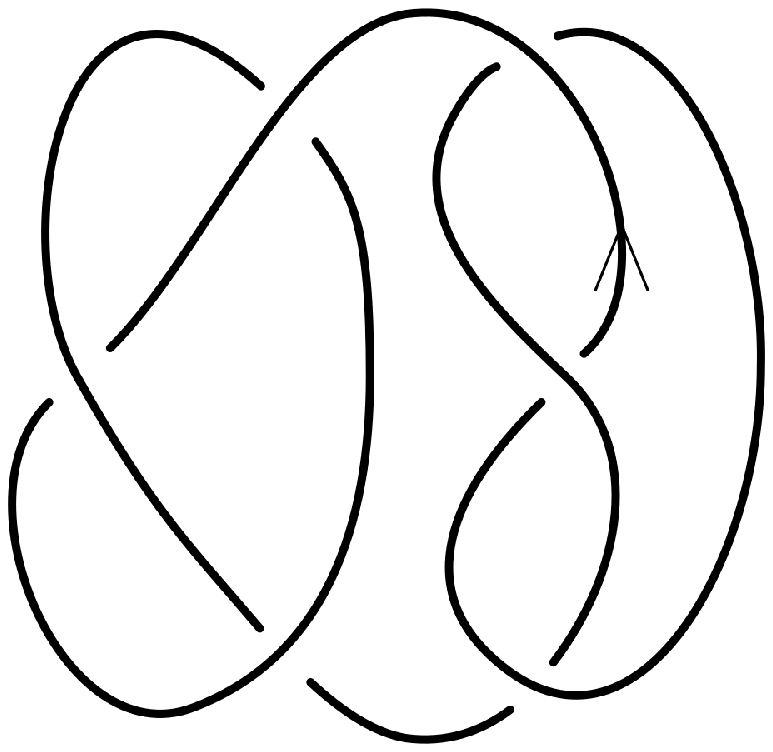
\def\svgwidth{2.9cm}   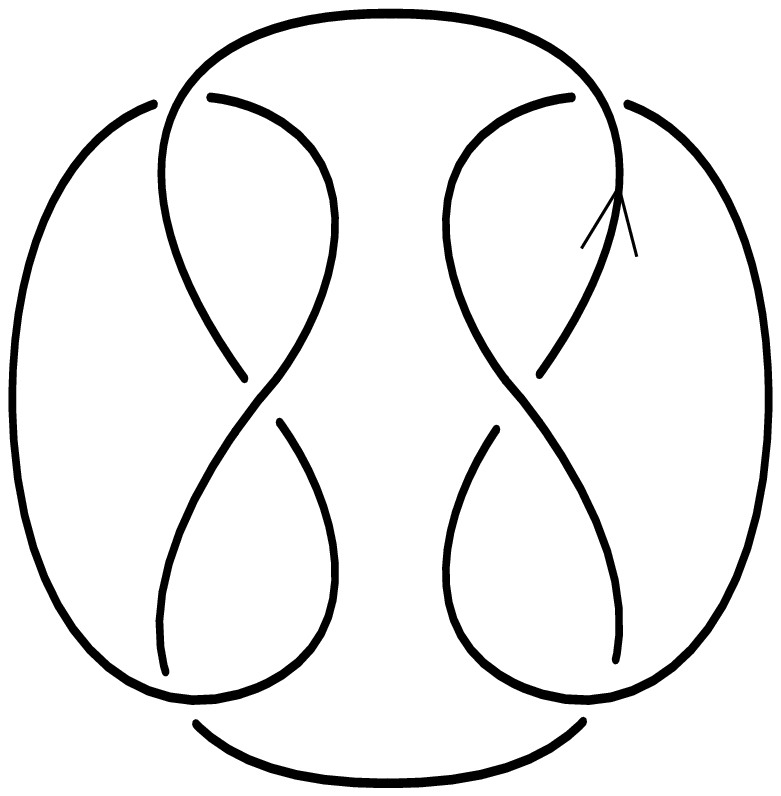
\def\svgwidth{2.9cm}   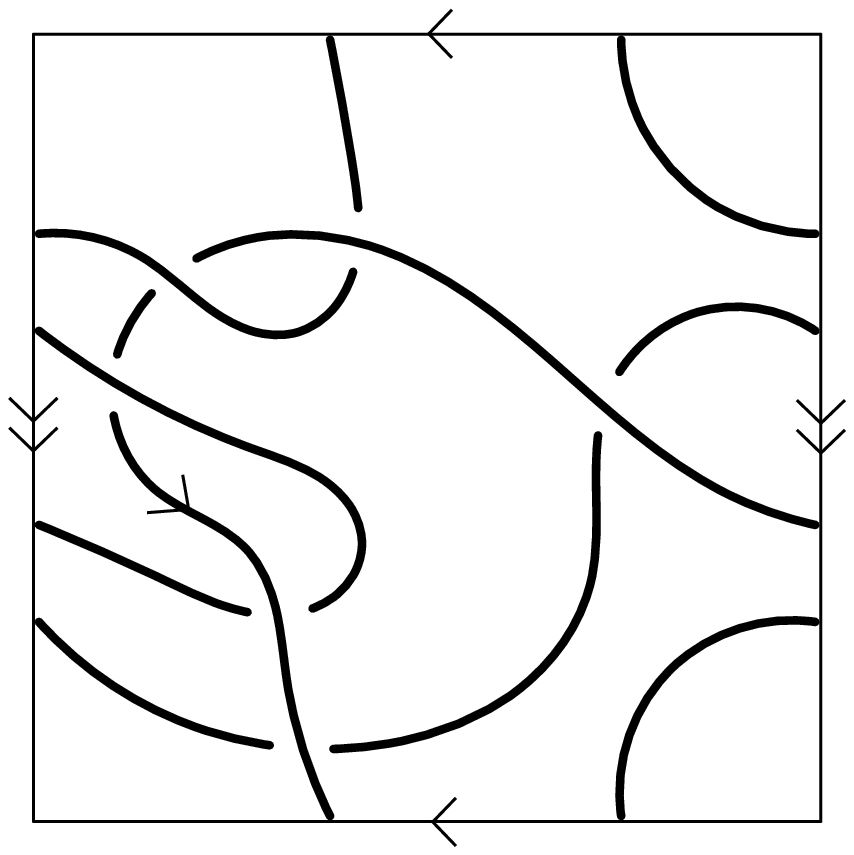
\def\svgwidth{2.9cm}   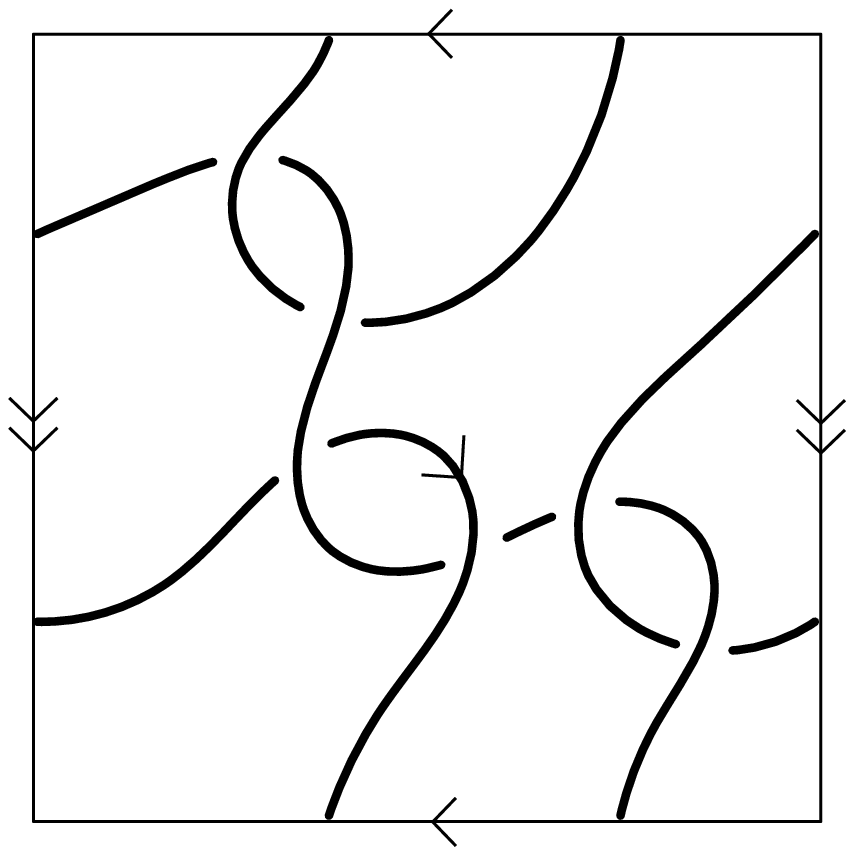

\def\svgwidth{2.9cm}  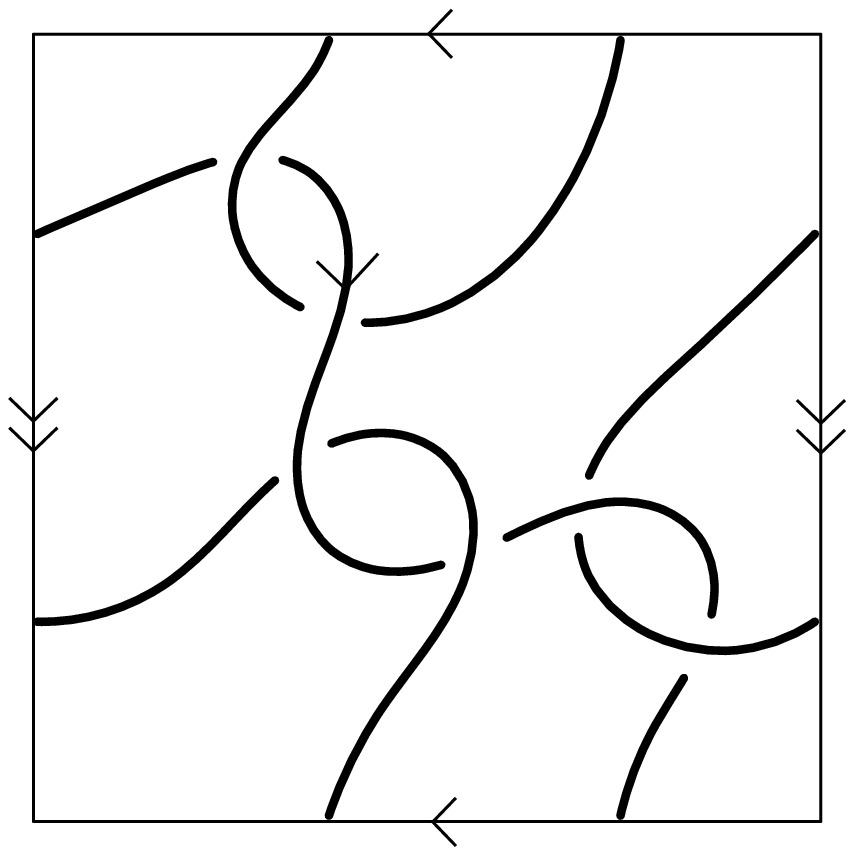
\def\svgwidth{2.9cm}   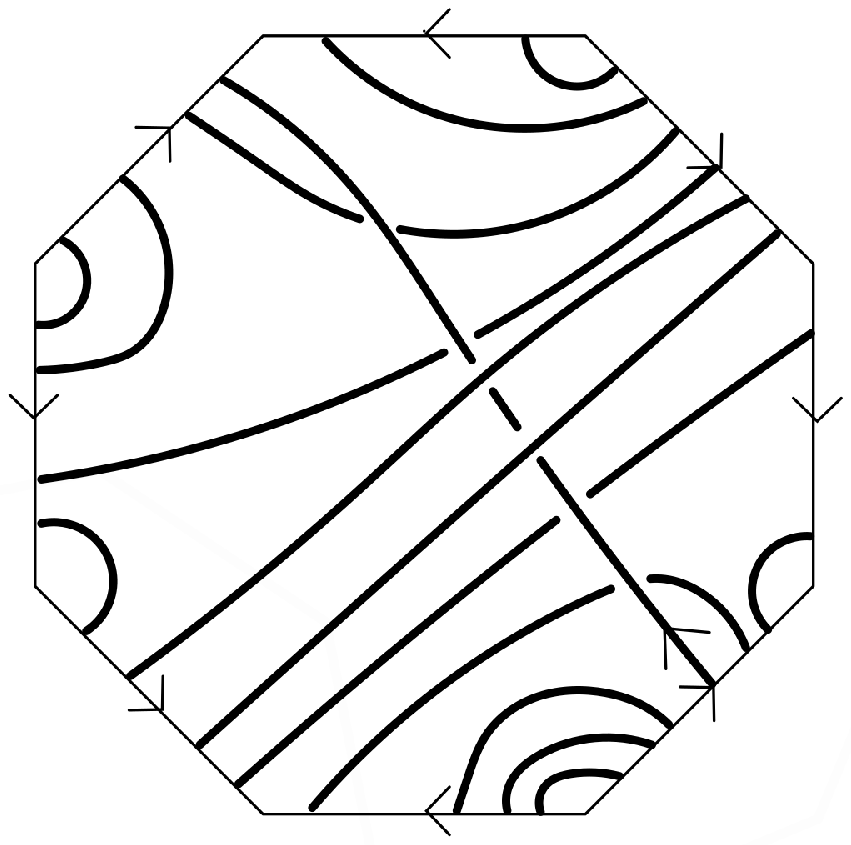
\def\svgwidth{2.9cm}   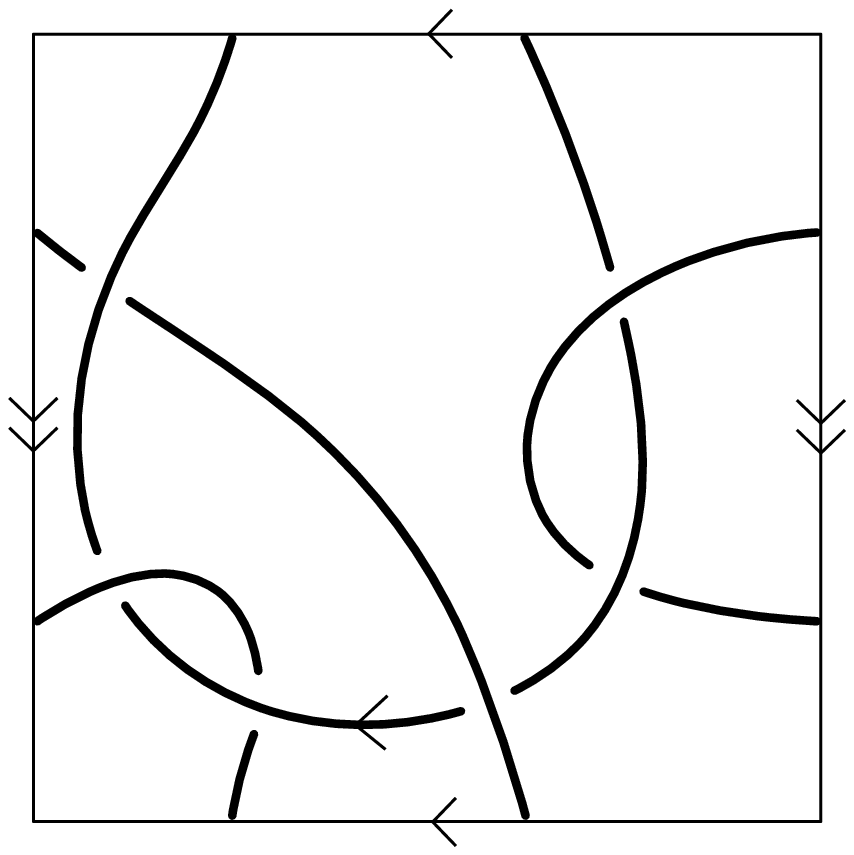
\def\svgwidth{2.9cm}   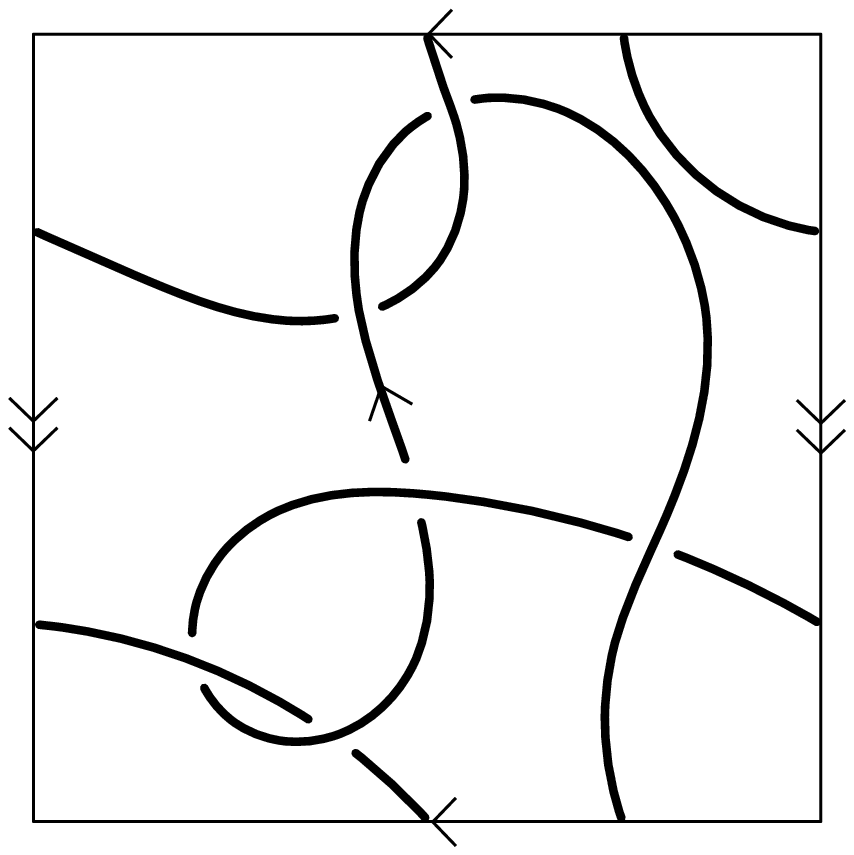
\def\svgwidth{2.9cm}  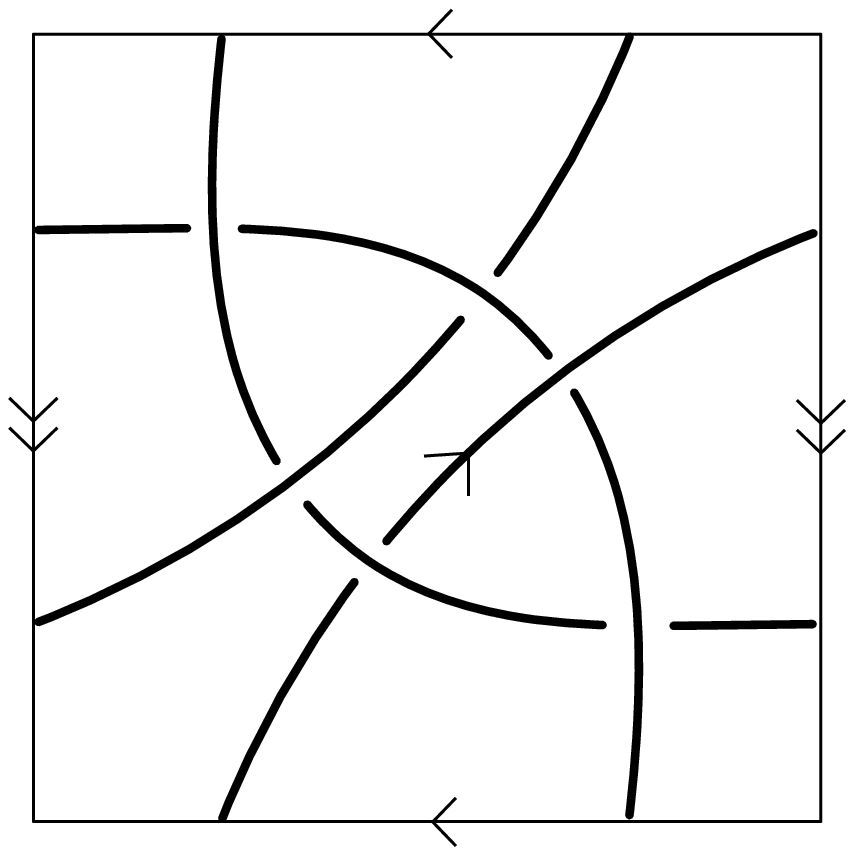

\def\svgwidth{2.9cm}  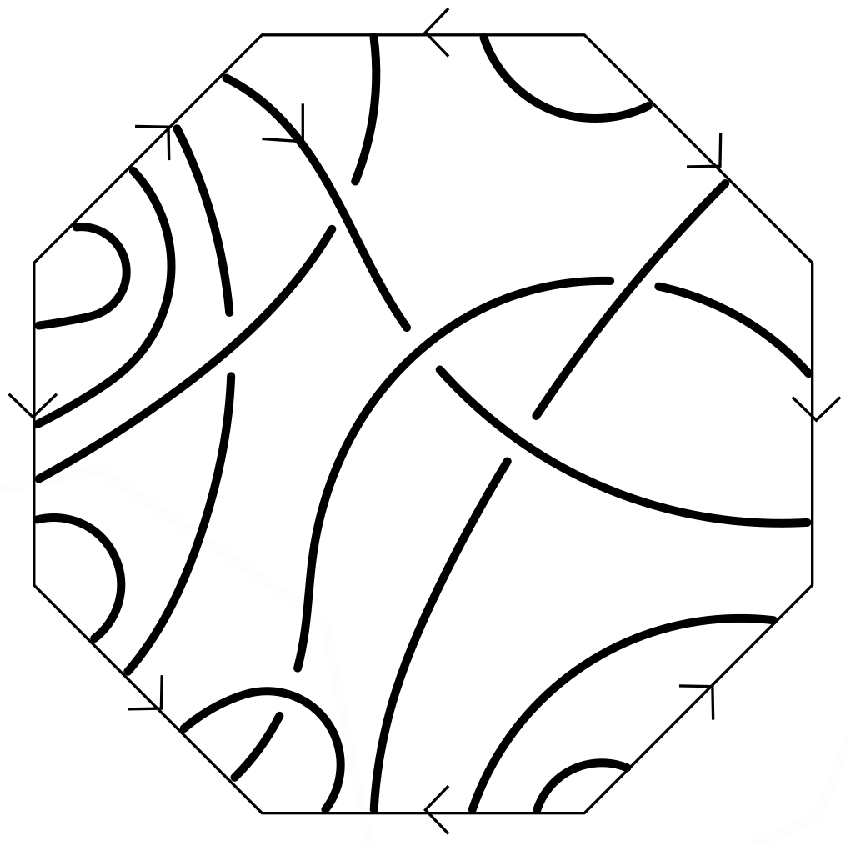
\def\svgwidth{2.9cm}   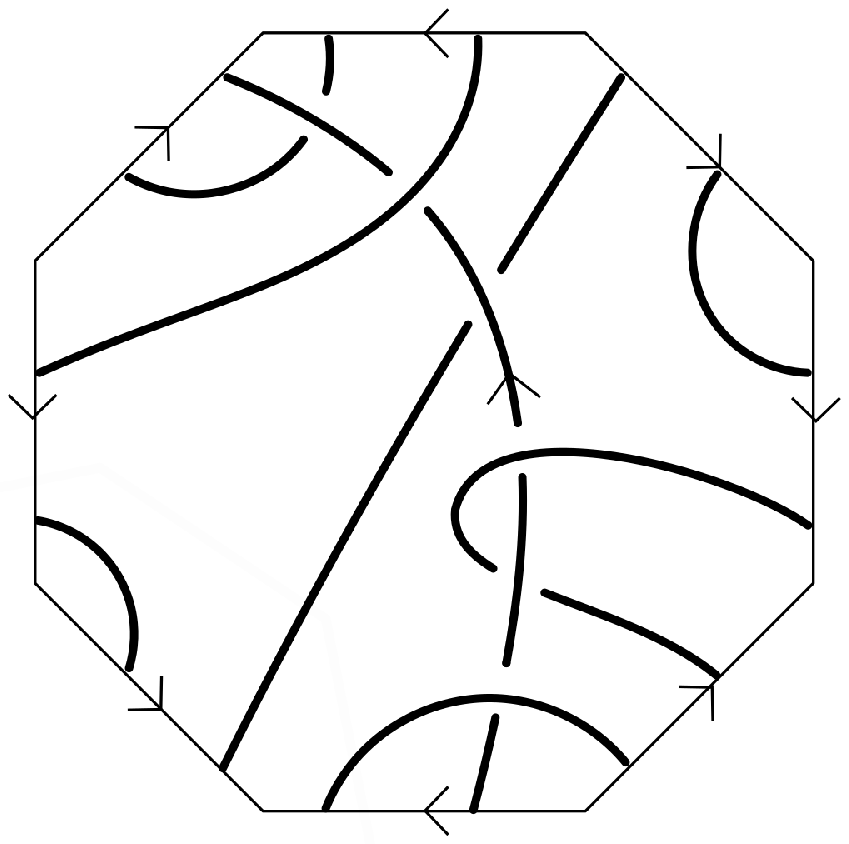
\def\svgwidth{2.9cm}   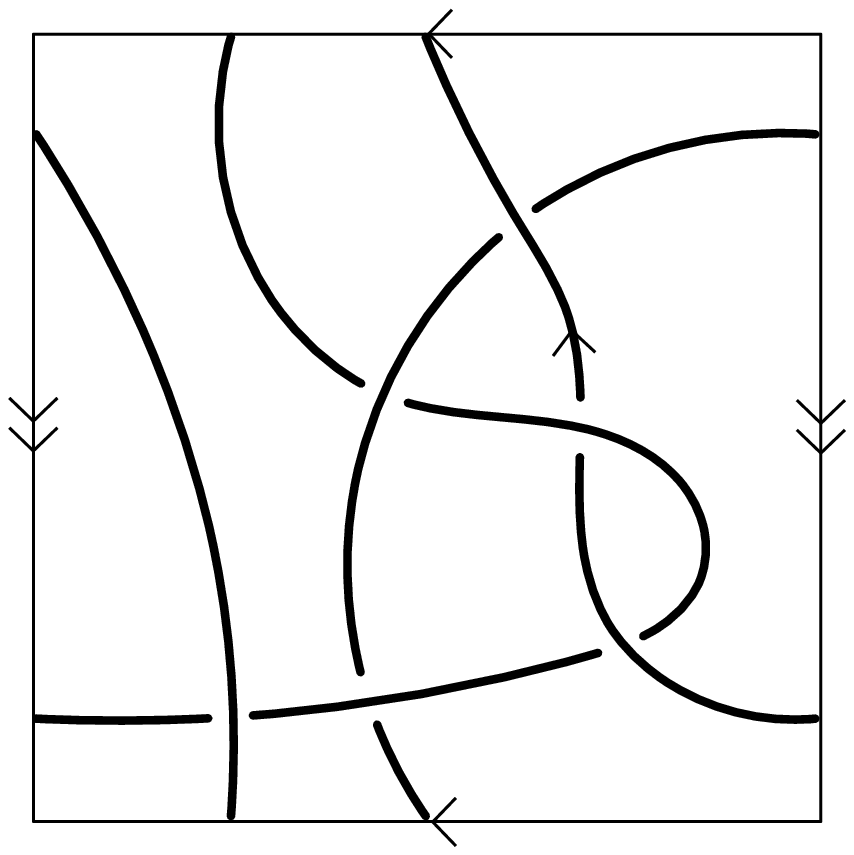
\def\svgwidth{2.9cm}   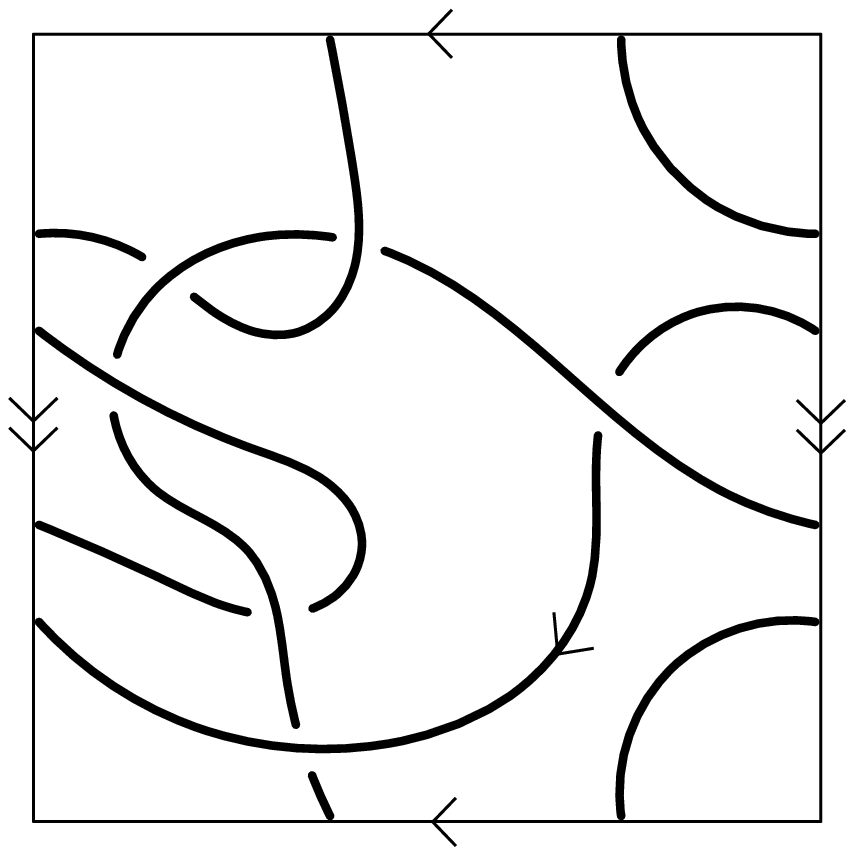
\def\svgwidth{2.9cm}   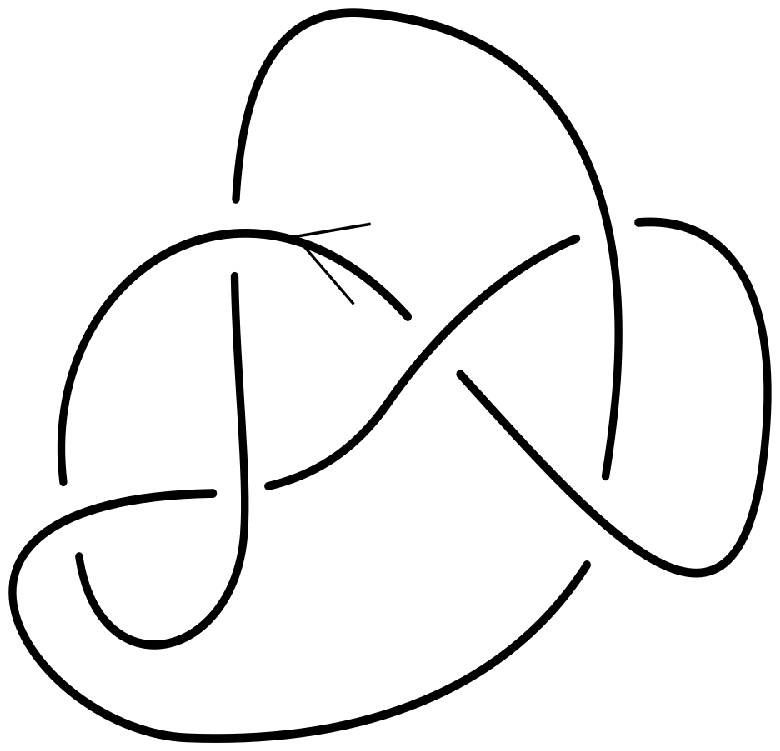

\def\svgwidth{2.9cm}   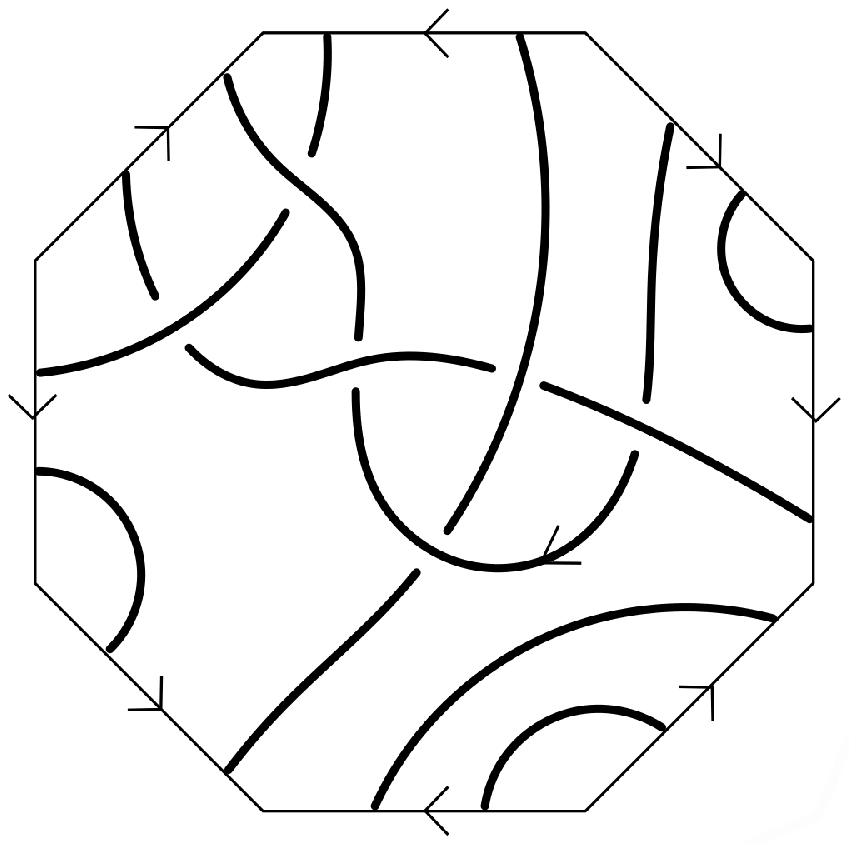
\def\svgwidth{2.9cm}   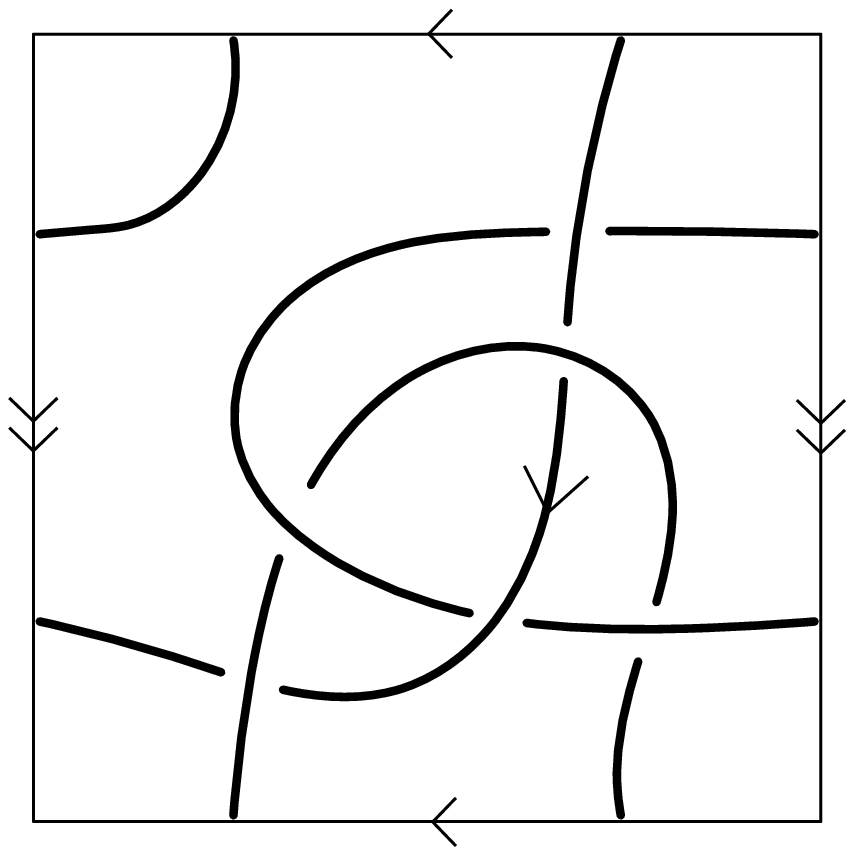
\def\svgwidth{2.9cm}   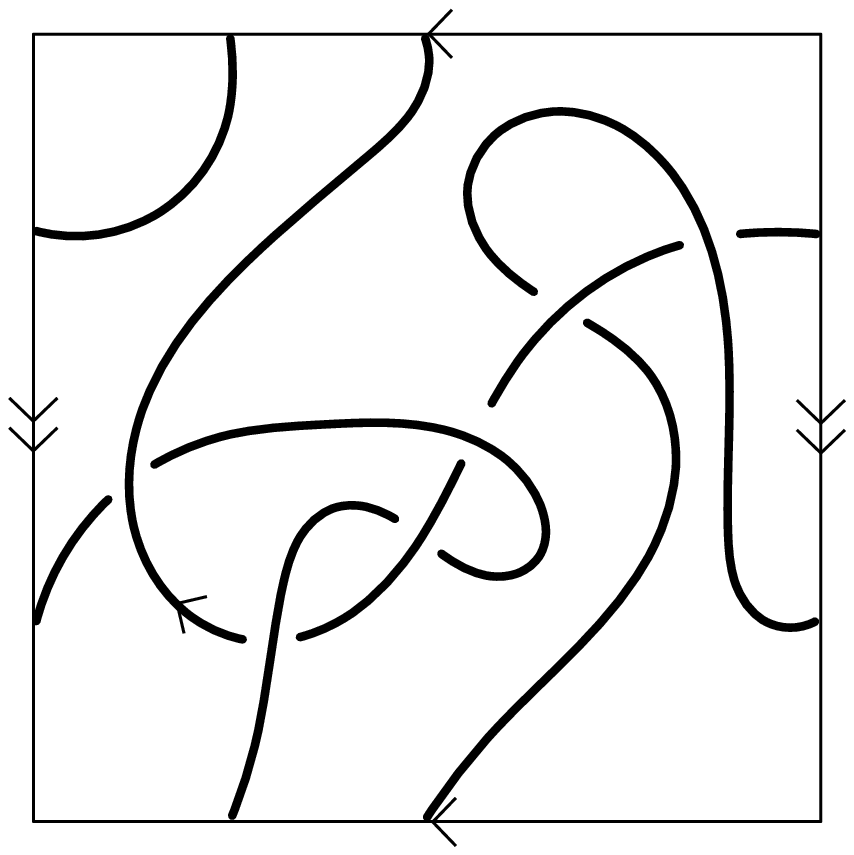
\def\svgwidth{2.9cm} 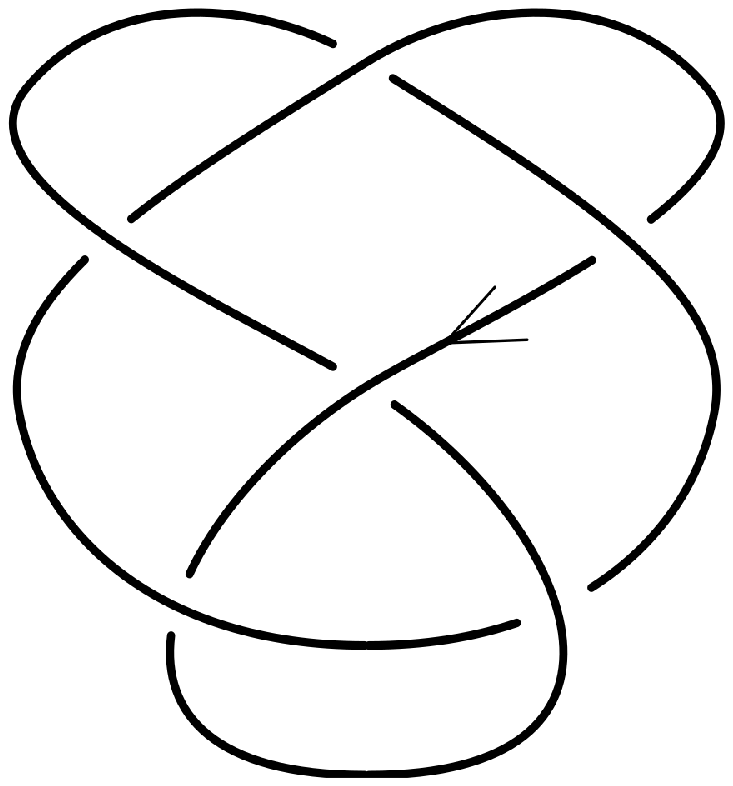
\def\svgwidth{2.9cm}   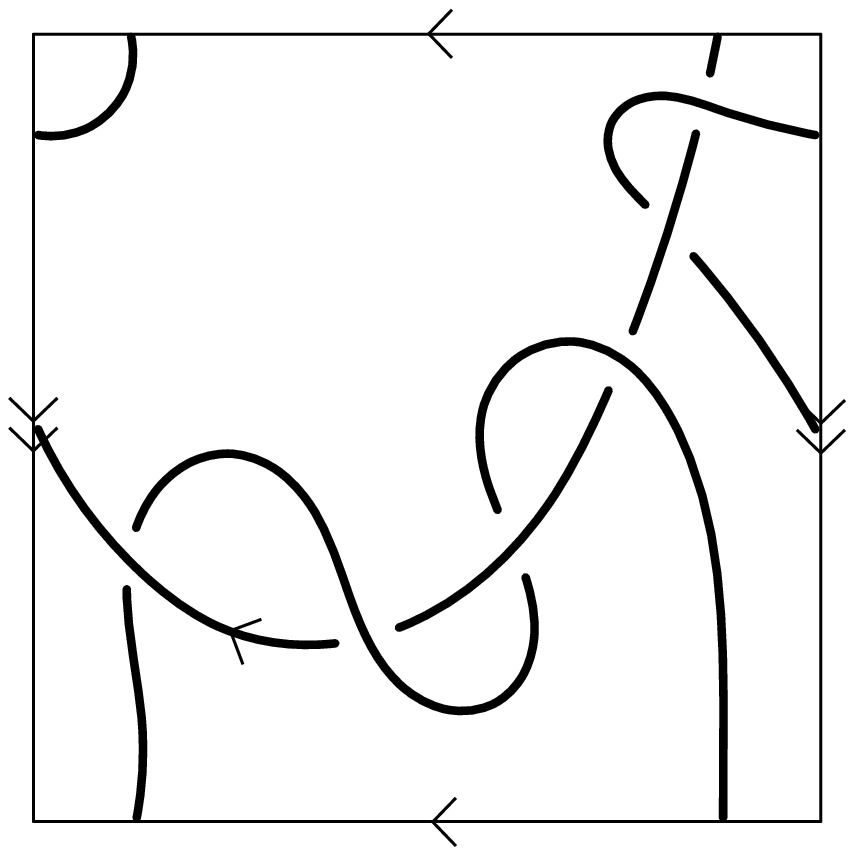
\end{figure}

\begin{figure}[H]
\def\svgwidth{2.9cm}   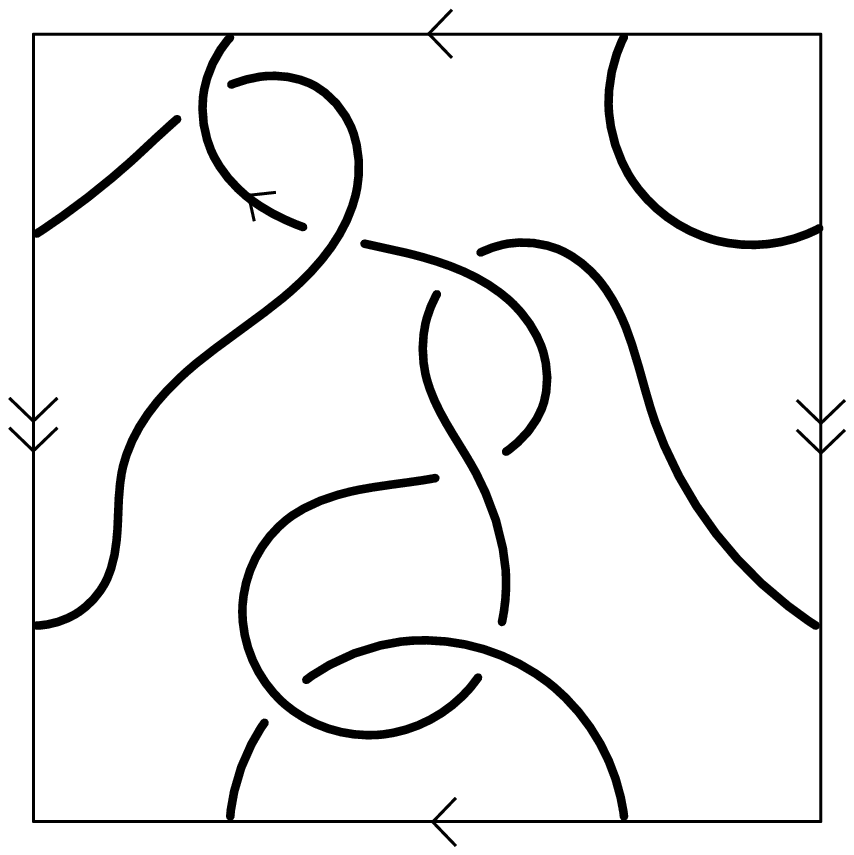
\def\svgwidth{2.9cm}   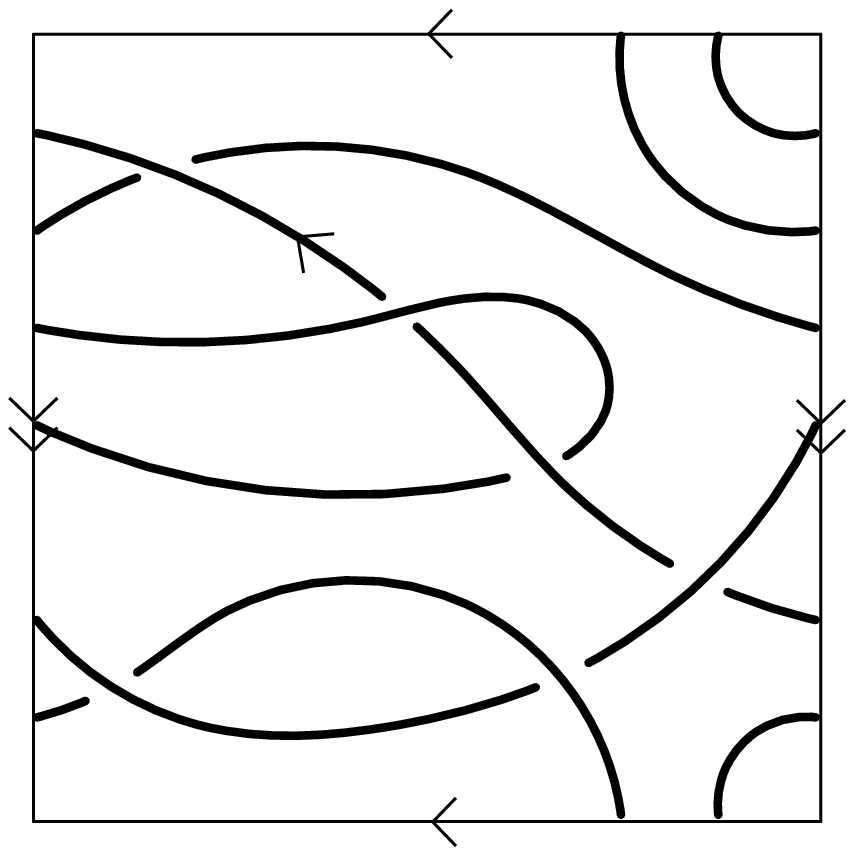
\def\svgwidth{2.9cm}   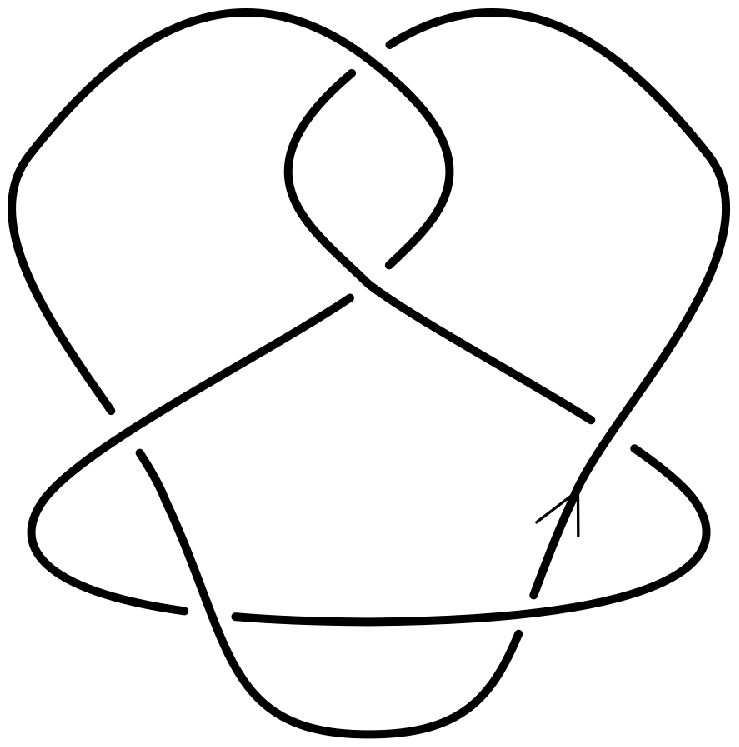
\def\svgwidth{2.9cm}  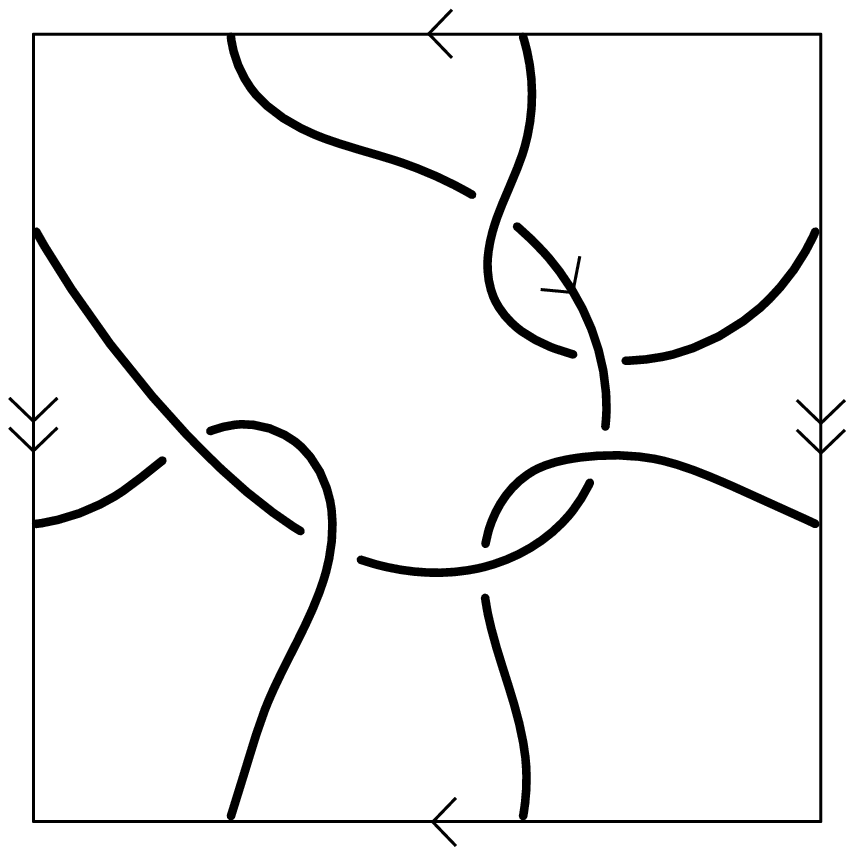 \\

\def\svgwidth{2.9cm}  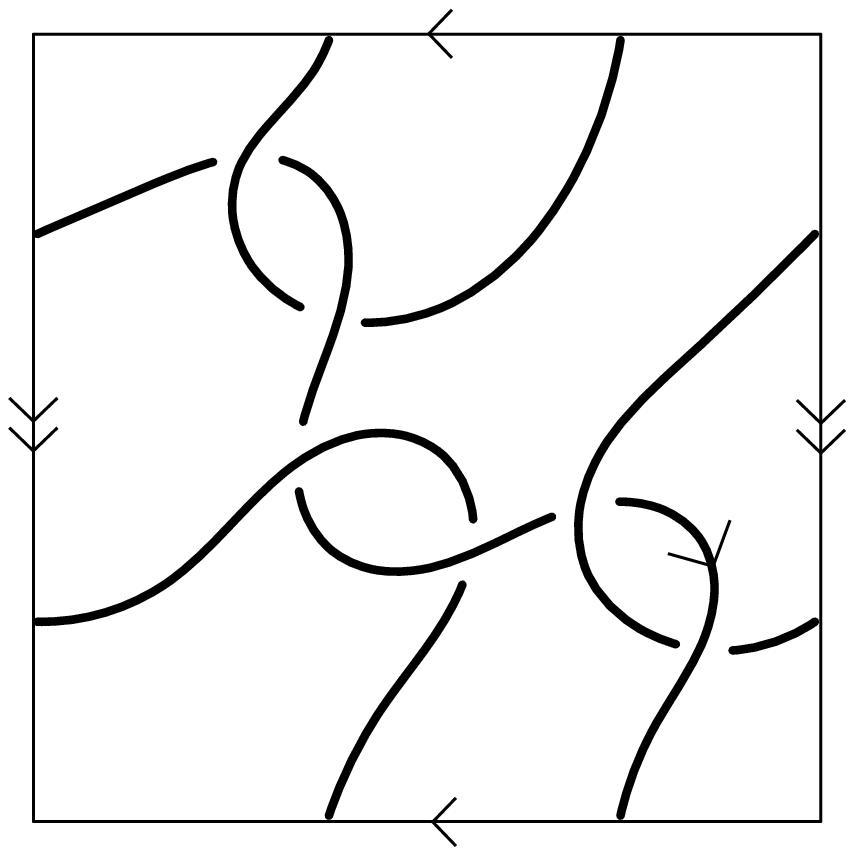
\def\svgwidth{2.9cm} 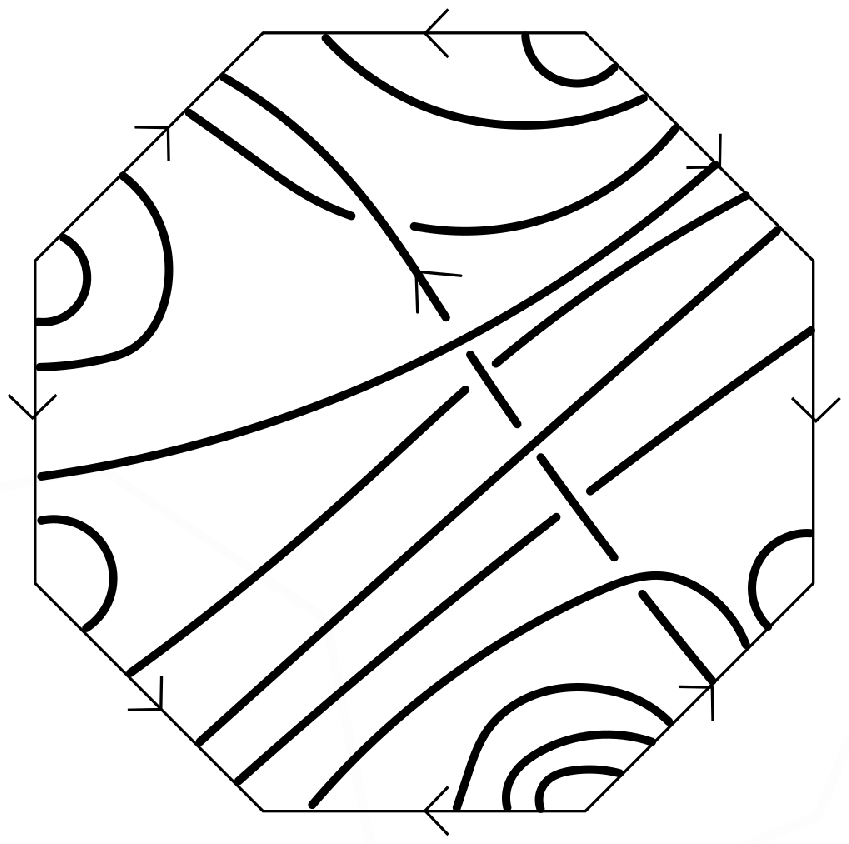
\end{figure}}

\bigskip

\begin{tabular}{|c|l|} \hline
{Knot} & {Seifert matrices} \\ \hline \hline  

{\bf 3.6} & $V^+ = \hbox{\tt [[1,-1],[0,1]]}$ \\ 
& $V^- = \hbox{\tt [[1,0],[-1,1]]}$ \\ \hline 
 
4.99 & $V^+ = \hbox{\tt [[-1,0],[0,1]]}$ \\
& $V^- = \hbox{\tt [[-1,1],[-1,1]]}$ \\ \hline
 
4.105 &  $V^+ = \hbox{\tt [[1,0],[0,1]]}$ \\ 
& $V^- = \hbox{\tt [[1,1],[-1,1]]}$ \\ \hline
 
{\bf 4.108}  & $V^+ = \hbox{\tt [[-1,0],[1,1]]}$ \\
& $V^- = \hbox{\tt [[-1,1],[0,1]]} $ \\ \hline 
 
5.2012 &  $V^+ = \hbox{\tt [[0,0,-1,1],[0,1,0,0],[0,0,0,0],[0,-1,0,1]]}$ \\
& $V^- =\hbox{\tt [[0,-1,-1,1],[1,1,1,0],[0,-1,0,1],[0,-1,-1,1]]}$ \\ \hline  

5.2025
&  $V^+ =\hbox{\tt [[0,0,-1,1],[0,1,0,1],[0,0,0,0],[0,0,0,0]]}$ \\
& $V^- =\hbox{\tt [[0,-1,-1,1],[1,1,1,1],[0,-1,0,1],[0,0,-1,0]]}$    \\ \hline 
 
5.2080
&$V^+ =\hbox{\tt [[1,1,0,0],[0,0,0,0],[0,0,1,1],[-1,0,0,1]]}$  \\
& $V^- =\hbox{\tt  [[1,0,0,-1],[1,0,0,-1],[0,0,1,0],[0,1,1,1]]}$   \\ \hline  

5.2133
&$V^+ =\hbox{\tt [[0,0,-1,1],[0,1,0,0],[0,0,-1,0],[0,-1,0,1]]}$ \\
& $V^- =\hbox{\tt [[0,-1,-1,1],[1,1,1,0],[0,-1,-1,1],[0,-1,-1,1]]}$    \\ \hline 
 
5.2160
&$V^+ =\hbox{\tt [[-1,0,0,1],[-1,-1,0,0],[0,0,1,1],[0,0,0,0]]}$ \\
&$V^- =\hbox{\tt [[-1,-1,0,0],[0,-1,0,-1],[0,0,1,0],[1,1,1,0]]}$    \\ \hline 
 
5.2331
& $V^+ =\hbox{\tt [[1,1,0,0],[-1,0,1,0],[0,0,1,0],[0,0,-1,1]]}$ \\
& $V^- =\hbox{\tt [[1,0,0,0],[0,0,0,0],[0,1,1,-1],[0,0,0,1]]} $   
\\ \hline
 
5.2426
& $V^+ =\hbox{\tt [[1,1,0,0],[0,1,0,0],[0,0,1,1],[-1,0,0,1]]} $ \\
& $V^- =\hbox{\tt [[1,0,0,-1],[1,1,0,-1],[0,0,1,0],[0,1,1,1]]} $  
\\ \hline
 
5.2433
& $V^+ = \hbox{\tt [[1,0,0,1],[0,1,0,0],[1,0,1,0],[0,-1,0,1]]} $ \\
& $V^- = \hbox{\tt [[1,-1,0,1],[1,1,1,0],[1,-1,1,1],[0,-1,-1,1]]} $  
\\ \hline
 
{\bf 5.2437} &$V^+ = \hbox{\tt [[2,-2],[-1,2]]} $ \\
& $V^- = \hbox{\tt [[2,-1],[-2,2]]} $
\\ \hline
 
5.2439
& $V^+ = \hbox{\tt [[-1,1,0,0],[-1,0,1,0],[0,0,1,0],[0,0,-1,1]]} $ \\
& $V^- = \hbox{\tt [[-1,0,0,0],[0,0,0,0],[0,1,1,-1],[0,0,0,1]]} $  
\\ \hline
 
{\bf 5.2445}  
& $V^+ = \hbox{\tt [[1,0,-1,0],[0,1,0,-1],[0,0,1,0],[-1,0,0,1]]} $ \\
& $V^- = \hbox{\tt [[1,0,0,-1],[0,1,0,0],[-1,0,1,0],[0,-1,0,1]]} $  
\\ \hline

\end{tabular}

\newpage
\begin{tabular}{|c|l|} \hline
{Knot} & {Seifert matrices} \\ \hline \hline  
 
6.72557
& $V^+ = \hbox{\tt [[0,1,0,0,0,0],[0,1,0,0,0,0],[1,1,0,-1,1,1],[0,0,0,-1,1,0],}$\\
& \qquad \quad $\hbox{\tt  [0,0,0,0,0,0],[0,0,0,0,0,0]]} $ \\
& $V^- = \hbox{\tt [[0,1,1,0,0,0],[0,1,1,0,0,-1],[0,0,0,0,0,0],[0,0,-1,-1,0,0],}$\\
&  \qquad \quad $\hbox{\tt [0,0,1,1,0,-1],[0,1,1,0,1,0]]}$ 
\\ \hline

6.72692
& $V^+ = \hbox{\tt [[1,1,1,1],[-1,0,-1,0],[-1,0,0,0],[-1,1,0,1]]}$ \\
& $V^- = \hbox{\tt [[1,0,0,0],[0,0,-1,1],[0,0,0,0],[0,0,0,1]]}$  
\\ \hline

6.72695  
& $V^+ = \hbox{\tt [[0,1,0,0],[0,-1,0,0],[1,1,1,1],[-1,1,-1,0]]}$ \\
& $V^- = \hbox{\tt [[0,0,0,0],[1,-1,1,1],[1,0,1,0],[-1,0,0,0]]}$  
\\  \hline

6.72938
& $V^+ = \hbox{\tt [[1,1,0,0,0,0],[0,1,0,0,0,0],[1,1,1,0,1,1],[0,0,1,0,1,0],}$\\
&\qquad \quad $\hbox{\tt [0,0,0,0,0,0],[0,0,0,0,0,0]]}$ \\ 
& $V^- = \hbox{\tt [[1,1,1,0,0,0],[0,1,1,0,0,-1],[0,0,1,1,0,0],[0,0,0,0,0,0],}$ \\
& \qquad \quad $\hbox{\tt [0,0,1,1,0,-1],[0,1,1,0,1,0]] } $ 
\\ \hline

6.72944  
& $V^+ = \hbox{\tt [[-1,-1,0,0],[0,0,0,0],[0,0,1,0],[0,0,1,1]]}$ \\
& $V^- = \hbox{\tt [[-1,-1,0,-1],[0,0,1,0],[0,-1,1,0],[1,0,1,1]]}$ 
\\ \hline

6.72975
& $V^+ = \hbox{\tt [[0,1,0,0,0,0],[0,1,0,0,0,0],[1,1,1,0,1,1],[0,0,1,0,1,0],}$\\
& \qquad \quad $\hbox{\tt [0,0,0,0,0,0],[0,0,0,0,0,0]]}$  \\
& $V^- = \hbox{\tt [[0,1,1,0,0,0],[0,1,1,0,0,-1],[0,0,1,1,0,0],[0,0,0,0,0,0],}$ \\
& \qquad \quad $\hbox{\tt [0,0,1,1,0,-1],[0,1,1,0,1,0]]}$  
\\ \hline

6.73007
& $V^+ = \hbox{\tt [[-1,1,1,1],[-1,0,-1,0],[-1,0,0,0],[-1,1,0,1]]}$  \\
& $V^- = \hbox{\tt [[-1,0,0,0],[0,0,-1,1],[0,0,0,0],[0,0,0,1]]}$  
\\ \hline

6.73053
& $V^+ = \hbox{\tt [[1,0,1,0],[-1,1,0,0],[-1,1,0,1],[0,0,0,1]]}$  \\
& $V^- = \hbox{\tt [[1,-1,0,0],[0,1,1,0],[0,0,0,0],[0,0,1,1]]}$  
\\ \hline

6.73583
& $V^+ = \hbox{\tt [[1,0,0,0],[0,0,0,0],[0,0,0,-1],[0,1,0,0]]}$  \\
& $V^- = \hbox{\tt [[1,-1,1,0],[1,0,0,1],[-1,0,0,-1],[0,0,0,0]]}$   
\\ \hline

6.75341  
& $V^+ = \hbox{\tt [[0,1,1,1],[-1,0,1,0],[0,0,1,0],[0,-1,0,-1]]}$  \\
& $V^+ = \hbox{\tt [[0,0,0,0],[0,0,0,0],[1,1,1,0],[1,-1,0,-1]]}$  
\\ \hline

6.75348
& $V^+ = \hbox{\tt [[1,0,0,0],[0,0,0,-1],[0,0,0,-1],[0,0,0,-1]]}$  \\
& $V^- = \hbox{\tt [[1,-1,1,0],[1,0,0,0],[-1,0,0,-1],[0,-1,0,-1]]}$  
\\ \hline

$6.76479^r$
& $V^+ = \hbox{\tt [[1,0,0,0],[0,0,0,-1],[0,0,1,0],[0,0,1,0]]}$  \\
& $V^- = \hbox{\tt [[1,-1,1,0],[1,0,0,0],[-1,0,1,0],[0,-1,1,0]]}$  
\\ \hline

6.77833
& $V^+ = \hbox{\tt [[1,1,0,0,0,0],[0,0,0,0,0,0],[1,1,0,-1,1,1],[0,0,0,-1,1,0],} $ \\
& \qquad \quad $\hbox{\tt  [0,0,0,0,0,0],[0,0,0,0,0,0]]}$  \\
& $V^- = \hbox{\tt [[1,1,1,0,0,0],[0,0,1,0,0,-1],[0,0,0,0,0,0],[0,0,-1,-1,0,0],} $ \\
& \qquad \quad $\hbox{\tt [0,0,1,1,0,-1],[0,1,1,0,1,0]]}$  
\\ \hline

6.77844
& $V^+ = \hbox{\tt [[0,1,0,0,0,0],[0,0,0,0,0,0],[1,1,0,-1,1,1],[0,0,0,-1,1,0],} $ \\
& \qquad \quad $\hbox{\tt [0,0,0,0,0,0],[0,0,0,0,0,0]]}$  \\
& $V^- = \hbox{\tt [[0,1,1,0,0,0],[0,0,1,0,0,-1],[0,0,0,0,0,0],[0,0,-1,-1,0,0],} $ \\
& \qquad \quad $\hbox{\tt  [0,0,1,1,0,-1],[0,1,1,0,1,0]]}$  
\\ \hline

6.77905
& $V^+ = \hbox{\tt [[1,1,1,1],[-1,0,-1,0],[-1,0,-1,0],[-1,1,0,1]]}$  \\
& $V^- = \hbox{\tt [[1,0,0,0],[0,0,-1,1],[0,0,-1,0],[0,0,0,1]]}$  
\\ \hline

6.77908 
& $V^+ = \hbox{\tt [[-1,1,-1,0],[0,-1,0,0],[0,1,0,1],[-1,1,-1,0]]}$  \\
& $V^- = \hbox{\tt [[-1,0,-1,0],[1,-1,1,1],[0,0,0,0],[-1,0,0,0]]}$  
\\ \hline

6.77985
& $V^+ = \hbox{\tt  [[-1,1,1,1],[-1,0,-1,0],[-1,0,-1,0],[-1,1,0,1]]}$  \\
& $V^- = \hbox{\tt [[-1,0,0,0],[0,0,-1,1],[0,0,-1,0],[0,0,0,1]]}$  
\\ \hline

6.78358
& $V^+ = \hbox{\tt [[0,0,1,0],[-1,1,0,0],[-1,1,0,1],[0,0,0,1]]}$  \\
& $V^- = \hbox{\tt [[0,-1,0,0],[0,1,1,0],[0,0,0,0],[0,0,1,1]]}$  
\\  \hline

\end{tabular}

\begin{tabular}{|c|l|} \hline
{Knot} & {Seifert matrices} \\ \hline \hline 

6.79342  
& $V^+ =  \hbox{\tt [[-1,1,1,1],[-1,0,0,-1],[-1,1,1,0],[-1,0,0,0]]} $ \\ 
& $ V^- = \hbox{\tt [[-1,0,0,0],[0,0,0,0],[0,1,1,0],[0,-1,0,0]]} $  \\ \hline

6.85091
& $ V^+ = \hbox{\tt [[-1,0,0,0],[0,1,-1,0],[0,1,0,1],[0,0,0,1]]} $ \\ 
& $ V^- = \hbox{\tt [[-1,1,0,0],[-1,1,0,0],[0,0,0,0],[0,0,1,1]]} $  \\ \hline

6.85103
& $ V^+ = \hbox{\tt [[-1,0,0,0],[0,0,-1,0],[0,1,0,1],[0,0,0,1]]} $ \\ 
& $V^- = \hbox{\tt [[-1,1,0,0],[-1,0,0,0],[0,0,0,0],[0,0,1,1]]} $  \\ \hline

6.85613
& $ V^+ = \hbox{\tt [[0,0,0,1],[-1,0,-1,0],[0,1,1,0],[-1,0,-1,1]]} $ \\ 
& $ V^- = \hbox{\tt [[0,-1,0,0],[0,0,0,0],[0,0,1,0],[0,0,-1,1]]} $  \\ \hline

6.85774
& $ V^+ = \hbox{\tt [[0,0,0,1],[-1,1,-1,0],[0,1,1,0],[-1,0,-1,1]]} $ \\ 
& $ V^- = \hbox{\tt [[0,-1,0,0],[0,1,0,0],[0,0,1,0],[0,0,-1,1]]} $  \\ \hline

6.87188 
& $ V^+ = \hbox{\tt [[-1,0,0,0],[0,1,0,0],[0,0,1,-1],[0,1,0,1]]} $ \\ 
& $ V^- = \hbox{\tt [[-1,1,0,0],[-1,1,-1,1],[0,1,1,0],[0,0,-1,1]]} $  \\ \hline

6.87262
& $ V^+ = \hbox{\tt [[1,1,0,0,0,0],[0,1,0,0,0,0],[1,1,1,0,1,1],[0,0,1,1,1,0],}$\\
& \qquad \quad $\hbox{\tt [0,0,0,0,0,0],[0,0,0,0,0,0]]} $ \\ 
& $ V^- = \hbox{\tt [[1,1,1,0,0,0],[0,1,1,0,0,-1],[0,0,1,1,0,0],[0,0,0,1,0,0],}$ \\
& \qquad \quad $\hbox{\tt [0,0,1,1,0,-1],[0,1,1,0,1,0]]} $  \\ \hline

6.87269
& $ V^+ = \hbox{\tt [[-1,-1,0,0],[0,-1,0,0],[0,0,1,0],[0,0,1,1]]} $ \\ 
& $ V^- = \hbox{\tt [[-1,-1,0,-1],[0,-1,1,0],[0,-1,1,0],[1,0,1,1]]} $  \\ \hline

6.87310
& $ V^+ = \hbox{\tt [[1,1,-1,-1],[0,1,0,-1],[0,0,1,1],[0,0,0,1]]} $ \\ 
& $ V^- = \hbox{\tt [[1,1,0,0],[0,1,0,0],[-1,0,1,0],[-1,-1,1,1]]} $  \\ \hline

6.87319
& $ V^+ = \hbox{\tt [[1,0,0,1],[-1,1,-1,0],[0,1,-1,1],[-1,0,0,-1]]} $ \\ 
& $V^- =  \hbox{\tt [[1,-1,0,0],[0,1,0,0],[0,0,-1,1],[0,0,0,-1]]} $  \\ \hline

$6.87369^r$  
& $V^+ =  \hbox{\tt [[1,1,1,1],[-1,-1,-1,-1],[-1,0,-1,0],[-1,0,0,-1]]} $ \\ 
& $V^- =  \hbox{\tt [[1,0,0,0],[0,-1,-1,0],[0,0,-1,0],[0,-1,0,-1]]} $  \\ \hline

6.87548
& $V^+ =  \hbox{\tt [[0,0,0,0],[1,1,0,0],[0,0,1,0],[0,0,0,-1]]} $ \\ 
& $ V^- = \hbox{\tt [[0,1,1,0],[0,1,0,0],[-1,0,1,-1],[0,0,1,-1]]} $ \\ \hline

6.87846
& $V^+ =  \hbox{\tt [[1,1,0,0],[0,1,0,0],[-1,0,1,0],[0,0,0,-1]]} $ \\ 
& $ V^- = \hbox{\tt [[1,0,-1,0],[1,1,0,-1],[0,0,1,0],[0,1,0,-1]]} $  \\ \hline

6.87857
& $V^+ =  \hbox{\tt [[2,0],[0,-1]]} $  \\  
& $ V^- = \hbox{\tt [[2,1],[-1,-1]]} $  \\ \hline

6.87859
& $V^+ =  \hbox{\tt [[-1,0,0,0],[0,1,-1,0],[0,1,0,1],[0,0,0,-1]]} $ \\ 
& $ V^- = \hbox{\tt [[-1,1,0,0],[-1,1,0,0],[0,0,0,0],[0,0,1,-1]]} $  \\ \hline

6.87875
& $ V^+ = \hbox{\tt [[-1,0,0,0],[0,0,-1,0],[0,1,0,1],[0,0,0,-1]]} $ \\ 
& $ V^- = \hbox{\tt [[-1,1,0,0],[-1,0,0,0],[0,0,0,0],[0,0,1,-1]]} $  \\ \hline

6.89156
& $ V^+ = \hbox{\tt [[1,0,0,0],[0,1,-1,0],[0,1,0,1],[0,0,0,1]]} $ \\ 
& $V^- =  \hbox{\tt [[1,1,0,0],[-1,1,0,0],[0,0,0,0],[0,0,1,1]]} $  \\ \hline

{\bf 6.89187}
& $ V^+ = \hbox{\tt [[1,-1,0,0],[0,1,0,0],[0,0,-1,1],[0,0,0,-1]]} $ \\ 
& $V^- =  \hbox{\tt [[1,0,0,0],[-1,1,0,0],[0,0,-1,0],[0,0,1,-1]]} $  \\ \hline

{\bf 6.89198}
& $ V^+ = \hbox{\tt [[1,-1,0,0],[0,1,0,0],[0,0,1,1],[0,0,0,1]]} $ \\ 
& $V^- =  \hbox{\tt [[1,0,0,0],[-1,1,0,0],[0,0,1,0],[0,0,1,1]]} $  \\ \hline

6.89623
& $V^+ =  \hbox{\tt [[-1,0,0,-1],[0,1,0,0],[0,-1,1,0],[0,0,0,-1]]} $ \\ 
& $V^- =  \hbox{\tt [[-1,-1,0,0],[1,1,-1,0],[0,0,1,0],[-1,0,0,-1]]} $  \\ \hline

6.89812
& $ V^+ = \hbox{\tt [[0,0,0,0],[1,1,0,0],[0,0,1,0],[1,0,1,1]]} $ \\ 
& $ V^- = \hbox{\tt [[0,1,1,1],[0,1,0,0],[-1,0,1,1],[0,0,0,1]]} $  \\ \hline

$6.89815^r$
& $ V^+ = \hbox{\tt [[0,0,0,0],[1,-1,0,0],[0,0,1,0],[1,0,1,1]]} $ \\ 
& $V^- =  \hbox{\tt [[0,1,1,1],[0,-1,0,0],[-1,0,1,1],[0,0,0,1]]} $  \\ \hline

6.90099
& $V^+ =  \hbox{\tt [[1,-1,0,0,1,0],[1,0,0,0,0,0],[0,1,1,0,0,0],[0,0,1,1,0,1],}$\\
&\qquad \quad $\hbox{\tt [0,0,0,0,1,-1],[0,0,0,0,1,0]]} $ \\ 
& $ V^- = \hbox{\tt [[1,0,0,0,1,0],[0,0,1,0,0,0],[0,0,1,1,0,0],[0,0,0,1,0,0],}$ \\
&\qquad \quad $\hbox{\tt [0,0,0,0,1,0],[0,0,0,1,0,0]]} $  \\ \hline

\end{tabular}

\begin{table}
\begin{tabular}{|c|l|} \hline
{Knot} & {Seifert matrices} \\ \hline \hline 

6.90109 
& $V^+ =  \hbox{\tt [[1,-1,0,0],[0,1,1,0],[0,0,1,0],[0,0,0,1]]} $ \\ 
& $ V^- = \hbox{\tt [[1,0,1,1],[-1,1,0,0],[-1,1,1,0],[-1,0,0,1]]} $  \\ \hline

6.90115
& $ V^+ = \hbox{\tt [[1,0,0,0],[1,-1,0,0],[0,0,1,-1],[1,0,0,1]]} $ \\ 
& $ V^- = \hbox{\tt [[1,1,0,0],[0,-1,0,0],[0,0,1,0],[1,0,-1,1]]} $  \\ \hline

6.90139
& $V^+ =  \hbox{\tt [[1,0,0,0],[1,1,0,0],[0,0,1,0],[0,0,1,1]]} $ \\ 
& $ V^- = \hbox{\tt [[1,0,0,-1],[1,1,1,0],[0,-1,1,0],[1,0,1,1]]} $  \\ \hline

6.90146  
& $ V^+ = \hbox{\tt [[1,1,0,0,0,0],[0,1,0,0,0,0],[1,1,-1,-1,1,1],[0,0,0,-1,1,0],}$ \\ 
&\qquad \quad $\hbox{\tt [0,0,0,0,0,0],[0,0,0,0,0,0]]} $ \\ 
& $V^- =  \hbox{\tt [[1,1,1,0,0,0],[0,1,1,0,0,-1],[0,0,-1,0,0,0],[0,0,-1,-1,0,0],}$ \\ 
&\qquad \quad $\hbox{\tt [0,0,1,1,0,-1],[0,1,1,0,1,0]]} $  \\ \hline

6.90147
& $V^+ =  \hbox{\tt [[1,1,1,1],[-1,1,0,0],[-1,1,1,0],[-1,1,0,1]]} $ \\ 
& $V^- =  \hbox{\tt [[1,0,0,0],[0,1,0,1],[0,1,1,0],[0,0,0,1]]} $  \\ \hline

6.90150  
& $V^+ =  \hbox{\tt [[1,1,0,1],[0,-1,0,0],[1,1,1,1],[0,1,-1,1]]} $ \\ 
& $V^- =  \hbox{\tt [[1,0,0,1],[1,-1,1,1],[1,0,1,0],[0,0,0,1]]} $  \\ \hline

$6.90167^r$
& $V^+ =  \hbox{\tt [[1,0,0,0],[0,1,0,0],[0,-1,1,0],[1,0,0,1]]} $ \\ 
& $V^- = \hbox{\tt [[1,-1,0,1],[1,1,-1,0],[0,0,1,0],[0,0,0,1]]} $  \\ \hline

{\bf 6.90172}
& $ V^+ = \hbox{\tt [[1,-1,0,0],[0,1,0,0],[1,-1,-1,0],[-1,1,1,-1]]} $ \\ 
& $ V^- = \hbox{\tt [[1,0,1,-1],[-1,1,-1,1],[0,0,-1,1],[0,0,0,-1]]} $ \\ \hline

$6.90185^r$
& $ V^+ = \hbox{\tt [[1,0,0,1],[-1,1,-1,0],[0,1,1,0],[-1,0,-1,1]]} $ \\ 
& $ V^- = \hbox{\tt [[1,-1,0,0],[0,1,0,0],[0,0,1,0],[0,0,-1,1]]} $  \\ \hline

6.90194 
& $V^+ =  \hbox{\tt [[-1,0,-1,-1],[-1,-1,0,-1],[0,0,1,1],[0,0,0,1]]} $ \\ 
& $ V^- = \hbox{\tt [[-1,0,0,0],[-1,-1,0,0],[-1,0,1,0],[-1,-1,1,1]]} $  \\ \hline

6.90195
& $ V^+ = \hbox{\tt [[1,1,0,0],[0,1,0,0],[-1,0,1,0],[0,0,0,1]]} $ \\ 
& $ V^- =\hbox{\tt [[1,0,-1,0],[1,1,0,-1],[0,0,1,0],[0,1,0,1]]} $  \\ \hline

{\bf 6.90209}
& $ V^+ = \hbox{\tt [[1,0,-1,0],[0,1,0,0],[0,-1,1,0],[-1,0,1,-1]]} $ \\ 
& $ V^- = \hbox{\tt [[1,0,0,-1],[0,1,-1,0],[-1,0,1,1],[0,0,0,-1]]} $ \\ \hline
 
6.90214
& $ V^+ = \hbox{\tt [[2,0],[0,1]]} $  \\ 
& $ V^- = \hbox{\tt [[2,1],[-1,1]]} $ \\ \hline

6.90217
& $ V^+ = \hbox{\tt [[0,0,0,0],[1,-1,0,0],[0,0,1,0],[0,0,0,1]]} $ \\ 
& $ V^- = \hbox{\tt [[0,1,1,0],[0,-1,0,0],[-1,0,1,-1],[0,0,1,1]]} $ \\ \hline
 
6.90219
& $ V^+ = \hbox{\tt [[1,0,0,0],[0,1,-1,0],[0,1,0,1],[0,0,0,-1]]} $ \\ 
& $ V^- = \hbox{\tt [[1,1,0,0],[-1,1,0,0],[0,0,0,0],[0,0,1,-1]]} $  \\ \hline
 
{\bf 6.90227}
& $V^+ =  \hbox{\tt [[-1,0],[1,2]]} $ \\ 
& $ V^- = \hbox{\tt [[-1,1],[0,2]]} $ \\ \hline

6.90228
& $ V^+ = \hbox{\tt [[2,1],[1,2]]} $  \\ 
& $ V^- = \hbox{\tt [[2,2],[0,2]]} $  \\ \hline
 
$6.90232^r$
& $V^+ =  \hbox{\tt [[0,0,0,0],[1,1,0,0],[0,0,-1,0],[1,0,1,1]]} $ \\ 
& $V^- =   \hbox{\tt [[0,1,1,1],[0,1,0,0],[-1,0,-1,1],[0,0,0,1]]} $ \\ \hline
 
6.90235
& $ V^+ = \hbox{\tt [[-1,0,0,0,0,-1],[-1,-1,-1,0,0,0],[0,1,0,-1,0,0],[0,0,0,1,-1,0],}$\\
&\qquad \quad $\hbox{\tt  [0,0,0,0,1,0],[1,0,0,0,1,0]]} $ \\ 
& $ V^- = \hbox{\tt [[-1,0,0,0,0,0],[-1,-1,0,0,0,0],[0,0,0,0,0,0],[0,0,-1,1,0,0],}$\\ 
&\qquad \quad $\hbox{\tt[0,0,0,-1,1,1],[0,0,0,0,0,0]]} $ \\ \hline
\end{tabular}

\bigskip
\caption{Seifert matrices of almost classical knots. Boldface is used for classical knots.}
\label{table-3}
\end{table}

\clearpage

\bibliographystyle{alpha}

\end{document}